\documentclass[12pt]{report}
\usepackage[margin=1in]{geometry}% Change the margins here if you wish.

\usepackage{amsmath}

\usepackage{amsmath,amsthm,amssymb}

\usepackage{graphicx}

\usepackage{color}
\usepackage{enumerate}
\usepackage{multicol}

\usepackage{hyperref}

\usepackage[nameinlink]{cleveref}
\usepackage[all]{xy}
\usepackage{ marvosym }
\usepackage{comment}

\usepackage{breqn}

\usepackage{halloweenmath}
\usepackage{tikz-cd} 

\newtheorem{thm}{Theorem}[section]
\newtheorem*{thm*}{Theorem}

\newtheorem{prop}[thm]{Proposition}
\newtheorem*{prop*}{Proposition}

\newtheorem{lemma}[thm]{Lemma}
\newtheorem*{lemma*}{Lemma}

\newtheorem*{assm*}{Assumption}

\newtheorem{corollary}[thm]{Corollary}
\newtheorem*{corollary*}{Corollary}

\newtheorem{claim}[thm]{Claim}
\newtheorem*{claim*}{Claim}

\newtheorem{defi}[thm]{Definition}
\newtheorem*{defi*}{Definition}

\newtheorem{fact}[thm]{Fact}
\newtheorem*{fact*}{Fact}

\newtheorem*{example*}{Example}

\newtheorem*{examples*}{Examples}

\newtheorem{remark}[thm]{Remark}
\newtheorem*{remark*}{Remark}

\newtheorem{conjecture}[thm]{Conjecture}
\newtheorem*{conjecture*}{Conjecture}

\newtheorem*{goal*}{Goal}

\newtheorem*{subgoal*}{Subgoal}

\newtheorem*{question*}{Question}

\newtheorem*{problem*}{Problem}

\newenvironment{claimproof}[1]{\par\noindent\emph{Proof:}\space#1}{\hfill $\square_{Claim}$ \bigskip}

\DeclareMathOperator{\acl}{acl}
\DeclareMathOperator{\dcl}{dcl}

\DeclareMathOperator{\cl}{cl}

\DeclareMathOperator{\stab}{stab}

\DeclareMathOperator{\RM}{RM}

\DeclareMathOperator{\DM}{DM}
\DeclareMathOperator{\tp}{tp}

\newcommand{\menosmult}{\cdot^{-1}}

\newcommand{\+}{\oplus}
\newcommand{\om}{\ominus}

\makeatletter
\newcommand{\tangent}{\mathrel{\vphantom{\cap}\mathpalette\@tangent\relax}}
\newcommand{\@tangent}[2]{%
  \vbox{\offinterlineskip
    \sbox\z@{$#1\pitchfork$}%
    \ialign{##\cr
      \raisebox{-0.37\height}[0pt][0pt]{%
        \resizebox{\ht\z@}{\height}{$\m@th\mspace{-1.5mu}-\mspace{-1.5mu}$}%
      }\cr
      \hidewidth\scalebox{1}[0.89]{$\m@th#1\cap$}\hidewidth\cr
    }%
  }%
}
\makeatother

\begin{document}

\title{Strongly minimal reducts of algebraically closed valued fields.} % You should make this the title of your project.
\author{\textbf{Santiago Pinzon} \\
Advisor: Alf Onshuus\\
Co-advisor: Assaf Hasson\\ \\ \\
Thesis submitted in fulfillment of requirements for the degree of \\
Doctor in Mathematics
} % Replace this with your name.
 %

 %I have this commented out so that the program will use whatever today's date is.  You can specify a particular date (or use this for other info if you need) in the {} or leave the {} empty for no date (or other extra info) to appear in the title

\maketitle % If you don't want a formal title, you can erase this and the bits above.  You will need to find a different way to include the same information.

\begin{center}
    \large\textbf{Abstract}
\end{center}

Let $\mathbb K=(K,+,\cdot,v,\Gamma)$ be an algebraically closed valued field and let $(G,\+)$ be a $\mathbb K$-definable group that is either the multiplicative group or contains a finite index subgroup that is $\mathbb K$-definably isomorphic to a $\mathbb K$-definable subgroup of $(K,+)$. Then if $\mathcal G=(G,\+,\ldots)$ is a strongly minimal non locally modular structure definable in $\mathbb K$ and expanding $(G,\oplus)$, it interprets an infinite field.

This document is the PhD thesis of the author and it was advised by professors Assaf Hasson and Alf Onshuus.
\tableofcontents

\chapter{Introduction and general facts}\label{introductionAndPreliminaries}

\section{Introduction}

Given any field $k$ there are many families of curves in $k^2$ that one can define. For example, one may define the family of all lines on $k^2$. Such lines together with the points of $k^2$ form a plane geometry in the sense of Chapter 2.1 of \cite{algGeoArt}.  Artin proved a converse: if one starts with a set $E$ and an incidence system of lines and points on $E$ satisfying some geometrical axioms, then one can build a field $k$ such that the lines of $E$ are given by linear equations on $k$. 

One can wonder whether this result holds in some other settings where one has a good notion of dimension and a big enough family of curves. For example, in \cite{rabinovich}, Rabinovich proved that if $k$ is an algebraically closed field and $\mathcal D=(D,\ldots)$ is some structure whose universe is $\mathbb A ^1 (k)$ and whose definable sets are constructible sets, then if $\mathcal D$ has a big enough definable family of curves contained in $D^2$, then $\mathcal D$ interprets an infinite field. In \cite{castleHasson} Castle and Hasson proved that the family of curves could be recovered as definable over the recovered field.

An example of such a family is the family of all lines contained in $\mathbb A^1 \times \mathbb A^1$. This is a two-dimensional family of curves, and in this case it is not hard to see that one can recover an infinite field. For example, if one considers the subfamily of lines passing through $(0,0)$ this can be identified with $\{l_m:m\in k\}\cup l_\infty$, where for $m\in k$ $l_m$ is the line of equation $y=mx$ and $l_\infty$ is the line of equation $x=0$. In this case, one has that $l_m\circ l_n=l_{m\cdot  n}$ so one can recover the multiplicative group of $k$, $\mathbb G_m$ from composition between elements of that subfamily. With some more work, we can also recover the additive group $\mathbb G_a$ and also the action of $\mathbb G_m$ on $\mathbb G_a$.

%In modern model theoretical language the analogous of this is Zilber's trichotomy. It was stated by Zilber in the 1970s and says that any strongly minimal set is either trivial, locally modular or it interprets a field  (precise definitions are given in Section \ref{preliminariesModelTheory}).

In the late 1970s, this principle was abstracted in the following conjecture:

\begin{conjecture}(Zilber's Trichotomy Principle)

If $\mathcal D$ is a strongly minimal structure and there is a big enough $\mathcal D$-definable family of plane curves, then $\mathcal D$ interprets an infinite field.
\end{conjecture}

In the conjecture, ``big enough'' will be read for us as ``Morley rank $2$''. The family of lines in the plane is an example.

Ravinovich's result is a particular positive case of this conjecture where $\mathcal D$ is some reduct of the full field structure on the affine line of an algebraically close field.

This conjecture was proved false by Hrushovski in \cite{hruNew}. He built examples of strongly minimal sets that are not trivial nor locally modular but do not interpret any infinite field.  However, the principle of Zilber's trichotomy still plays an important role in modern model theory.

In \cite{hrz} Hrushovski and Zilber defined (1-dimensional) Zariski geometries and proved that Zilber's trichotomy holds for them. This covers a vast class of examples, generalizing the algebraic case: if $k$ an algebraically closed field and $N$ is the set of $k$-points of a smooth algebraic curve (over $k$) then, if $\mathcal N$  is a structure with universe $N$ having as definable subsets of $N^k$ all the $k$-constructible sets, then $\mathcal N$ is a Zariski Geometry and it interprets an infinite field isomorphic to $k$. In fact there is a finite cover of $N$ by an algebraic curve over $k$.

%In \cite{chatzidakis1999model} Chatzidakis and Hrushovski used Zariski Geometries for proving that the conjecture is true for strongly minimal sets definable in a deferentially closed valued field. That was used for Hrushovki in his prove of Mordell-Lang conjecture in positive characteristic.

There are other settings where the conjecture has been proved true without the use of Zariski Geometries but using intersection theory. That is the case of \cite{HS}. Here Hasson and Sustretov proved that if $D$ has as universe an algebraic curve over an algebraically closed field $k$ and all the definable sets on $D^k$ are definable in the field structure, then Zilber's trichotomy is true for $D$. Note that this is a generalization of Ravinovich result.

In \cite{HE} the authors used intersection theory coming from open functions to prove that if $k$ is an o-minimal expansion of a real closed field and $(D,+)$ is a $k$ definable group of o-minimal dimension $2$, then if $\mathcal D=(D,+,\ldots)$ is $\mathbb K$-definable, strongly minimal and not locally modular structure, then $\mathcal D$ interprets an infinite field. 

There are also cases where the conjecture was proved to be true using intersection theory coming from analytic functions. It is proved in \cite{KR} that if $D$ has as universe an algebraically closed valued field $K$ with $\text{char} K=0$ and we assume that addition is definable and that all the definable sets of $D^k$ are definable in the valued field structure, then $D$ satisfies Zilber's trichotomy.

This is the setting in which we are interested. In the introduction of \cite{HS} it is suggested that their methods should be suitable to be used for proving generalizations of \cite{KR}. For example, we should be able to prove the result for positive characteristic or get rid of the assumption that $+$ is definable on $D$.

In this thesis we work on both possible generalizations:

First, we prove that the results of \cite{KR} are true even in positive characteristic. This is Theorem \ref{thmAditiveVersion}.

Moreover, we also prove that if $\mathcal G$ is a strongly minimal non locally modular expansion of an definable group  $G$, and we assume that $G$ is either the multiplicative group or it is locally isomorphic to $\mathbb G_a=(K,+)$, then $\mathcal G$ interprets an infinite field.

For us, an infinite group $G$ is locally isomorphic to $\mathbb G_a=(K,+)$ if there is $A\leq \mathbb G_a$ and $i:A\to G$ an injective group homomorphis, definable in the valued field structure such that $G/i(A)$ is finite.

%There is a problem arising when one wants to move to interepretable groups and it is the fact that there is no topology notion. However for the arguments we make it is enough if we ensure that the interpretable group is indeed locally isomorphic to a definable group contained in $K$. That is because we are only interested in construct group configuration and we prove the algebraic relations by proving that intersection numbers between definable curves drops but this is preserved under local isomorphism.  See proofs of Proposition \ref{propGroupInterpretableAdditive} and \ref{groupInterpretsField}.

We expect that our results can be used for proving Zilber's trichotomy for $\mathcal D=(D,\ldots)$, assuming that $D$ is any definable subset of an algebraically closed valued field $K$ whose Zariski closure is $1$-dimensional and whose definable sets are definable in the valued field structure. This could be done by finding a group interpretable in $\mathcal D$ as in Section \ref{findingAGroup}, proving that such a group is locally isomorphic either to $(K,+)$ or to $(K\setminus 0,\cdot)$. In the first case we could use the results of Chapter \ref{additiveCaseVersion} to conclude, in the second one, a generalization of the results of Chapter \ref{multiplicativeCaseVersion} would be needed.

%In this document we restrict ourselves to the case in which the valuation group is contained in the real numbers and the field is complete. This is because we want to use the results of \cite{LAG}. It does not seems to be a crucial assumption but we were not able to get rid of it. 
We will use results from \cite{LAG} which require $\mathbb K$ to be complete. In order to do so we will use completeness of the field only when we prove first order statements. Since $ACVF_{p,q}$ is complete, the results will hold in any model.

Now we present the structure of the document:

In Chapter \ref{introductionAndPreliminaries} we present the basic preliminaries on model theory and valued fields that we will need. 

In Chapter \ref{preliminaries} we present the basics that we will use in the subsequent chapters. Importantly, in Section \ref{computations} we present the main intersection theory needed in order to interpret groups and fields.

In Chapter \ref{additiveCaseVersion} we state and prove Theorem \ref{thmAditiveVersion}. This theorem asserts that Zilber's trichotomy is true in the case where the strongly minimal structure is a $\mathbb K$-definable expansion of a locally additive group. 

In Chapter \ref{multiplicativeCaseVersion} we prove Theorem \ref{thmMultiplicativeVersion}. This asserts that Zilber's trichotomy is true for $\mathbb K$-definable expansions of the multiplicative group. The proof of this is divided in two cases:
If there is a definable set with infinitely many derivatives, one may construct a field in the same way as in Chapter \ref{additiveCaseVersion}. If this is not the case, we build an interpretable group that is locally isomorphic to $(K,+)$, so we use the results of Chapter \ref{additiveCaseVersion} to conclude.

\section{Preliminaries on Model Theory}\label{preliminariesModelTheory}

Let $\mathcal L$ be a first order language and let $\mathcal P=(P,\ldots)$ be an infinite $\mathcal L$ structure. 
Here we adopt the usual definitions of $\mathcal P$-definable and $\mathcal P$-interpretable sets. See for example, Section 1.3 of \cite{marker}. In particular, whenever we say $\mathcal P$-definable (interpretable) we mean with parameters in $P$.  %We usually say 'definable' and 'interpretable' instead of '$\mathcal P$-definable' (intepretable) when $\mathcal P$ is clear from the context.

Assume moreover that $\mathcal P$ is $\kappa$-saturated, for some large enough cardinal $\kappa$. In particular $\kappa>\omega$.

\begin{defi}
Let $D$ be a $\mathcal P$-interpretable set. 

We say that $D$ is strongly minimal if $D$ is infinite and the only $\mathcal P$-interpretable subsets of $D$ are the finite and the cofinite sets. 

We say that $\mathcal P$ is strongly minimal if $P$ is strongly minimal as a $\mathcal P$-definable set.

If $T$ is a theory, we say that $T$ is strongly minimal if $\mathcal P$ is strongly minimal for all $\mathcal P\models T$.
\end{defi}
We will use the following notion of rank:

\begin{defi}
If $X\subseteq P^n$ is an interpretable set we say that $\RM_{\mathcal P}(X)\geq 0$ if $X$ is non empty and for an ordinal $\alpha$, $\RM_{\mathcal P}(X)\geq \alpha$ if there are $X_1,X_2\ldots$ infinitely many interpretable subsets of $X$ such that:
\begin{itemize}
    \item $\bigcup X_i=X$, 
    \item for all $i\neq j$, $X_i\cap X_j=\emptyset$  and
    \item for all $i$ and for all $\beta <\alpha$,  $\RM_{\mathcal P}(X_i)\geq \beta$.
\end{itemize}

We say that $\RM_{\mathcal P}(X)=\alpha$ if $\RM_{\mathcal P}(X)\geq \alpha$ and is not the case that $\RM_{\mathcal P}(X)\geq \alpha +1$. 

If there is some $\alpha<\kappa$ such that $\RM_{\mathcal P}(X)=\alpha$, we say that $X$ has bounded Morley rank and that $\alpha$ is the morley rank of $X$.

If $\RM_{\mathcal P}(X)=n\in \omega$ then we say that $X$ has finite Morley rank.

\end{defi}

%The following is well known, a proof can be find in Lemma 6.2.7 of \cite{marker}.

%\begin{fact}
%If $X,Y$ are subsets of $P^n$, then 
%\end{fact}

We omit the subscript if $\mathcal P$ is clear from the context.

\begin{defi}
If $A\subseteq P$ and $p(x)\in S_n(A)$ we define 
$$\RM(p)=\min\{\RM(X):X \in p\}.$$

For $a\in P^n$ let  $\RM(a/A)=\RM(tp(a/A))$.
\end{defi}

The following is well known and can be found on Lemma 6.2.7 of \cite{marker}.

\begin{fact}\label{factGoodDimension}
If  $X,Y$ are  definable subsets of $P^n$, then:

\begin{enumerate}
    \item If $X\subseteq Y$ then $\RM(X)\leq\RM(Y)$.
    \item $\RM(X\cup Y)=\max(\RM(X),\RM(Y))$.
    \item If $X\neq \emptyset$ then $\RM(X)=0$ if and only if $X$ is finite.
\end{enumerate}
\end{fact}

\begin{defi}
If $\RM(X)=\alpha$ then there is no partition of $X$ into infinitely many definable subsets of Morley rank greater than $\alpha$. Therefore, by compactness, there is a maximal natural number $n$ such that there are $X_1,\ldots,X_n$ definable and disjoints subsets of $X$ covering $X$ with $\RM(X_i)=\alpha$.
We define the Morley degree of $X$ as $\DM(X)=n$. 

We say that $X$ is stationary if $\DM(X)=1$.
\end{defi}

\begin{lemma}
An interpretable set $D$ is strongly minimal if and only if $\RM(D)=\DM(D)=1.$
\end{lemma}

\begin{proof}
Suppose $D$ is strongly minimal, then as $D$ is infinite and $D=\cup_{d\in D}\{d\}$ one has that $\RM (D)\geq 1$. As each infinite definable subset of $D$ is cofinite it is not the case that there are two infinite and disjoints definable subsets of $D$. It shows that $\RM(D)=1$ and also that $\DM(D)=1$.

Now assume that $\RM (D)=\DM (D)=1$, and suppose by contradiction that there is an infinite definable set $Y\subseteq D$ such that $D\setminus Y$ is also infinite. Then as $D=Y\cup (D\setminus Y)$ one has that $\DM(X)\geq 2$, a contradiction.
\end{proof}

The following can be found on page 196 of \cite{marker2017strongly}.
\begin{fact}\label{factDimN}
If $\mathcal P=(P,\ldots)$ is strongly minimal, then $\RM(P^n)=n$.
\end{fact}

\begin{corollary}
If $\mathcal P=(P,\ldots)$  is strongly minimal and $X\subseteq P^n$ is $\mathcal{P}$-definable, then $X$ has finite Morley Rank.
\end{corollary}

\begin{proof}
By Fact \ref{factDimN} one has that $\RM (P^n)=n$ and by clause 1 of Fact \ref{factGoodDimension} one conclude that $\RM (X)\leq n$, in particular, $X$ has finite Morley Rank.
\end{proof}

\begin{defi}
Suppose $\mathcal P$ is strongly minimal and let $X$ be a $\mathcal P$-interpretable set over $A$. We say that a tuple $z=(z_1,\ldots,z_n)\in X^n$ is $\mathcal P$-generic independent (over $A$) if $\RM(z/A)=n\RM(X)$. 

Whenever we say that some first order property holds for all generic elements we mean that it holds for any generic in any elementary extension of $\mathcal P$.
\end{defi}

We now define curves and families of curves. 

\begin{defi}
If $\mathcal P$ is an strongly minimal, a plane curve $C$ (of $\mathcal P$) is a $\mathcal P$-definable subset of $P^2$ with $\RM(C)=1$. 

A definable family of plane curves is a $\mathcal P$-definable set $X\subseteq P^{2+n}$ such that for all $a$ in some $\mathcal P$-definable set, $Q\subseteq P^n$ one has that $$X_a:=\{(x,y)\in P^2:(x,y,a)\in X\}$$ is a plane curve.

We usually write such a family as $(X_a)_{a\in Q}$.
\end{defi}

\begin{defi}
We say that a family of curves $(X_a)_{a\in Q}$ is \emph{almost faithful} if for all except finitely many $b\in Q$ the set $\{a\in Q: |X_a\cap X_b|=\infty \}$ is finite. 
\end{defi}

The idea of this definition is that an almost faithful family of curves has as many different curves as parameters, that is why from now when we say ``definable family of curves'' we mean ``definable and almost faithful family of curves''.

Now we define the notion of locally modular strongly minimal structures.

\begin{defi}
If $\mathcal P$ is strongly minimal we say that $\mathcal P$ is locally modular if any $\mathcal P$-definable and almost faithful family of plane curves $(X_q)_{q\in Q}$ satisfies that $\RM(Q)\leq 1$.
\end{defi}

Zilbers Trichotomy states:

Let $\mathcal P$ be strongly minimal, and assume that it is not locally modular, then there is an infinite field $k$ interpretable in $\mathcal P$.

In its full generality, it was proved false by Hrushovski on \cite{hruNew}. But it can be restated relative to a restricted setting:

Let $T$ be some complete theory 

\begin{conjecture}\label{zilbTric}(Zilber's trichotomy principle relative to $T$)

Let $\mathcal N=(N,\ldots)$ be any model of $T$ and let $P$ be some $\mathcal N$-interpretable set. Let $\mathcal P=(P,\ldots)$ be some  $\mathcal L'$-structure over $P$ (for some language $\mathcal L'$) such that any $\mathcal P$-definable set is also $\mathcal N$-definable. Then if $\mathcal P$ strongly minimal and non locally modular, there is an infinite field interpretable in $\mathcal P$.
 
\end{conjecture}
 
There are several instances of this conjecture that have been proved true, more relevant for us are:

\begin{fact} (Theorem 4.3.3 of \cite{HS})
If $T$ is the theory of algebraically closed fields of a fixed characteristic then Conjecture \ref{zilbTric} is true if we assume that $P$ is an algebraic curve. 
\end{fact}

Recently, Castle proved in \cite{CAS} the following:

\begin{fact}
 Conjecture \ref{zilbTric} is true relative to the theory $ACF_0$.
\end{fact}

The following is also true:

\begin{fact} (Theorem  3.17  on \cite{KR})
If $T$ is the theory of algebraically closed valued fields of characteristic zero (as defined in Section \ref{preliminariesOnValuedFields}) then Conjecture \ref{zilbTric} is true if we assume that $\mathcal N=(N,+,\cdot,v,\Gamma)$, $ P=N$ and addition is definable in $\mathcal P$.
\end{fact}

We will use techniques of both, \cite{HS} and \cite{KR} for almost all of our work.

\section{Group Configurations}

In this section we introduce the group configuration which is the main tool we will use to interpret groups and fields in strongly minimal structures.

From now on, in this section, we fix $\mathcal N=(N,\ldots)$ a strongly minimal structure.
\begin{defi}
A rank $d$ group configuration for $\mathcal N$ over a set of parameters $A$ is a $6$-tuple,  $\mathfrak g=(a,b,c,x,y,z)$ of tuples of elements of $N$ such that:

\begin{itemize}
    \item $\RM (\mathfrak g/A)= 2d+1$
    \item $\RM(\alpha,\beta/A)=\RM(\alpha/A)+\RM(\beta/A)$ for all $\alpha\neq \beta\in \mathfrak g$,
    \item  $\RM(a/A)=\RM(b/A)=\RM(c/A)=d$, 
    \item  $\RM(x/A)=\RM(y/A)=\RM(z/A)=1$, 
    \item  $\RM(a,b,c/A)=2d$, 
    \item $\RM(a, x, y/A)=\RM(b,z,y/A)=\RM(c,x,z/A)=d+1$.
\end{itemize}
\end{defi}

We often write such a configuration in the form of a diagram

  
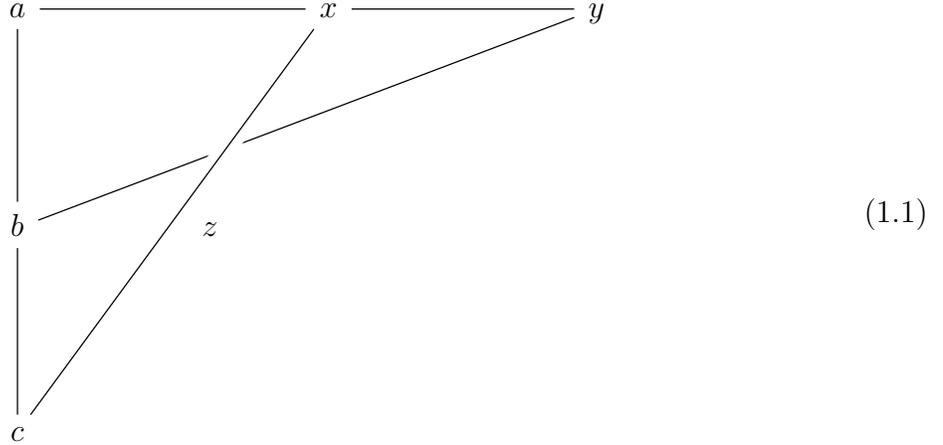
\begin{equation}\label{GCDiagram}
 	\begin{tikzcd}
 	a \arrow[ddd, dash] \arrow[rrr, dash] &&& x \arrow[rrr, dash]  &&& y \arrow[dddllllll,  dash] 
  \\
&&&&&&
\\
& & && %\arrow[dll, dash]  
%\arrow[ddddll, dash]  
&&
\\
 b  \arrow[ddd, dash] && z  &&&& \\ &&&&&&
 \\
 &&&&&&
 \\
 c\arrow[uuuuuurrr, crossing over, dash] &&&&&&
  	\end{tikzcd}
 	\end{equation}
 in which the collinear vertices are dependent and each set of tree non-collinear vertices are independent.

\begin{defi}
If $G$ is an interpretable Abelian group with $\RM (G)=1$, a \emph{standard group configuration of $G$} is 
$$\mathfrak g_G=(a,b, ab, c, ac, abc)$$ for some choice of $a,b,c$ $\mathcal N$-generic independent elements of $G$.
\end{defi}

\begin{defi}
We say that $\mathfrak g=(a,b,c,x,y,z)$ is a reduced group configuration for $\mathcal N$ if it is a group configuration and in addition if for $\alpha\in\mathfrak g$, $\alpha'\in \acl(a)$ are such that $\mathfrak g'=(a',b',c',x',y',z')$ is still a group configuration, then $\alpha\in \acl(\alpha')$ for each $\alpha\in\mathfrak g$.  

If $\mathfrak g_1$ is another group configuration, we say that $\mathfrak g$ and $\mathfrak g_1$ are interalgebraic if the corresponding coordinates satisfies $\acl(\alpha)=\acl(\alpha')$.
\end{defi}

The following is due to Hrushovski (\cite{hrushovski1986contributions}). The precise statement we need can be found in Facts 4.4 and  4.6 of \cite{HS}.

\begin{fact}\label{groupConfig}

If $\mathfrak g$ is a rank $d$ group configuration for $\mathcal N$ (over some set of parameters), then there is a minimal group $G$ with $\RM(G)=d$, a strongly minimal set $X$ and a faithful action of $G$ on $X$ all the data interpretable in $N$.

In addition, if $\mathfrak g$ is reduced, then there is a group configuration of $G$ that is interalgebraic with $\mathfrak g$.

\end{fact}

The following is also due to Hrushovski. Can be found in  \cite{bouscaren1989group} Main Theorem, Proposition 2.

\begin{fact}\label{fieldConfig}
Let $G$ be an $\mathcal N$-interpretable group acting transitively and faithfully on a strongly minimal set $X$. Assume $\RM(G)=2$ then there is an $\mathcal N$-interpretable field such that $G$ is isomorphic to the semi-direct product of the multiplicative and the additive group of such a field.\\

In particular, if $\mathfrak g$ is a rank 2 reduced group configuration, then $\mathcal N$ interprets an infinite field.
\end{fact}

So it makes sense to define:

\begin{defi}
A field configuration is a reduced rank $2$ group configuration.
\end{defi}

\section{Valued Fields}\label{preliminariesOnValuedFields}

We start with some basic definitions.

\begin{defi}
Given a field $\mathbb K=(K,+,\cdot,0,1)$ we treat it as a first order structure in the language of rings $\mathcal L_R=\{+,\cdot,0,1\}$. Let $ACF_p$ be the first order theory of algebraically closed fields of characteristic $p$.
\end{defi}

\begin{fact} (It follows for example from Theorem 3.2.2 of \cite{marker})
The theory $ACF_{p}$ is complete and strongly minimal.
\end{fact}

\begin{defi}
Given a field $(K,+,\cdot,0,1)$ we say that a set $X\subseteq K^n$ is Zariski closed if there is a set of polynomials $\mathfrak a \subseteq K[x_1,\ldots,x_n]$ such that for all $x\in K^n$, $x\in X$ if and only if $f(x)=0$ for all $f\in \mathfrak a$. 

If $X$ is Zariski closed, we say that $X$ is irreducible if there are no proper Zariski closed subsets $X_1$ and $X_2$ of $X$ such that $X=X_1\cup X_2$.
\end{defi}

Now we make some definitions:

\begin{defi}
    If $F(x,y)$ is an irreducible polynomial and $C\subseteq K\times K$ is the set of zeros of $F$, we say that a point $(a,b)\in C$ is a \emph{singular point of $C$} if:
    $$\frac{\partial F}{\partial x}(a,b)=\frac{\partial F}{\partial y}(a,b)=0.$$ We say that $(a,b)$ is \emph{regular} if it is not singular.
\end{defi}

The next fact is a consequence of Theorem 5.3 and the proof of Theorem 5.1 of \cite{hartshorne}.
\begin{fact}\label{finiteSingularPointsFact}
    If $F(x,y)$ is any polynomial, then its set of zeros has finitely many singular points.
\end{fact}

%\begin{fact}
%If $\mathbb K=(K,+,\cdot,0,1)$ is a model of $ACF_p$, then $\RM_{\mathbb M}(X)=\dim (\cl(X))$ where $\cl(X)$ is the Zariski closure of $X$. 
%\end{fact}

We will need:

\begin{fact}\label{bezoutTheorem}(Bezout Theorem, see for example Theorem 2.2.7. of \cite{annala2016bezout})

Let $k$ be an algebraically closed field and let $F(x,y)$ and $G(x,y)$ be polynomials with coefficients on $k$ with no common non-constant divisors, then if $V$ is the set of zeros for $F$ and $W$ is the set of zeros for $G$ then $V\cap W$ is finite, and it has at most $\deg(F)\deg(G)$ points. If we consider the closures of $V$ and $W$ in the projective space $\mathbb P^2$, then the number of points of intersection is exactly $\deg(F)\deg(G)$ (counting multiplicities).
\end{fact}

\begin{defi} 
If $K$ is a field, a valuation on $K$ is an ordered Abelian group $\Gamma$ together with a valuation map $v:K\to \Gamma\cup\{\infty\}$ (where $\infty$ is an extra element such that $\infty>\gamma$ and $\infty+\gamma=\gamma+\infty=\infty$ for each $\gamma\in \Gamma$) such that for all $x,y\in K$, $v$ satisfies:

\begin{enumerate}
    \item $v(x)=\infty$ if and only if $x=0$.
    \item $v(x\cdot y)=v(x)+v(y)$.
    \item $v(x+y)\geq \min\{v(x),v(y)\}$. 
\end{enumerate}
\end{defi}
If $K$ is valued, the valuation ring is $\mathcal O_K=\{x\in K:v(x)\geq 0\}$. It is a sub-ring of $K$ and its only maximal ideal is $\mathfrak m=\{x\in K:v(x)>0\}$. The residue field is $k=\mathcal O_K/\mathfrak m$.

We treat a valued field as a two sorted structure $(K,\Gamma,v)$ 
where $K$ is a field, $\Gamma$ is an Abelian ordered group and $v:K\to \Gamma\cup\{\infty\}$ is the valuation. Let $\mathcal L_{R,v}$ be the corresponding two sorted language.
\begin{defi}
Let $ACVF_{p,q}$ be the first order theory in the language $\mathcal L_{R,v}$ saying that $K$ is an algebraically closed field of characteristic $p$ and $v$ is a valuation into an ordered Abelian group $\Gamma$:

In ACVF we have the axioms stating that $(K,+,\cdot,0,1)$ is an algebraically closed field of fixed characteristic $p$, $\Gamma$ is an Abelian group plus the extra axioms

$(\Gamma,\leq)$ is a linear order:

\begin{itemize}
\item $\forall x,y\in \Gamma(x\leq y \vee y\leq x)$,
\item $\forall x,y,z\in \Gamma(x\leq y \wedge y\leq x\implies x\leq z)$, 
\item $\forall x,y\in \Gamma(x\leq y\wedge y\leq x\implies x=y)$ and 
\item $\forall x\in \Gamma (x\leq x)$.

\end{itemize}

Addition on $\Gamma$ is compatible with $\leq$:

$$\forall x,y,z\in \Gamma (x\leq y\implies x+z\leq y+z).$$

Moreover, the residual field $$\{x\in K:v(x)\geq 0\}/\{x\in K:v(x)> 0\}$$ has characteristic $q$.

\end{defi}

%In this document we  prove (some particular cases of) the following theorem:

%\begin{thm}\label{thmGrupo}

%If $(G,\cdot)$ is a $\mathbb{K}$-definable group and $X\subseteq G^2$ is a $\mathbb K$-definable set that is not a boolean combination of cosets of subgroups of $G$, then the structure $(G,\cdot,X)$  interprets an infinite field.
%\end{thm}

\begin{defi}
An open ball is a subset of $K$ of the form $$B_{\gamma}(a):=\{x\in K:v(x-a)>\gamma\}$$ where $a\in K$ and $\gamma\in \Gamma$.
A closed ball is a subset of $K$ of the form
$$B_{\geq\gamma}(a)=\{x\in K:v(x-a)\geq\gamma\}.$$

We say that $\gamma$ is the (valuative) radius of $B_{\gamma}(a)$

\end{defi}

In this setting we have two natural topologies on $K$:  the Zariski and the valuation topologies. The latter is generated by balls. When we say `open' we mean in the valuation topology. When we want to refer to the Zariski topology we will be explicit about it.

\begin{defi}\label{defiDimension}
If $D\subseteq K^n$ is a $\mathbb K$-definable set we define $\dim (D)$ as the usual algebraic dimension of the Zariski closure of $D$. Moreover, if $p(x)$ is an $n$-type over $A$, we define $\dim p=\min\{\dim X:X\in p\}$ and if $a\in K^n$ then $\dim (a/A)$ is defined as $\dim (tp(a/A))$.
\end{defi}

\begin{defi}\label{genericDefinition}
If $X\subseteq K^n$ is $\mathbb K$-definable over $A$, we say that $x=(x_1\ldots,x_m)\in X^m$ is a tuple of $\mathbb K$-independent generic elements of $X$ if $\dim (x/A)=m\dim (X)$. 

Whenever we say that a first order property $P$ holds for all the $\mathbb K$-generic elements of $X$ we mean that if $\bar {\mathbb K}$ is any elementary extension of $\mathbb K$, then $P$ holds for all the $\bar{\mathbb K}$- generic elements of the interpretation of $X$ in $\bar{\mathbb K}$.
\end{defi}

We have the next classic theorem that follows from Holly's work in \cite{HO}. As stated here is Theorem 7.1 of \cite{hru-haskell2005stable}.

\begin{fact}\label{QEff} (Quantifier Elimination)
$ACVF_{p,q}$ has quantifier elimination in the language $\mathcal L_{R,v}$. 
\end{fact}

As a consequence we have:

\begin{prop}\label{decompositionProp}
Suppose $Y\subseteq  K^m$ is $\mathbb{K}$-definable and infinite. Then $Y$ can be written as a finite union of  subsets of $K^m$ that are relatively open subsets of irreducible Zariski closed sets.

If $Y\subseteq K\times K$ does not have isolated points one can write $Y$ as a finite union $$Y=\bigcup\{(x,y)\in V_i:L_i(x,y)=0\}$$ where $L_i(x,y)$ is an irreducible polynomial, $V_i\subseteq K\times K$ is open and $\{(x,y)\in V_i:L_i(x,y)=0\}\neq \emptyset$. %In particular, if $Y$ has no isolated points, then each point of $Y$ is in the relative interior of $Y$.
\end{prop}

   \begin{proof}

   This is very similar the the proof of Lemma 3.4 of \cite{KR}. We include a proof since there is a small gap in their proof.
   
Using Fact \ref{QEff}  write $Y$ as a finite union of sets of the form:

\begin{equation}\label{QEDecompositionEcuation}
\bigcap_i \{x\in K: v(f_i(x)) \Box_i v(g_i(x))\}
\end{equation}
where $\Box_i\in \{<, =\}$ and each $f_i$ and $g_i$ are polynomials.
For each $i$ let $$A_i:=\{x\in K:v(f_i(x))\Box_i v(g_i(x))\}.$$

If $\Box_i$ is $<$ and $g_i$ is not constant $0$ then 
$$A_i=\{x\in K: g_i(x)\neq 0\text{ and }v(f_i(x)/g_i(x))< 0\}\cup \{x\in K: f_i(x)\neq 0 =g_i(x)\}$$ which is the  intersection of an open (as the polynomials are continuous in the valuation topology) with a Zariski closed. If $\Box_i$ is $<$ and $g_i$ is constant $0$ then 
$$A_i=\{x\in K:v(f_i(x)) < \infty\}=\{x\in K:f_i(x)\neq 0\}$$ which is open.

If $\Box_i$ is $=$ and both, $f_i$ and $g_i$ are not the zero polynomial, then
$$A_i=\{x\in K:g_i(x)\neq 0\text{ and }v(f_i(x)/g_i(x))=0\}$$ which is open.

Finally, if $\Box_i$ is $=$ and either $f_i$ or $g_i$ is the zero polynomial, then $A_i$ is Zariski closed. 
Therefore, the intersection of Equation \ref{QEDecompositionEcuation} is a Zariski closed $C$ set intersected with an open set $U$, take the irreducible decomposition of $C$:
$$C=C_1\cup\ldots\cup C_d$$ so $$C\cap U=(C_1\cap U\cap V_1)\cup \ldots\cup (C_d\cap U\cap V_d)$$ where $$V_i=K\times K\setminus \left(C_1\cup C_2\cup\ldots\cup C_{i-1}\cup C_i\cup\ldots\cup C_d\right).$$ 
Then each $C_i\cap U\cap V_i$ is an open subset of $C_i$ that is its Zariski closure so the intersection of Equation \ref{QEDecompositionEcuation} has the desired form and then $Y$ has also the desired form.\end{proof}

% For the second part, if $Y\subseteq K\times K$, one can write $Y$ as a finite union  
 %   $$Y=Y_0\cup Y_1\cup\ldots\cup Y_r$$ where $Y_0$ is finite and each $Y_i$ is a relatively open subset of its Zariski closure. So we may assume that $\dim Y_i=1$ for $i>0$. Then if $Y$ does not have isolated points $Y=Y_1\cup\ldots\cup Y_r$ where each $Y_i$ is a relatively open subset of its closure that is an irreducible curve. Assume that the Zariski closure of $Y_i$ is the set of zeros of $L_i(x,y)$ so there is an open set $V_i\subseteq K\times K$ such that 
  %  $$Y_i=\{(x,y)\in V_i:L_i(x,y)=0\}.$$ 
   % Then
    %$$Y=\bigcup_i\{(x,y)\in V_i:L_i(x,y)=0\}$$ 
    %as desired.  
 
As a corollary we have:

\begin{corollary}\label{existsInf}
    If $W\subseteq K$ is $\mathbb K$ definable then it is infinite if and only if there is an open ball contained in $W$.  
\end{corollary}

\begin{proof}
    It follows from Proposition \ref{decompositionProp} since the only Zariski closed subsets of $K$ are the finite sets and $K$ itself.
\end{proof}

   We also have:

\begin{fact}\label{aclIgualFact}

If $\bar a\in \acl_{\mathbb K} (B)$ for some set of parameters $B$ then there is a finite set $B$-definable in the field structure of $\mathbb K$ containing $\bar a$

In other words, $\acl_{\mathbb K}=\acl_{\mathbb K^{f}}$ where $\mathbb K^{f}$ is $\mathbb K$ seen as a structure in the language of rings.
    \end{fact}

From now on we fix $\mathbb K=(K,+,\cdot,\Gamma,v)$ a model of ACVF.

If $\Gamma$ is contained in the real numbers we can provide to $K$ with a metric:

For $x,y\in K$ let $d(x,y)=q^{-v(x-y)}$ where $q$ is any real number bigger than one. In this case we can define:

\begin{defi}
We say that $\mathbb K$ is complete if the valuation group of $\mathbb K$ is contained in the real numbers and $K$ is complete as a metric space with the distance $d(x,y)=q^{-v(x-y)}$ for some fixed real number $q>1$. That is, any Cauchy sequence converges.
\end{defi}

This definition does not depend on the choice of $q>1$.

\begin{quote}
\textbf{For the rest of this section we assume that $K$ is complete.}
\end{quote}

We will write a power series on $n$ variables as 
$$\sum_I a_I x^I$$ where $I=(i_1,\ldots,i_n)$ varies in the $n$-tuples of natural numbers, $a_I\in \mathbb K$, $x=(x_1,\ldots,x_n)$ is a $n$-tuple of variables and $x^{(i_1,\ldots,i_n)}$ is the monomial $x_1^{i_1}x_2^{i_2}\ldots x_n^{i_n}$.

Now we fix some concepts:
\begin{defi}
Let $U \subseteq K^n$ be open and $f:U\to K$ a function,  we say that $f$ is analytic in $U$ if there is $a\in U$ and 
$$g=\sum a_I x^I$$ a power series converging in 
$$U-a:=\{x-a:x\in U\},$$
such that $f(z)=g(z-a)$ for each $z$ in $U$. 

\end{defi}

\begin{defi}
If $$F(x_1\ldots,x_n)=\sum_{I=(i_1,\ldots,i_n)}a_I x^{(i_1,\ldots,i_n)},$$  is a power series then $\frac{\partial F}{\partial x_k}(x_1\ldots,x_n)$ is the (formal) partial derivative of $F$ with respect to $x_k$, that is $$\frac{\partial F}{\partial x_k}(x_1\ldots,x_n)=\sum_{I=(i_1,\ldots,i_n)} i_k a_I  x^{(i_1,\ldots,i_k-1,\ldots,i_n)}.$$
\end{defi}

\begin{defi}
Let $U\subseteq K$ be an open set and let $a\in U$. Suppose $f:U\to K$ analytic at $a$ and its power expansion around $a$ is $$f(x)=\sum_{n\geq 0} b_n (x-a)^n.$$ We say that $a$ is a zero of $f$ with multiplicity $d$ if  $d=\min\{n:b_n\neq 0\}$.
\end{defi}

\begin{lemma}\label{derivativeZeroIsDoubleLemma}
    Let $U\subseteq K$ be open and let $a\in U$. Assume that $$f(x)=\sum_{n\geq 0} b_n (x-a)^n$$ is the power expansion of $f$ around $a$. Then if $f(a)=f'(a)=0$, $a$ is a zero of multiplicity at least $2$ for $f$.
\end{lemma}

\begin{proof}
    If $f(a)=0$ then $b_0=0$. As $$f'(x)=\sum_{n\geq 1} n b_n (x-a)^{n-1},$$ $f'(a)=b_1$. The conclusion follows from the definition. 
\end{proof}

%\begin{lemma}
%Let $f:U\to V$ be a $\mathbb K$ definable function with $U$ and $V$ open subsets of $K^n$ and $K^m$ respectively, then there is some $a\in U$ such that $f$ is analytic at $a$.
%\end{lemma}

%\begin{proof}

%\end{proof}

\begin{comment}

\end{comment}

We will also need an implicit function theorem, the following is well known and can be found for example in \cite{LAG} (Theorem (10.8), page 84)

\begin{fact}(Implicit Function Theorem)\label{implicitFunctionTheoremf}
Suppose $K$ is complete, $U\subseteq K^n$ is open and $F:U\to K$ is an analytic function at $z\in U$. Assume $z$ is such that $\frac{\partial F}{\partial x_n}(z)\neq 0$, then there are $U_1\subseteq K^{n-1}$ and $U_2$ open subsets of $K$, $U'\subseteq U$ open with $z\in U' \cap (U_1\times U_2)$ and $f:U_1\to U_2$ analytic such that:

\begin{equation*}
    \{u\in U':F(u)=0\}=\{(x,f(x)):x\in U_1\}.
\end{equation*}
\end{fact}

As a corollary we have the Inverse Function Theorem:

\begin{fact}\label{inversionf}(Inverse Function Theorem, Theorem 10.10 in \cite{LAG})
Let $F(x)$ be analytic at $0$, assume $F(0)=0$ and $F_x (0)\neq 0$ then there is an unique analytic function $G$ converging in a neighborhood of $0$ such that $F(G(y))=y$ for each $y$ in such a neighborhood.
\end{fact}

Now we state some strong results on valued fields, first a theorem about continuity of roots:

\begin{fact}\label{continuityOfRootsTheoremf}(Continuity of Roots)

Assume $K$ is complete. Suppose $U_1\subseteq K$ and $U_2\subseteq K^m$ are open sets. Let 
$$F:U_1\times U_2\to K$$ be an analytic function at $(a,b)\in U_1\times U_2$. Assume that the function $x\mapsto F(x,b)$ has a zero of multiplicity $d>0$ at $x=a$. Then there are open sets $U_a$ and $U_b$ with 
$a\in U_a\subseteq U_1$ and
$b\in U_b\subseteq U_2$ such that:

\begin{enumerate}
    \item $a$ is the unique zero of $F(x,b)$ in $U_a$.
    \item For each $y\in U_b$ the function $x\mapsto F(x,y)$ has exactly $d$ zeros in $U_a$ (counting multiplicities)
\end{enumerate}

\end{fact}

\begin{proof}
We may assume $(a,b)=(0,0)$. By Theorem (10.3)(2) of \cite{LAG} there is some $F^*(x,y)=f_0(y)+f_1(y)x+\ldots+x^d$ such that each $f_j$ is analytic, $f_i(0)=0$ for $i=1,2,\ldots, d-1$, and there is some $\delta$ a unit of $K[[x,y]]$ such that $F=F^* \delta$. Because of Theorem (11.3) of \cite{LAG} there is an open set $U_b$ such that each $f_i$ is convergent in $U_b$, $0\in U_b$ and for all $y\in K$ if $y\in U_b$ and $F^*(x,y)=0$ then $x\in U_1$. Moreover, for all $y\in U_b$ the polynomial $F^*(*,y)$ is a polynomial of degree $d$ in the first variable so as $K$ is algebraically closed there are exactly $d$ roots (counting multiplicities) and all of them belong to $U_1$ 
\end{proof}

We will also need an identity theorem for expansion of power series:

\begin{fact}\label{identityTheoremf} (Identity Theorem (10.5.2) of \cite{LAG})

Let $$f(x_1,\ldots,x_n)=\sum_I a_I x_1^{i_1}\ldots x_x^{i_n}$$ be a power series converging in a neighborhood $D(f)$ of $(0,\ldots,0)$. Assume that $f(x)=0$ for all $x$ in some open set $U\subseteq D(f)$, then $a_I=0$ for all $I$.
\end{fact}

As a corollary we have:

\begin{lemma}\label{puedoRestringirderivada}
    Let $L(s)$ be an analytic function converging in some open set $V\ni 0$. Assume that 
    $L(V)$ is infinite and also that $L(0)=0$ then, there is a neighborhood $W$ of $0$ such that for all $s\in W\setminus\{0\}$, $L(s)\neq 0$.
\end{lemma}

\begin{proof}
    If it is not the case, then there is an open ball $B\subseteq V$ such that for all $s\in B$ $L(s)=0$ but as $L$ is a power series we can use Fact \ref{identityTheoremf} and conclude that $s\mapsto L(s)$ is the zero function on $V$ but this contradicts the assumption that $$\{L(s):s\in V\}$$ is infinite.
\end{proof}

\section{Binary Polynomials}\label{somePreparations}

%Note that as $G$ is infinite, it contains an open ball, so there is $B\subseteq G$ some open ball say around $s$. Then given any $r\in G$ as the map $x\maps to $ 

Here we give some basic definitions and reductions on plane curves that we use in Chapter \ref{additiveCaseVersion}. In this section we will not use the valuative structure on $\mathbb K$ so all the results follow for any algebraically closed field $K$. We set $p:=\text{char}(K)\geq 0$.

\begin{lemma}\label{pickCoordinateLemmaP}
 Let $F(x,y)\in K[x,y]$ be an irreducible polynomial and suppose that $F$ is not constant. If there are infinitely many points $(a,b)\in K^2$ such that 
 \begin{equation}
     F(a,b)= \frac{\partial F}{\partial y}(a,b)=0,
 \end{equation}
 then $F(x,y)=G(x,y^{p})$ for some polynomial $G$.
\end{lemma}

\begin{proof}

Note first that $\deg(\frac{\partial F}{\partial y})<\deg(F)$. So as $F$ is irreducible, if $\frac{\partial F}{\partial y}\neq 0$ we can use B\'ezout's Theorem (Fact \ref{bezoutTheorem}) to conclude that the number of common zeros of $F$ and $\frac{\partial F}{\partial y}$ is at most $\deg F\cdot \deg \frac{\partial F}{\partial y}$. So if $\frac{\partial F}{\partial y}$ has infinitely many common zeros with $F$ it is because $\frac{\partial F}{\partial y}=0$.

 If char$(K)=0$ and $\frac{\partial F}{\partial y}=0$ it follows that $F(x,y)=G(x)$ for some polynomial $G(x)$.
 
So assume char $(K)=p>0$. Let $V\subseteq K^2$ the set of zeros for $F$. 

%It is clear that $\deg (F_y)\leq \deg(F)$ and since $F$ is irreducible it follows that $F$ and  $F_y$ have no common factors. So if $F_y\neq 0$ we can use B\'ezout's Theorem to conclude that the number of common zeros of $F$ and $F_y$ equals $\deg F\cdot \deg F_y$. So i

Assume $\frac{\partial F}{\partial y}$ has infinitely many zeros at $V$ so $\frac{\partial F}{\partial y}=0$ and write $$F(x,y)=f_0(x)+f_1(x) y +\ldots+f_n(x) y^n$$ with $f_i(x)\in K[x]$ for $i=0,1,\ldots, n$.  As $$\frac{\partial F}{\partial y}(x,y)=f_1(x)+2f_2(x)y+\ldots+nf_n(x)y^{n-1}=0,$$ so for any $a$ and any $i=1,\ldots, n$ we have that $if_i(a)=0$. If some $f_i$ is different from zero we can take $a\in K$ such that $f_i(a)\neq 0$ and $if_i(a)=0$ implies that $i=0$ so $i=mp$ for some $m$. Therefore, if $f_i$ is different from zero then $p\mid i$ so $F(x,y)$ is a polynomial in the variable $y^p$.   
 \end{proof}
 
 As an immediate corollary we have:
 
 \begin{lemma}\label{pickCoordinateLemma}
 Assume $F(x,y)$ is an irreducible polynomial and let $V$ the set of zeros of $F$ then there is $E$, some finite subset of $V$, such that either for all $a\in V\setminus E$, $F_x(a)\neq 0$ or for all $a\in V\setminus E$, $F_y(a)\neq 0$.
 
 \end{lemma} 

\begin{proof}
Assume it is not the case, then, both $F_x$ and $F_y$ have infinitely many zeros in $V$. So using Lemma \ref{pickCoordinateLemmaP} one can conclude that $F(x,y)=G(x^p,y^p)$ for some polynomial $G$ but then $F(x,y)=(\tilde G(x,y))^p$ for some polynomial $\tilde G$ and then $F$ is not irreducible.
\end{proof}

\section{Power Series}

In this section we provide the basic facts on power series that we will use. Here we work in the setting of a complete algebraically closed valued field $\mathbb K$ and let $p=\text{char}(K)$

We start with a well known fact:

\begin{lemma}\label{convergenciaFacilLemma}
    The series 
    $$\sum_{n\geq 0} a_n$$ 
    converges if and only if 
    $$\lim_{n\to \infty} v(a_n)=\infty.$$
\end{lemma}

\begin{proof}
If the series converges then the sequence of partial sums $(s_n)_{n\geq 0}$ is Cauchy, in particular if $\gamma\in \Gamma$, there is some $k$ such that for $n\geq k$, $v(a_{n+1})=v(s_{n+1}-s_n)\geq \gamma$ so 
$$\lim_{n\to \infty} v(a_n)=\infty.$$

For the other direction let $\gamma\in \Gamma$, as 
$$\lim_{n\to \infty} v(a_n)=\infty$$ there is some $k$ such that for $n\geq k$, $v(a_n)\geq \gamma $ so if $m\geq n\geq k$ 
$$v(s_m-s_n)= v(a_{n+1}+\ldots+a_{m})\geq \min \{v(a_{n+1}),\ldots,v(a_m)\}\geq \gamma.$$
Thus, the sequence of partial sums is Cauchy and as $\mathbb K$ is complete, the series converges.
\end{proof}

\begin{prop}\label{coefficientsAreAnalyticProp}
   Let $U\ni 0$ be open and let $f:U\to K$ be an analytic function. For $a\in U$ and $m\in \mathbb N$ let $e_m(a)$ be the coefficient of  $(x-a)^m$ in the power expansion of $f$ at $a$, so
    $$f(x)=\sum_{m\geq 0} e_m(a) (x-a)^m,$$
    then $e_m(a)$ is a power series in the variable $a$ that converges at $a$.
\end{prop}

\begin{proof}
Let $g(x)=f(x+a)$, so $g(x)$ converges in some neighborhood of $0$. Indeed, if we take $U$ as an open ball, then $g(x)$ converges on $U$. Assume that the power expansion of $f(x)$ around $0$ is:

$$f(x)=\sum_{n\geq 0} b_n x^n$$ then 

$$f(x+a)=\sum_{n\geq 0}b_n (x+a)^n  = \sum_{n\geq 0} b_n \sum_{k=0}^{n} \binom{n}{k} x^{k} a^{n-k}$$

where if $m$ is the natural number $\binom{n}{k}$, the coefficient $\binom{n}{k}$ in the last expression is gotten by adding $1$ $m$ times on $K$. 

In this expansion, the coefficient of $x^m$ is 
\begin{equation}\label{definitionEm}
e_m(a)=\sum_{n\geq m} \binom{n}{k} b_n a^{m-n}=\sum_{k\geq 0} \binom{m+k}{m} b_{m+k} a^k.
\end{equation}

Thus, $f(x)=g(x-a)$ has power expansion at $a$ given by:

$$f(x)=\sum_{m\geq 0} e_m(a)(x-a)^m.$$

So we have to prove that the power series of Equation \ref{definitionEm} converges. By Lemma \ref{convergenciaFacilLemma} it is enough to prove that: 

$$\lim_{k\to \infty} v\left(\binom{m+k}{m} b_{m+k} a^k\right)=\infty.$$ 

For this let $\gamma \in \Gamma$. Note that

$$ v\left(\binom{m+k}{m} b_{m+k} a^k\right)\geq v\left( b_{m+k} a^k \right)$$

and as the series 
$$\sum_{k\geq 0} b_k a^k$$ converges, 
$$\lim_{k\to \infty} v\left(b_k a^k\right)=\infty$$ 
so for $k$ big enough we have that $v(b_{m+k} a^{k+m})\geq \gamma+mv(a)$, but 
$$v(b_{m+k} a^{m+k})=v(b_{m+k} a^{k})+mv(a)$$

Thus, 
$$v(b_{m+k} a^k)\geq \gamma$$
so
$$\lim_{k\to\infty} v(b_{m+k} a^k)=\infty$$ and we get the result. 

\end{proof}

Notice that it makes sense to multiply power series and the product converges whenever both series converges.

\begin{defi}\label{defN}
Let $$f(x)=\sum_{n\geq 1} b_n x^n$$ be a power series converging in a neighborhood $D(f)$ of $0$. For $a\in D(f)$ define $f_a(x)=f(x+a)-f(a)$. Suppose $$f_a(x)=\sum_{n\geq 1}b_{n,a} x^n$$ and define 
\begin{equation*}
\begin{split}
s_n(f) &:=\{b_{n,a}:a\in D(f)\}  \\ 
N_+(f) &:=\min\{n:s_n(f)\text{ is infinite}\}
\end{split}
\end{equation*}
\end{defi}

\begin{lemma}\label{N}
Let $f(x)$ be a function analytic at $0$ such that $f(0)=0$ and $f'(0)\neq 0$. Let $g(x)$ be an analytic inverse for $f$ converging in some neighborhood of $0$. 
Let $$f(x)=\sum_i b_i x^i$$ and 

$$g(x)=\sum_j c_j x^j$$  be the power expansions of $f$ and $g$ respectively.

Then for each $n$, $c_n$ depends only on $b_1,\ldots,b_n$. Moreover, if $N_+(f)$ is finite, $N_+(g)\geq N_+(f)$.

\end{lemma}

\begin{proof}
Fact \ref{inversionf} gives us the existence of $g$.

The inequality $N_+(g)\geq N_+(f)$ is a consequence of the usual Lagrange inversion formula but we present a proof anyway. We will prove that $c_n$ depends only on $b_1,b_2,\ldots,b_n$ by induction on $n$. For $n=1$, since $f\circ g(x)=x$, composing the power series we get $c_1=1/b_1$. Assume therefore that $c_n$ depends only on $b_1,\ldots, b_n$ and we will prove it for $n+1$. As we have that $$x=f(g(x))=\sum_i b_i\left(\sum_j c_j x^j\right)^i$$ the coefficient of $x^{n+1}$ on the right side of equality has to be $0$. But this coefficient is 
\begin{equation}\label{coefCompuestaEcuacion}
b_1 c_{n+1} + b_2 d_2 +\ldots+b_n d_n+b_{n+1}c_1^{n+1},
\end{equation}
where $d_2,\ldots,d_n$ are some polynomials quantities depending only on $c_1,\ldots,c_n$ that, by induction, depend only on $b_1,\ldots,b_n$. Therefore, $$c_{n+1}=\frac{-b_2 d_2 - \ldots -b_n d_n - b_{n+1} c_1^n}{b_1}$$ depends only on $b_1,\ldots,b_n,b_{n+1}$.

Thus $c_{n,a}$ depends only on $b_{1,a},\ldots,b_{n,a}$ and as for all $k< N_+(f)$ $b_{k,a}$ is constant as $a$ varies then if $k\leq N_+(F)$, $c_{n,a}$ is constant as $a$ varies. This implies that $N_+(g)\geq N_+(f)$.

\end{proof}

\begin{corollary}
Keeping the notation of Lemma \ref{N}. If $N_+(f)$ and $N_+(g)$ are both finite then $N(f)=N(g)$.
\end{corollary}
\begin{proof}
By Lemma \ref{N} one has that $N_+(f)\leq N_+(g)$ apply the same lemma for $g$ and get that $N_+(g)\leq N_+(h)$ where $h$ is the inverse of $g$ but then $h=f$ in some open neighborhood of $0$ so $N_+(f)=N_+(g)$.
\end{proof}

\begin{lemma}\label{lemmaN}
Assume $$f=\sum_{n\geq 1}b_n x^n$$ is an analytic function converging in a neighborhood of $0$. Assume moreover that $N_+(f)$ is finite. Let $$l=\min \{e\in \mathbb N: \exists n>1 \text{ such that }p\nmid n\wedge  b_{np^e}\neq 0\}.$$
Then for all $k<l$ and all $m\geq 1$ such that $p\nmid m$ the coefficient $b_{mp^k,a}$ does not depend on $a$. Moreover $N_+(f)=p^l$. 
\end{lemma}

\begin{proof}
Keeping the notation of Definition \ref{defN}, let $l$ be as in the statement of the lemma. It is well defined because of our assumption on $N(f)$: If $$\{e\in \mathbb N: \exists n>1 \text{ such that }p\nmid n\wedge  b_{np^e}\neq 0\}=\emptyset$$ then $f$ is a function in which the degree of all the non-zero monomials are powers of $p$ so $f(x+a)=f(x)+f(a)$ and therefore $s_n(f)$ is finite for all $n>0$.
 
We will show that for all $k<l$ and all $m\geq 1$ with $p\nmid m$ the coefficient $b_{mp^k,a}$ is constant as $a$ varies. And we will also show that the coefficient $b_{p^l,a}$ assumes infinitely many values as $a$ varies. 
 
Let $n>1$ such that $b_{np^l}\neq 0$ and $p\nmid n$. Therefore $$f(x)=b_1 x + b_{p^2} x^{p^2} + \ldots + b_{p^l} x^{p^l} + \sum_{i>p^l} b_i x^i.$$ 

We will show first that if $k<l$ and $m\geq 1$ with $p\nmid m$, then the coefficient $b_{mp^k,a}$ is constant as $a$ varies. Remember that $b_{mp^k,a}$ is the coefficient of $x^{mp^k}$ in the expansion of $f_a(x)=f(x+a) - f(a)$. In this expansion $x^{mp^k}$ appears only in terms of the form  $b_t (x+a)^t$ with $t\geq mp^k$. Write $t=u p^s$ with $p\nmid u$. Suppose first $u>1$ then if $s<l$, $b_t=0$. If $s\geq l>k$ then $(x+a)^t=(x^{p^s} + a^{p^s})^u$. And therefore in the expansion of $(x+a)^t$ all the exponents of $x$ are multiples of $p^s$ so none of those equals $mp^k$ (as $s>k$ and $p\nmid m$). We can assume then that $u=1$ but in this case $(x+a)^t=x^{p^s} +a^{p^s}$ and the only way the term $x^{mp^k}$ appears is $k=s$ and $m=1$. In this case the coefficient of $x^{mp^k}$ in $f_a(x)$ is $b_{mp^k}$ so it is constant as $a$ varies.\\

We show now that the coefficient $b_{p^l,a}$ is not constant as $a$ varies. As $$f(x)=b_1 x + b_{p^2} x^{p^2} + \ldots + b_{p^l} x^{p^l} + \sum_{i>p^l} b_i x^i$$ and $b_{np^k}=0$ for all $n>0$ and $k<l$ the function $\displaystyle\sum_{i>p^l} b_i x^i$ is indeed a function in the variable $x^{p^l}$. So $$f(x)=b_1 x + \ldots + b_{p^l} x^{p^l} +\sum_{i>1} b_{i p^l}x^{ip^l}$$ therefore the coefficient of $x^{p^l}$ in $f_a$ is $$b_{p^l,a}=b_{p^l} + \sum_{i>1} b_{i p^l} (a^{p^l})^i $$ but this is a series in the variable $a$ with some non-zero coefficients, using Fact \ref{identityTheoremf} we conclude that it is not constant as $a$ varies in any open neighborhood of $0$, so it takes infinitely many values in any open neighborhood of $0$ where it converges and we conclude. 
\end{proof}

We will also need:

\begin{lemma}\label{compositionMultLemma}
 Let $U\ni 1$ be open and $h:U\to K$ and $g:U\to K$ be analytic functions such that $h(1)=g(1)=1$
If $$h(x)=1+\sum_{n\geq 1} a_n (x-1)^n,$$  $$g(x)=1+\sum_{n\geq 1} b_n (x-1)^n$$ 
and $$(h\circ g)(x)=1+\sum_{n\geq 1} c_n (x-1)^n$$ then for each $N$, $c_N$ depends only on $a_j$ and $b_j$ for $j\leq N$.

In addition, if $a_1=b_1=1$ and $N\geq 2$ is such that $a_n=b_n=0$ for $1<n<N$,  then $c_N=a_N+b_N$.

\end{lemma}

\begin{proof}
Notice that 
$$h\circ g(x)=1+\sum_{n\geq 1} a_n \left(\sum_{m\geq 1} b_m(x-1)^m\right)^n$$
and the coefficient that multiply $(x-1)^N$ in the right side of this expression is 
\begin{equation}\label{composeCoefEquation}
    c_{N}= a_1 b_{N} + a_2 f_2 + \ldots +a_{N-1}f_{N-1} + a_N b_1^N
\end{equation}

Where for $n=2,\ldots, N-1$, $$f_n=\sum_{i_1+\ldots+i_n=N} b_{i_1}\cdots b_{i_n}.$$ So $f_n$ is an expression that only depends on $b_1,\ldots,b_N$ and then $c_N$ only depends on $a_1,\ldots,a_N$ and $b_1,\ldots,b_N$.

Now if $a_n=b_n=0$ for all $1<n<N$, then $f_2=f_3=\ldots=f_{N-1}=0$ so Equation \ref{composeCoefEquation} turns into:

$$c_N = a_1 b_N+ a_N b_1^N$$ and if $a_1=b_1=1$, then 
$$c_N=a_N+b_N$$ which is what was left to prove.
\end{proof}

\chapter{Preliminaries}\label{preliminaries}

In this chapter we present some of the results that we will use in Chapters \ref{additiveCaseVersion} and \ref{multiplicativeCaseVersion}.

We fix $\mathbb K=(K,+,\cdot,v,\Gamma)$ an algebraically closed valued field and denote by $\mathbb K^{f}$ the reduct of $\mathbb K$ to the language of rings so $\mathbb K^{f}$ is an algebraically closed field. We set $p=\text{char}(K)\geq 0$.

\section{Preliminaries on Groups}

%From now on in this Section we fix $(G,\oplus,e)$ a $\mathbb K$-definable group. For $a\in G$ we denote by $\om a$ to the group inverse of $a$ in $G$ and we write $a\om b$ instead of $a\+(\ominus b)$.

\begin{defi}

    Let $(G,\oplus,e)$ be any Abelian group. For $a,b\in G$ denote by $\ominus a$ the inverse of $a$ in $G$ and we write $a\ominus b$ instead of $a\+(\ominus b)$. For $a=(a_1,a_2)\in G\times G$ and $Y\subseteq G\times G$ we define the translation of $Y$ by $a$ as:

    $$t_a(Y):=\{(x_1\ominus a_1,x_2\ominus a_2):(x_1,x_2)\in Y\}.$$

    %Notice that if $Y$ is $\mathcal G$-definable, then so is $t_a(Y)$.
\end{defi}

\begin{defi}
    If $Y,Z\subseteq G\times G$ we define:

   $$Y\+Z:=\{(x,y_1\+y_2):(x,y_1)\in Y\text{ and }(x,y_2)\in Z\},$$ 
   $$Y\om Z:=\{(x,y_1\ominus y_2):(x,y_1)\in Y\text{ and }(x,y_2)\in Z\}$$
   and 
   $$Y\circ Z:=\{(x,y)\in G\times G:\exists z\ (x,z)\in Z\text{ and }(z,y)\in Y \}.$$   
\end{defi}

\begin{lemma}\label{lemmaSumaComp}
Let $Y,Z\subseteq G\times G$. Assume that $U\ni e$ is any subset of $G$ and $g:U\to G$ and $h:U\to G$ are functions such that $g(e)=h(e)=e$ and whose graphs are contained in $Y$ and $Z$ respectively. 

Then for all $x\in U$,
    
        $$(x,g(x)\+h(x))\in Y\+Z$$ 
        and
        $$(x,g(x)\ominus h(x))\in Y\ominus Z.$$
    Moreover if  $G$ is a topological group, $g$ is continuous and $U$ is open, there is an open set $U'\ni e$ contained on $U$ such that for all $x\in U'$
      
      $$(x,g(h(x))\in Y\circ Z.$$
    \end{lemma}

\subsection{$\mathbb K$-definable and locally isomorphic groups}

\begin{defi}
  By an algebraic group we mean the $\mathbb K^f$ points of an algebraic group defined over $\mathbb K^f$.
\end{defi} 

\begin{defi}\label{definitionLocllyIsomorphic}
    If $(G,\oplus)$ is a $\mathbb K$-definable group we say that $G$ is locally isomorphic to an algebraic one dimensional group $(H,+_H)$ if there is $A\leq H$, an infinite $\mathbb K$-definable subgroup of $H$ and $i:A\to G$ a $\mathbb K$-definable injective group homomorphism such that $G/i(A)$ is finite. 
\end{defi}

We will need the following technical lemma:

\begin{lemma}\label{noIsolated}
 Let $(H,+_H)$ be a $\mathbb K$ definable infinite subgroup of $\mathbb G_a$ or $\mathbb G_m$. Let $Y$ and $Z$ be one dimensional $\mathbb K$-definable subsets of $H^2$. If $Y$ and $Z$ have no isolated points, then $Y+_H Z$  does not have isolated points. If, moreover, we assume that $Z$ does not contain infinitely many points in any horizontal line and $Y$ does not contain infinitely many points in any vertical line, then the composition $Y\circ Z$ does not have isolated points.
 \end{lemma}
 
 \begin{proof}
 %Let's prove first that $Y_a + Y_b$ don't have any isolated point. It is clear that if $a=(a_1,a_2)$ and $x=(x_1,x_2)\in Y_a$ is isolated then $(x_1+a_1,x_2+a_2)$ is isolated in $Y$ so $Y_a$ and $Y_b$ don't have isolated points.
 
 If $(a,b)\in Y +_H Z$ there are $c$, $d$ such that $(a,c)\in Y$, $(a,d)\in Z$ and $b=c +_H d$. Take open sets $B_1$ containing $a$ and $B_2$ containing $b$. As $+_H$ is continuous, one can find $D$ and $E$ open sets containing $c$ and $d$ respectively such that $D+_H E \subseteq B_2$. 
 
 Given that $(a,c)\in Y$ is not isolated, by Proposition \ref{decompositionProp} there is a polynomial $F$ such that $F(a,c)=0$ and for all $(a',c')$ in some neighborhood of $(a,c)$ if $F(a',c')=0$ then $(a',c')\in Y$. If the polynomial $F_a(x):=F(a,x)$ is not zero we can use continuity of roots and get $B$, a neighborhood of $a$ (we can assume $B\subseteq B_1$) such that for each $a'\in B$ there is $c'\in D$ with $F(a',c')=0$ and $(a',c')\in Y$. In the same way (by restricting $B$ if necessary) for each $a'\in B$ there is $d'\in E$ such that $(a',d')\in Z$ so for any such $a'$ the point $(a',c'+_H d')\in (Y+_H Z)\cap (B_1\times B_2)$ and then $(a,b)$ is not isolated in $Y+_H Z$. 
 %$\overline{Y} \bar Y$  $\tilde Y$ $\widetilde Y$  $\mathbb N$
 
 Note that if one of $F(a,\_)$ is the zero polynomial, then $Y$ contains infinitely many points in the vertical line $\{a\}\times K$, so that any open containing $(a,b)$ contains points in such a line. Similarly for $Z$.
 
We will now prove that $Y\circ Z$ does not contain isolated points. Let $(a,b)\in Y\circ Z$ so there is $c$ with $(a,c)\in Z$ and $(c,b)\in Y$.
 Let $B_1$ and $B_2$ be open sets containing $a$ and $b$ respectively.  There is some set $B$ containing $c$ such that for each $c'\in B$ there is: $a'\in B_1$ with $(a',c')\in Z$ (using that the closure of $Z$ does not contain any horizontal line) and $b'\in B_2$ such that $(c',b')\in Y$ (using that the closure of $V$ doesn't contain any vertical line). Then, for each such $c'$, $(a',b')\in (Y\circ Z)\cap(B_1\times B_2)$. This shows that $(a,b)$ is not an isolated point of $Y\circ Z$.
 \end{proof}

\subsection{Strongly minimal expansions of groups}
 
 Let $(G,\oplus)$ be a group and let $\mathcal G=(G,\oplus,\ldots)$ be a strongly minimal first order structure expanding $(G,\oplus)$.

 \begin{defi}
     We say that $X\subseteq G\times G$ is $\mathcal G$-affine if it is a (finite) boolean combination of cosets of $\mathcal G$-definable subgroups of $(G,\oplus)\times (G,\oplus)$.
  \end{defi}

The following fact follows from Theorem 7.2 of \cite{castleHasson}.

\begin{fact}\label{existsXVersion}
    If $\mathcal G=(G,\+,\ldots)$ is strongly minimal, then $\mathcal G$ is non locally modular if and only if there is $X\subseteq G\times G$ a $\mathcal G$-definable set with $\RM_{\mathcal G}(X)=1$ that is not $\mathcal G$-affine.
\end{fact}

Assume that $\mathcal G$ is strongly minimal and non locally  modular, let $X$ be as provided by Fact \ref{existsXVersion} above. If $X$ has Morley degree $d$ we can write
$$X=X_1\cup\ldots \cup X_d$$ where each $X_i$ is a $\mathcal G$-definable strongly  minimal subset of $G\times G$. As $X$ is not $\mathcal G$-affine, there is some $i\leq d$ such that $X_i$ is not $\mathcal G$-affine.Thus, if we replace $X$ by $X_i$, we may assume that $X$ is strongly minimal and it is not $\mathcal G$-affine.

The following is a very well known fact, it follows for example from Lemma 3.8 of \cite{HE}.

\begin{fact}\label{nonAffineIntersectsFiniteLemmaF}
      Let $Y\subseteq G\times G$ be a $\mathcal G$-definable (over $\bar d$) strongly minimal set that is not $\mathcal G$-affine. Then if $(a_1,a_2)\in Y^2$ satisfies the $\mathcal G$-generic type of $Y\times Y$ over $\bar d$, $$t_{a_1}(Y)\cap t_{a_2}(Y)$$ is finite. 
\end{fact}

If $\mathcal G$ is definable in $\mathbb K$ then we may extend the notion of not $G$-affine for  $\mathbb K$-definable sets:

\begin{defi}\label{GaffineDefi}
    If $G$ and $Y\subseteq G\times G$ are $\mathbb K$-definable we say that $Y$ is not $G$-affine if for all but finitely many $a\in Y$ there are cofinitely many $b\in Y$ such that $t_a(Y)\cap t_b(Y)$ is infinite.
\end{defi}

Moreover this notion is definable in first order:
By Corollary \ref{existsInf}, if $\phi(\bar x,\bar y,\bar z)$ is a first order formula in the language of valued fields and for each tuple $\bar d$ we set $$Y(\bar d):=\{(\bar x,\bar y)\in G\times G:\mathbb K\models \phi(\bar x,\bar y, \bar d)\}$$ then
$$\{\bar d:Y(\bar d)\text{is not }G\text{-affine}\}$$ is $\mathbb K$-definable.

%We will use the results of Section \ref{computations} in order to find a group configuration on $\mathcal G$. Given that Lemma \ref{intersectionSubeLemmaAnalitic} assumes that the intersection is finite, we need to address the case where the intersection is infinite. In this section, we develop the machinery needed to deal with that case.

\section{Generic Intersections}

In this section we present the basic results on generic intersection between families of curves that we will need. Most of the results here are standard are well known, anyhow we present some proofs for completeness.

In this section we use the notion of generic of Definition \ref{genericDefinition} and we fix $Y\subseteq K\times K$, a one-dimensional $\mathbb K$-definable set. 
Given $Z\subseteq K^d$ we denote by $\cl(Z)$ the Zariski closure of $Z$ on $K^d$.
\begin{defi}   
    Let $(X_a)_{a\in Y}$ be a family of one-dimensional subsets of $K\times K$ $\mathbb K$-definable over $\bar d$. 

    We say that $(X_a)_{a\in Y}$ has \emph{generic finite intersection} if $X_a\cap X_b$ is finite for any $(a,b)\in Y^{2}$ $\mathbb K$-generic over $\bar d$.
\end{defi}

We start proving:

\begin{lemma}\label{bastaIndiscernibleLemma}
    Let $(X_a)_{a\in Y}$ be a family of one-dimensional Zariski closed subsets of $K\times K$ $\mathbb K$-definable over $\bar d$. Then $(X_a)_{a\in Y}$ has generic finite intersection if $X_{b_1}\cap X_{b_2}$ is finite for any $(b_1,b_2)\in Y^2$ such that there is an infinite indiscernible sequence starting with $(b_1,b_2)$.
\end{lemma}

\begin{proof}
$(b_1,b_2)\in Y^2$ is $\mathbb K$-generic if and only if  $b_1\in Y$ is generic and $b_2\in Y$ is generic over $b_1$ if and only if $(b_1,b_2)$ starts an indiscernible sequence.
  \end{proof}

\begin{lemma}\label{noGenericFiniteZLemma}
    Let $(X_a)_{a\in Y}$ be a family of one-dimensional Zariski closed subsets of $K\times K$ $\mathbb K$-definable over $\bar d$. Then if $(X_a)_{a\in Y}$ has generic finite intersections, $X_a\cap Z$ is finite for any $Z\subseteq K\times K$ one dimensional $\mathbb K$-definable over $\bar e$ and any $a\in Y$ generic over $\bar d, \bar e$. 
\end{lemma}
\begin{proof}

Assume that there is $Z\subseteq K\times K$ one-dimensional and $\mathbb K$-definable over $\bar e$ and some $a\in Y$ generic over $\bar e, \bar d$ such that $X_a\cap Z$ is infinite. 

As $X_a$ is Zariski closed we can write

$$X_a=X_0 \cup \bigcup_i C_i $$ where the union is finite, $X_0$ is a finite set of isolated points and each $C_i$ is a one-dimensional irreducible Zariski closed set.

By Proposition \ref{decompositionProp} we can write 

$$Z=Z_0\cup \bigcup_j(D_j\cap U_j)$$

$Z_0$ is a finite set of isolated points, each $D_j$ is an  irreducible Zariski closed set of dimension one and each $U_j$ is an open subset of $K\times K$.

As $X_a\cap Z$ is infinite, there is some $i$ and some $j$ such that $C_i \cap (D_j\cap U_j)$ is infinite. In particular $C_i\cap D_j$ is infinite and as both sets are Zariski closed and irreducible it implies that $C_i=D_j$. Let $\bar f$ be a tuple of parameters defining $C:=C_i=D_j$. Then if $b\in Y$ is generic over $a,\bar d,\bar e,\bar f$, as $X_a\cap C$ is infinite, $X_b\cap C$ is infinite but as $X_a$ and $X_b$ are both Zariski closed it implies that $C$ is a irreducible component of both, so $C\subseteq X_a\cap X_b$ and then $(X_a)_{a\in Y}$ has not generic finite intersection. \end{proof}

For the following lemma $(H,+_H)$ is either the additive or the multiplicative group and $-_H$ denotes the correspondent inverse operation. 

\begin{lemma}\label{restaCompuestaFiniteIntersectionLemma}
      Let $(X_a)_{a\in Y}$ be a $\mathbb K$-definable family of one-dimensional Zarsiki closed sets with  generic finite intersection. Let $Z\subseteq K\times K$ be a one-dimensional Zariski closed set, then:
    $$(\cl(Z-_H X_a))_{a\in Y}$$
     has generic finite intersection. 

     Moreover, if we assume that $Z$ does not contain infinitely many points in any horizontal line and also that for all $a\in Y$, $X_a$ does not contain infinitely  many points in any vertical line, then
    $$(\cl(Z\circ X_a))_{a\in Y}$$
    has generic finite intersection. 
    
    If we assume that $Z$ does not contain any infinitely many points in any vertical line, and also that $X_a$ does not contain infinitely  many points in any horizontal line  for all $a\in Y$, then    $$(\cl(X_a\circ Z))_{a\in Y}$$
    has generic finite intersection.
   
\end{lemma}

This is standard but we include a proof anyway.

\begin{proof}
    Let $\bar d$ be a tuple of parameters defining the family $(\cl(Z-_H X_a))_{a\in Y}$. Assume that $Z$ and the family $(X_a)_{a\in Y}$ are also definable over $\bar d$. Let $a_1,a_2\in Y$ be generic independent over $\bar d$.  Then $X_{a_1}\cap X_{a_2}$ is finite so 
    $$Z-_H (X_{a_1}\cap X_{a_2})=(Z -_H X_{a_1})\cap (Z -_H X_{a_2})$$
    is also finite. As $Z-_HX_{a_1}$ is constructible, $$\cl(Z-_H X_{a_1})\setminus (Z-_H X_{a_1})$$ has dimension strictly smaller than $\dim(Z-_H X_{a_1})=1$ so it is finite. The same is true for $Z-_H X_{a_2}$ so $$\cl(Z-_H X_{a_1}) \cap \cl(Z-_H X_{a_2})$$ is finite. 

   For the second statement, assume that $Z$ does not contain infinitely many points in any horizontal line. Let $\bar d$ be a tuple of parameters such that the family $(Z\circ X_a)_{a\in Y}$ is $\mathbb K$-definable over $\bar d$ as well as $Z$ and the family $(X_a)_{a\in Y}$. Let $a_1,a_2\in Y$ be generic independent elements over $\bar d$ and assume that 
   $$\cl(Z\circ X_{a_1})\cap \cl(Z\circ X_{a_2})$$
   is infinite. By an argument similar to the previous one it implies that 
   $$(Z\circ X_{a_1})\cap (Z\circ X_{a_2})$$
   is infinite. By Lemma \ref{bastaIndiscernibleLemma} we may assume that there is $\mathcal I$, an indiscernible sequence over $\bar d$ that starts with $(a_1,a_2)$.
   
   Fix $(\alpha,\beta)\in (Z\circ X_{a_1})\cap (Z\circ X_{a_2})$ generic over $a_1 a_2\bar d$. Then there is some $\gamma_{1}$ such that $(\alpha,\gamma_{1})\in X_{a_1}$ and $(\gamma_{1},\beta)\in Z$. As $Z$ does not contain infinitely many points in any horizontal line, 
   $$Z\cap \{(x,\beta):x\in K\}$$ is finite. Since it is definable over $\beta\bar d$ and contains $\gamma$, $\gamma\in \acl_{\mathbb K}(\beta,\bar d)$. 

 Notice that as $a_1$ and $a_2$ are generic $\mathbb K$-independent over $\bar d$ and $(\alpha,\beta)\in (Z\circ X_{a_1})\cap (Z\circ X_{a_2})$ is generic over $a_1,a_2,\bar d$ it implies that $a_1$ and $a_2$ are independent over $\acl(\alpha,\beta,\bar d)$. But as $\gamma_{i}\in \acl(a_i)$, $(a_1,\gamma_{1})$ and $(a_2,\gamma_{2})$ are independent over $\acl(\alpha,\beta,\bar d)$. 

As $a_1,a_2$ starts an infinite indiscernible sequence,  there is some indiscernible sequence that starts with $(a_1,\gamma_1),(a_2,\gamma_2)$ say $\mathcal I'=((a_1,\gamma_1),(a_2,\gamma_2),(a_3,\gamma_3)\ldots)$. Thus,  $\gamma_i=\gamma_j$ for all $i,j$ or $\gamma_i\neq \gamma_j$ for all $i\neq j$. As there are only finitely many $\gamma$ such that $(\gamma,\beta)\in Z$ the second option is not possible, so there is some $\gamma$ such that $\gamma=\gamma_i$ for all $i\in \mathbb N$. It implies that $(\alpha,\gamma)\in X_{a_1}\cap X_{a_2}$. But this is true for any $\alpha'$ such that there is some $\beta'$ with $(\alpha',\beta')\in (Z\circ X_{a_1})\cap (Z\circ X_{a_2})$ generic. In particular as $X_{a_1}$ does not contain any vertical line, for infinitely many $\alpha'$ there is some $\gamma'$ such that $(\alpha',\gamma')\in X_{a_1}\cap X_{a_2}$ which is a contradiction with the assumption that $(X_a)_{a\in Y}$ has generic finite intersection.

   The third statement is very similar. \end{proof}

\section{Property $(\star)$}\label{propertyStar}

%\section{Property $(\star)$}\label{propertyStar}

In this section we define Property $(\star)$ and provide a way of finding sets with this property. Having a definable set with the Property $(\star)$ is the first step in order to get a definable field in Chapters \ref{additiveCaseVersion} and \ref{multiplicativeCaseVersion}.
%This $Y$ allows us to find a definable family of curves with some good properties (see Lemma \ref{existsHLemma}) and, in turn, this family allows us to find a field configuration in $\mathcal G$.

%Since the results of this section will be needed in Chapter \ref{multiplicativeCaseVersion} we state them in a more general setting:

\begin{quote}
\textbf{In Section \ref{propertyStar}, $(H,+_H,e)$ will denote either the multiplicative group or an infinite $\mathbb K$-definable subgroup of the additive group .} %Moreover, $\mathcal H=(H,+,\ldots)$ is a strongly minimal structure expanding $(H,\oplus)$ and definable in $\mathbb K$.}
  % \textbf{ In both cases $H$ is an open subset of $K$.}
 \end{quote}

%We fix $A\subseteq K$ an infinite $\mathbb K$-definable subgroup of $H$ and $i:A\to G$ an injective $\mathbb K$-definable group homomorphism whose image has finite index on $G$. We assume that $i$ is an inclusion so $A\leq G$ and $G/A$ is finite. We also fix $g_0=e,g_1,\ldots, g_m\in G$ a set of representatives for $G/A$ and for $k=0,\ldots,m$ we  define $G_k:=g_k \oplus A$ the translation of $A$ and $i_k:A\to G$ is the injective map with image $G_k$ defined by $i_k(x)=g_k\oplus x$. 
% Whenever we say that a $\mathcal H$-definable set is strongly minimal, we mean that it is strongly minimal respect to $\mathcal H$. Moreover, for $a\in H$ we denote by $\ominus a$ to the inverse of $a$ in $H$. 

By Proposition \ref{decompositionProp} if $Y\subseteq K\times K$ is a $1$-dimensional $\mathbb K$-definable set with no isolated points we can write:
$$Y=\bigcup_{i=1}^l\{(x,y)\in V_i:L_i(x,y)=0\}$$
where each $L_i$ is an irreducible polynomial and $V_i$ is an open set. If we require that 
$$\{(x,y)\in V_i:L_i(x,y)=0\}\neq \emptyset$$ and that $L_i\neq \lambda L_j$ for $i\neq j$ and all $\lambda\in K$, the set of polynomials $L_i$ is uniquely determined by $Y$, in this sense we define:

\begin{defi}\label{versionLemmaBuenYdef}
  Given $\mathbb K$-definable $1$-dimensional $Y\subseteq H\times H$ and a finite set of polynomials $L_i$ and $L_{i,n_i}$, we say that $Y$ has the Property $(\star)$ witnessed by $L_i$ and $L_{i,n_i}$ if: 

  \begin{enumerate}

       \item\label{clause1} $Y$ has no isolated points.
        
       \item\label{clause2} $$Y=\bigcup^{l}_{i=1} \{(x,y)\in V_i:L_i(x,y)=0\},$$ with $V_i$ open subsets of $K\times K$.

       \item\label{clause3} $(e,e)\in Y$, $L_1$ is irreducible and $L_1(e,e)=0$. Moreover, $L_i(e,e)\neq 0$ for $i\neq 1$.

       \item\label{clause4} $$L_i(x,y)=L_{i,n_i}(x,y^{p^{n_i}})$$ and if $(a_1,a_2)\in Y$ is such that $L_i(a_1,a_2)=0$ then $$\frac{\partial L_{i,n_i}}{\partial y}(a_1,a^{p^{n_i}}_2)\neq 0.$$
       Moreover $n_1=0$, in particular $$\frac{\partial L_{1}}{\partial y}(e,e)\neq 0.$$

       %\item\label{clause5} There are at most finitely many $c=(c_1,c_2)\in H^2$ such that $L_1(c)=0$ and 
       %$$\left\{(x,y)\in K^2:L_1(x,y)=L_1(x+_m c_1,y+_m c_2)=0\right\}$$ is infinite.

       %there is a neighborhood of $(0,0)$, say $O\subseteq O_1$, such that for all $(x,y)\in V_1\cap O$ one has that $F_y(x,y)\neq 0$ and the set:
      % $$s_1(Y)=\left\{\frac{-F_x(x,y)}{F_y(x,y)}:(x,y)\in V_1\cap O\right\}$$ is infinite.
     
  \end{enumerate}

\end{defi}

Note that if $\phi(x,y,\bar d)$ is a formula in the language of valued fields and $L_i(x,y,\bar d)$ and $L_{i,n}(x,y,\bar d)$ are polynomials, then if we define:

$$Y(\bar d)=\{(x,y)\in K\times K: \mathbb K\models \phi(x,y,\bar d)\}$$ we have that

$$\left\{\bar d:Y(\bar d)\text{ has the Property }(\star)\text{ witnessed by }L_i(x,y,\bar d)\text{ and }L_{i.n}(x,y,\bar d)\right\}$$ is $\mathbb K$-definable:

Clauses 1. and 4. of Property $(\star)$ are clearly definable. Clause $3$ is definable because $L_1$ is irreducible if and only if $L_1\neq F_1\cdot F_2$ for all $F_1$ and $F_2$ with degree smaller than $ \deg L_1$ which can be defined quantifying over the coefficients of $F_1$ and $F_2$. %Clause $5$ is definable by Corollary \ref{existsInf}. 
Clause 2 is equivalent to:

\begin{quote} 
If $(\alpha_1,\alpha_2)\in Y(\bar d)$ then $\bigvee_i L_i(\alpha_1,\alpha_2,\bar d)=0$ and if $L_i(\alpha_1,\alpha_2,\bar d)=0$, there is some $\gamma\in \Gamma$ such that for all $(a_1,a_2)$, if $v(\alpha_1-a_1)>\gamma$, $v(\alpha_2-a_1)>\gamma$ and $L_i(a_1,a_2,\bar d)=0$, then $(a_1,a_2)\in Y(\bar d)$.

\end{quote}

     Which is a first order statement.

%From now we fix $Y$ as in Lemma \ref{versionLemmaBuenY}. We also fix $c\in V_1\cap Y$ such that $L_1(c)=0$.

%\begin{defi}\label{versionGoodFamily}

 %   Now we define a family of curves $\mathcal X=(X_a)_{a\in Y}$ given by:

 %$$X_a=(Y-t_a(Y))\circ (Y-t_c(Y))^{-1}$$ for $a\in Y$.
%\end{defi}

%\begin{prop}\label{versionPropPolinomioUnico}
 %  There is an open neighborhood of $(0,0)$, say $O\subseteq O_1$ and a polynomial $G(x,y,z_1,z_2)$ with coefficients on $K$, such that for all $a=(a_1,a_2)\in V_1\cap O$:
  % \begin{enumerate}
   %    \item $G(0,0,a_1,a_2)=0.$
    %   \item $G_y(0,0,a_1,a_2)\neq 0.$
     %  \item There is some open set $O_a\subseteq O$ containing $(0,0)$ such that for all $b=(b_1,b_2)\in O_a$ if $G(b_1,b_2,a_1,a_2)=0$ then $b\in X_a$.
 %  \end{enumerate}
    
%\end{prop}

%\begin{defi}
 %   Let $O$ and $G$ be as in Proposition  \ref{versionPropPolinomioUnico}, then we define $$s_1(\mathcal X)=\left\{\frac{-G_x(0,0,a_1,a_2)}{\ \ G_y(0,0,a_1,a_2)}:(a_1,a_2)\in V_1\cap O\right\}.$$
%\end{defi}

%\begin{prop}
 %   Let $\mathcal X$ be as in Definition \ref{versionGoodFamily}, then the set $s_1(\mathcal X)$ is infinite.
%\end{prop}

\begin{lemma}\label{lemmaYStar}
    Let $X\subseteq H\times H$ be a one dimensional $\mathbb K$-definable set. Then there is a point $a\in X$, some $j\in \{-1,1\}$ and a finite set $E\subseteq t_a(X)^j$ such that  $Y=t_a(X)^j\setminus E$ has Property $(\star)$ witnessed by some polynomials $L_i$ and $L_{i,n_i}$.

    Here the translations $t_a$ are taken respect to $H$, $t_a(X)^{1}=t_a(X)$ and $t_a(X)^{-1}$ is the inverse of $t_a(X)$:
    $$t_a(X)^{-1}:=\left\{(x,y)\in H\times H:(y,x)\in t_a(X)\right\}.$$
\end{lemma}

\begin{proof}

    By Proposition \ref{decompositionProp} we can write $$X=X_0\cup \bigcup_i (C_i\cap V_i)$$ 
    where $X_0$ is a finite set of isolated points of $X$, $C_i$ is a one dimensional irreducible Zariski closed set and $V_i$ is an open set. Let $\tilde L_i$ be irreducible polynomials such that 
    $$C_i=\left\{(x,y):\tilde L_i(x,y)=0\right\}.$$

   % We replace $X$ by $X\setminus X_0$ so we may assume that $X$ has no isolated points.

  %  \begin{claim}
   %     LAssume that $C_i$ is, except by finitely many points, a coset of a subgroup of $(K,+)\times (K,+)$. Then 
    %\end{claim}

    %\begin{claimproof}
     %    Taking a translation we may assume that $(0,0)\in X\cap C_i$ is an interior point and also that $C_i$ is, except by finitely many points, a subgroup of $(K,+)\times (K,+)$.

         %Let $U_0\subseteq K\times K$ be an open set such that $(0,0)\in U_0$ and %$C_i\cap U_0\subseteq X$.

       %  Let  $S$ be the $\mathcal G$-definable set       
      %        $$S:=\{c\in G\times G: t_c(X)\cap X \text{ is infinite} \}.$$

     %   First let us prove that $S$ is a subgroup of $(G,+)\times (G,+)$. 

    %    Now we prove that $S\triangle X$ is finite. 
        
        %Let $c\in C_i\cap X$ be an interior point such that $t_c(C_i)=C_i$. Let %$U\subseteq K\times K$ be an open neighborhood of $c$ such that $C_i\cap %U\subseteq X$, then $$t_c(U\cap C_i)=t_c(U)\cap t_c(C_i)=t_c(U)\cap C_i$$ so %$t_c(U)\cap U_0\cap C$ is infinite and it is contained on $X$. Then all of %those $c$ belong to $S$ so $S\cap X$ is infinite and as $X$.

     %   So $S=t_{-a}(S')$ satisfies the conclusion of Claim.
    %\end{claimproof}
    
    By Fact \ref{finiteSingularPointsFact} there is $a\in C_1 \cap X$, a regular point of $C_1$ such that $a\not\in C_k$ for $k\neq 1$, so either 
    $$\frac{\partial \tilde L_1}{\partial y}(a)\neq 0$$  or
    $$\frac{\partial \tilde L_1}{\partial x}(a)\neq 0.$$

   In the former case take $j=1$ and set 
    $Z:=t_a(X)$.
    
   In the latter case take $j=-1$ and $Z:=(t_a(X))^{-1}$.
    
   % Apply Proposition \ref{decompositionProp} with $Z$ and
   Write \begin{equation}\label{decompositionZ}Z=Z_0\cup \bigcup_i\left\{(x,y)\in V_i: L_i(x,y)=0\right\}\end{equation}
    where $Z_0$ is a finite set of isolated points of $Z$ and for all $i$, $ L_i$ is an irreducible polynomial and $V_i$ is open. 

    If $Z=t_a(X)$ with $a=(a_1,a_2)$,
    $$L_i(x,y)=\tilde L_i(x+_ma_1,y+_ma_2)$$
    and if $Z=(t_a(X))^{-1}$,
    $$L_i(x,y)=\tilde L_i(y+_m a_1,x+_m a_2).$$ 

   If $H$ is a subgroup of $\mathbb G_a$ and then:
   
   If $Z=t_a(X)$, $$\frac{\partial  L_1}{\partial y}(e,e)=\frac{\partial \tilde L_i}{\partial y}(a),$$
    and if $Z=(t_a(X))^{-1}$, $$\frac{\partial  L_1}{\partial y}(e,e)=\frac{\partial \tilde L_i}{\partial x}(a).$$
    
    If $H$ is the multiplicative group then:
    
If $Z=t_a(X)$, $$\frac{\partial  L_1}{\partial y}(e,e)=\frac{\partial \tilde L_i}{\partial y}(a)\frac{a_1}{a_2},$$
and if $X=(t_a(X))^{-1}$, $$\frac{\partial  L_1}{\partial y}(e,e)=\frac{\partial \tilde L_i}{\partial x}(a)\frac{a_1}{a_2}.$$  

    In any case 
    $$\frac{\partial  L_1}{\partial y}(e,e)\neq 0.$$ 
    %and, as $C_1$ is not a boolean combination of cosets of subgroups of $(G,+)\times (G,+)$, the same is true for

This shows that Clause 3 of Property $(\star)$ holds for $Y$ and $L_1$. Clause 1 will be gotten by removing the finite set of isolated points from $Z$ and then Equation \ref{decompositionZ} will give us Clause 2. So we are missing only Clause 4. 

 We state this as a general claim that we will use latter:

 \begin{claim}\label{existsFrClaim}
 Let $L_1,L_2,\ldots,L_r$ be a finite set of polynomials and $V_i\subseteq H\times H$ be open sets. Set $$Z=\bigcup_i \{x\in V_i:L_i(x)=0\}.$$

 Then there is a finite set $E\subseteq Z$ and some polynomials $L_{i,n_i}$ such that Clause 4 of Property $(\star)$ witnessed by $L_1$ and $L_{i,n_i}$ holds for $Z\setminus E$.
 \end{claim}
 \begin{claimproof}

 For each $i$ such that 
    $$\left\{q\in Z:L_i(q)=\frac{\partial L_i}{\partial y}(q)=0\right\}$$ is finite  we take $L_{i,n_i}=L_i$ (so $n_i=0$), in this case define: 
    $$E_i=\left\{q\in Z:L_i(q)=\frac{\partial L_i}{\partial y}(q)=0\right\}.$$

    If $i$ is such that 
    $$\left\{q\in Z:L_i(q)=\frac{\partial L_i}{\partial y}(q)=0\right\}$$ is infinite, by Lemma \ref{pickCoordinateLemmaP} there is some polynomial $G$ and some natural number $n$ such that 
    $L_i(x,y)=G(x,y^{p^n})$. Let $n_i$ be maximal with this property and let $L_{i,n_i}$ be the polynomial such that $$L_i(x,y)=L_{i,n_i}\left(x,y^{p^{n_i}}\right).$$ By the maximality of $n_i$ one has that $L_{i,n_i}(x,y)$ is not a polynomial in the variable $y^p$ so using again Lemma \ref{pickCoordinateLemmaP},

    $$E_i=\left\{q=(q_1,q_2)\in Z:L_i(q_1,q_2)=\frac{\partial L_{i,n_i}}{\partial y}\left(q_1,q_2^{p^{n_i}}\right)=0\right\}$$ is finite.

    So $E:=\bigcup_i E_i$ is a finite set such that $Z\setminus E$ satisfies Clause 4. of Property $(\star)$.
\end{claimproof}

    Thus, if $E$ is as provided by Claim \ref{existsFrClaim}, $E\cup Z_0 $ is a finite subset of $Z$ such that $Y:=Z\setminus (E\cup Z_0)$ has Property $(\star)$ witnessed by $L_i$ and $L_{i,n_i}$.
\end{proof}

If $G$ is a $\mathbb K$-definable group that is locally isomorphic to $\mathbb G_a$ we may also define Property $(\star)$ for subsets of $G\times G$. Fix $A\leq \mathbb G_a$ an infinite $\mathbb K$-definable subgroup and $i:A\to G$ an injective isomorphism with $G/i(G)$ finite. We may assume that $i$ is an inclusion so $A\leq G$.
Fix $e=g_1,g_2,\ldots, g_m$ a finite set of representatives for $G/A$ and let $G_k:=g_k\oplus A$ so $G=\bigcup_k G_k$ is a disjoint union. For $k,l$ let $i_{k,l}:A\times A\to G\times G$ be defined  by 
$i_{k,l}(x,y)=(g_k\oplus x,g_l\oplus y)$ and for $Z\subseteq G\times G$ let $Z^{(k,l)}\subseteq A\times A$ be the preimage of $Z$ via $i_{k,l}$.

\begin{defi}
We say that $Y\subseteq G\times G$ has Property $(\star)$ witnessed by some polynomials $L_{i},L_{i,n_i}$ and $L_{i}^{(k,l)}$ if $Y\cap (A\times A)$ has Property ($\star$) witnessed by $L_i$ and $L_{i,n}$ and for each pair $(k,l)$ such that $(G_k\times G_l)\cap Y\neq \emptyset$, we have that $(G_k\times G_l)\cap Y$ is infinite and $Y^{(k,l)}$ satisfies Clauses \ref{clause1}, \ref{clause2} and \ref{clause4} of Property $(\star)$ witnessed by the polynomials $L_{i}^{(k,l)}$.
\end{defi}

If $\phi(\bar x,\bar y,\bar d)$ is a formula in the language of valued fields and $L_i(x,y,\bar d)$, $L_{i,n}(x,y,\bar d)$ and $L_i^{(k,l)}(x,y,\bar d)$ are polynomials, then if we define:

$$Y(\bar d)=\{(\bar x,\bar y)\in G\times G: \mathbb K\models \phi(\bar x,\bar y,\bar d)\}$$ we have that

$$\left\{\bar d:Y(\bar d)\text{ has the Property }(\star)\text{ witnessed by }L_i(x,y,\bar d)\text{, }L_{i.n}(x,y,\bar d)\text { and }L_i^{(k,l)}(x,y,\bar d)\right\}$$ is $\mathbb K$-definable.

We have an analogous to Lemma \ref{lemmaYStar}:

\begin{lemma}\label{lemmaYStarG}

Let $X\subseteq G\times G$ be a one-dimensional $\mathbb K$-definable set. Then there is a point $a\in X$, $j\in \{-1,1\}$ and a finite set $E\subseteq t_a(X)^{j}$ such that $Y=t_a(X)^j\setminus E$ has Property $(\star)$ witnessed by some polynomials $L_i$, $L_{i,n_i}$ and $L_i^{(l,k)}$.
     
\end{lemma}

\begin{proof}
    Taking some translate and removing finitely many points from $X$ we may assume that $(e,e)\in X$ and also that $(G_k\times G_l)\cap Y$ is infinite whenever $(G_k\times G_l)\cap Y\neq \emptyset$. So we apply Lemma \ref{lemmaYStar} to $Y\cap ( A\times A)$ and find some $(a_1,a_2)\in Y\cap (A\times A)$, $j\in \{-1,1\}$ and some finite set $E_0\subseteq t_{a}(Y\cap (A\times A))^{j}$ such that $t_{a}(Y\cap (A\times A))^{j}\setminus E_0$ has property $(\star)$ witnessed by some polynomials $L_i$ and $L_{i,n_i}$.  Let $Z:=t_a(Y)\setminus E_0$. Then for each $(k,l)$ if $Z\cap (G_k\times G_l)$ is finite we remove those finitely many points from $Z$ and if $Z\cap G_k\times G_l$ is infinite we remove the finite set of isolated points of $Z^{(k,l)}$ and apply Proposition \ref{decompositionProp} and Claim \ref{existsFrClaim} to find some finite set $E^{(k,l)}\subseteq Z^{(k,l)}$ such that $Z^{(k,l)}\setminus E^{(k,l)}$ satisfies Clauses \ref{clause1}, \ref{clause2} and \ref{clause4} of Property $(\star)$ witnessed by some polynomials $L_i^{(k,l)}$. Thus, if we take $$E:=\bigcup_{k,l}i_{k,l}(E^{(k,l)})\subseteq Z$$ then $Z\setminus E$ has Property $(\star)$. 
 \end{proof}

 \section{Intersection theory in complete fields}\label{computations}

 In this section we present the main ingredients of intersection theory that we will use in order to find a field configuration for $\mathcal G$. Throughout this section we assume that $\mathbb K$ is complete and we fix $q\in K$ an arbitrary point.

 We start with a definition.

\begin{defi}

Given $X\subseteq K\times K$ $\mathbb K$-definable, $b=(b_1,b_2)\in X$ and $U\ni b_1$ an open subset of $K$, we say that an analytic function $h:U\to K$ is a \emph{Fr-power expansion of $X$ at $b$} if there is a natural number $n$ such that $b_2=\text{Fr}^{-n}( h(b_1))$ and $(s,\text{Fr}^{-n}( h(s)))\in X$ for all $s\in U$.

If $n=0$ we say that $h$ is a \emph{power expansion of $X$ at $b$}.

\end{defi}

From now on in this section we fix $Y\subseteq K\times K$ a one-dimensional $\mathbb K$ definable set containing $(q,q)$, $B\subseteq K$ an open neighborhood of $q$ and $h:B\to K$, a power expansion of $Y$ at $(q,q)$. For $s\in B$ we define $\bar s:=(s,h(s))$, an element of $Y.$

 \begin{defi}

  Let $(X_a)_{a\in Y}$ be a $\mathbb K$-definable family of $1$ dimensional subsets of $K\times K$ indexed by $Y$ . In this case we define:
  
Given $\alpha\in B$, $b=(b_1,b_2)\in X_{\bar\alpha}$,  $U\ni b_1$ and $V\ni \alpha$ open sets with $V\subseteq B$. We say that an analytic function $\Phi:U\times V\to K$ \emph{is a uniform Fr-power expansion of $(X_a)_{a\in Y}$ for $X_{\bar\alpha}$ at  $b$} if there is a natural number $n$ such that $b_2=\text{Fr}^{-n}(\Phi(b_1,\alpha))$ and for all $x\in U$ and all $s\in V$, $(x,\text{Fr}^{-n}(\Phi(x,s)))\in X_{\bar s}$.

 If moreover we assume that for all $a\in Y$, $(q,q)\in X_a$ we say that  $(X_a)_{a\in Y}$ is \emph{uniformly analytic at $(q,q)$} if there is $U\ni q$ an open subset of $B$, and an analytic function $H:U\times U\to K$ such that for all $s,x\in U$, $H(q,s)=q$ and $(x,H(x,s))\in X_{\bar s}$. In this case we say that $H$ is  a \emph{uniform power expansion of $(X_a)_{a\in Y}$ at $(q,q)$}.

\end{defi} 

First we prove:

\begin{lemma}\label{lemmaCanCompose}

Let $(X_a)_{a\in Y}$ be a $\mathbb K$-definable family of $1$-dimensional subsets of $K\times K$ indexed by $Y$ such that $(X_a)_{a\in Y}$ is uniformly analytic at $(q,q)$. Let $U\ni q$ be an open subset of $B$ and $H:U\times U\to K $ a uniform power expansion of $(X_a)_{a\in Y}$ at $(q,q)$. Then, for each $\beta \in U$, the family 

$$(X_{\bar\beta}\circ X_a)_{a\in Y}$$ is uniformly analytic at $(q,q)$.
 
Moreover, assume that $\alpha\in U$ and $b=(b_1,c)\in X_{\bar\alpha}$ are such that there exists a uniform Fr-power expansion of $(X_a)_{a\in Y}$ for $X_{\bar \alpha}$ at $b$. Let $b_2\in K$ be such that $b'=(c,b_2)\in X_{\bar \beta}$ and there is an Fr-power expansion of $X_{\bar \beta}$ at $b'$. Then there is a uniform Fr-power expansion of the family $(X_{\bar\beta}\circ X_a)_{a\in Y}$ for $X_{\bar\beta}\circ X_{\bar \alpha}$ at $(b_1,b_2)$.

\end{lemma}

\begin{proof}
   For the first part, by continuity as $H(q,s)=q$ for all $s\in U$, one can find an open set $U'$ such that the function $H_\beta:U'\times U'\to K$ defined by: 

    $$H_\beta(x,s):=H(H(x,s),\beta)$$ satisfies the conclusion of the lemma.

    %For the last part let $a\in U$ and let $(b_1,b_2)\in X_{\bar \beta} \circ X_{\bar a}$, then there is some $c$ such that $(b_1,c)\in X_{\bar a}$ and $(c,b_2)\in X_{\bar\beta}$.
    
    %By Clause 5 applied to the family $(X_a)_{a\in Q}$ there is 
    For the last part let $\Phi_1(x,s)$ be an analytic function converging in some neighborhood $U_1\times V_{1}$ of $(b_1,\alpha)$ and let $n_1$ be a natural number such that for all $(x,s)\in U_1\times V_1$

    $$\left(x,\text{Fr}^{-n_1}(\Phi_1(x,s))\right)\in X_{\bar s}.$$
    
   Let $\Phi_2(x):U_2\to K$ be an analytic function converging in a neighborhood $U_{2}$ of $c$, and let $n_2$ be a natural number such that for all $x\in U_2$
    
    $$\left(x,\text{Fr}^{-n_2}(\Phi_2(x)\right)\in X_{\bar\beta}.$$

    Let $$\Psi_1(x,s):=\text{Fr}^{-n_1}(\Phi_1(x,s))$$ 
    and  $$\Psi_2(x):=\text{Fr}^{-n_2}(\Phi_2(x)).$$

    Therefore the graph of 
    $$x\mapsto \Psi_2(\Psi_1(x,s)) $$ 
    is contained in $X_{\bar\beta}\circ X_{\bar s}$ for all $(x,s)$ in some neighborhood of $(b_1,\alpha)$. Now note that 
    $$\Phi_2(\Psi_1(x,s))=\Phi_2(\text{Fr}^{-n_1}(\Phi_1(x,s))=\text{Fr}^{-n_1}(\tilde \Phi_2(\Phi_1(x,s))). $$

 Where $\tilde \Phi_2(x)$ is an analytic function such that for all $x$, 
 $$\text{Fr}^{-n_1}(\tilde \Phi_2(x))=\Phi_2(\text{Fr}^{-n_1}(x)).$$

This $\tilde \Phi_2(x)$ is obtained by changing all the coefficients of $\Phi_2(x)$ by their $p^{n_1}$-powers, more precisely, if $\Phi_2(x)$ has power expansion given by 
 $$\Phi_2(x)=\sum_i a_i x^{i} $$ then 
 $$\tilde\Phi_2(x)=\sum_I a_i ^{p^{n_1}} x^{i}.$$
 
 Notice that if $\Phi(x)$ converges at $a$ then $\tilde\Phi(x)$ converges at $a^{p_{n_1}}$ so $\tilde\Phi(x)$ is a convergent series at some non-empty open set.

 Therefore, 

 $$\Psi_2(\Psi_1(x,s))=\text{Fr}^{-n_2-n_1}(\tilde\Phi_2(\Phi_1(x,s))$$
 
 so we take $n=n_1+n_2$ and 
 $$\Phi(x,s)=\tilde\Phi_2 (\Phi_1(x,s))$$ in order to get a uniform Fr-power expansion of $(X_{\bar\beta}\circ X_a)_{a\in Y}$ for $X_{\bar\beta}\circ X_{\bar\alpha}$ at $(b_1,b_2)$.  \end{proof}

Now we need some definitions:

\begin{defi}\label{defidn}
Let $U\ni q$ open, $H:U\times U\to K$ analytic and assume that $H(q,s)=q$ for all $s\in U$. Then, $d_n(H,s)$ is the $n$-th coefficient of the power expansion of $x\mapsto H(x,s)$ around $q$. So $$H(x,s)= q+\sum_{n\geq 1} d_n(H,s) (x-q)^n.$$

Let $N(H)$ be the minimal natural number $n$ (if it exists) such that $$s_n(H):=\{d_n(H,s):s\in U\}$$ is infinite.

If $T:U\to K$ is a function in just one variable such that $T(q)=q$ we view it as a two variable function setting $T(x,s)=T(x)$. In this sense we define $d_n(T)$.
\end{defi}

      \begin{lemma}\label{propIntCreceGeneralVersionOrigen}
   Let $(X_a)_{a\in Y}$ be a $\mathbb K$-definable family of one dimensional subsets of $K\times K$ indexed by $Y$. Let $H:U\times U\to K$ be a uniform power expansion of $(X_a)_{a\in Y}$ at $(q,q)$ and let $N:=N(H)$.
   
   Let $Z\subseteq K\times K$ be a one dimensional $\mathbb K$-definable set with $(q,q)\in Z$ and let $T:U\to K$ be a power expansion of $Z$ at $(q,q)$. 
   Assume that  $\alpha\in U$ is such that
   $$d_n(H,\alpha)=d_n(T)$$ for all $n\leq N$. 
  % and assume that
   %$$s_1(H):=\left\{\frac{\partial H}{\partial x}(q,s):s\in U\right\}$$ is infinite.
   Then, for any open $U'\times V'\subseteq K\times K$ containing $(q,q)$, there is $W\ni \alpha$, an open subset of $U$, such that for all $s\in W\setminus \{\alpha\}$

    $$|X_{\bar s} \cap Z \cap (U'\times V')|\geq 2.$$

 %   \item If $\alpha\in U$ and $(x',y')\in X_{\bar \alpha}\cap Z$ then, for any $U'\times V'\subseteq K\times K$, open neighborhood of $(x',y')$, there is an open set $W\ni \alpha$ contained on $U$ such that for all $s\in W\setminus\{\alpha\}$

  %  $$|X_{\bar s} \cap Z \cap (U'\times V')|\geq 1.$$

   % \item If $\alpha\in U$ is such that $d_N(\alpha)=e_N$, there is $W\ni \alpha $ an open set contained on $U$ such that for all $s\in W\setminus \{\alpha\}$

    %$$|X_{\bar s} \cap Z| > |X_{\bar \alpha} \cap Z|. $$
%\end{enumerate}
   
\end{lemma}

\begin{proof}

By shrinking $U'$ if necessary, we may assume that $U'\subseteq U$ and $T(x)\in V'$ for all $x\in U'$.

Let $F$ be the analytic function defined by:

$$F(x,s)=T(x)-H(x,s).$$

Our assumptions imply that $x=q$ is a zero of $F(x,\alpha)$ of order at least $N+1$. By Fact \ref{continuityOfRootsTheoremf} applied to the function $F(x,s)$ together with the open sets $U'$ and $U$, there is an open neighborhood $W\subseteq U$ of $\alpha$ such that for all $s\in W$ the function $F(x,s)$ has at least $N+1$ zeros (counting multiplicities) at $U'$. Moreover, using Lemma \ref{puedoRestringirderivada} with the function $L(s):=d_N(F, s)$, we may assume -possibly after shrinking  $W$- that $d_N(F,s)\neq 0$ for all $s\in W\setminus\{\alpha\}$. 
%As for $n<N$
%$\{d_n(H,s):s\in W\}$ is finite, shrinking $W$ we may assume that $d_n(H,s)=d_n(H,q)$ and so that $d_n(F,s)=0$ for all $s\in W$.

Therefore, for all $s\in W\setminus\{\alpha\}$ the function $f_s(x):=F(x,s)$ has only zeros of order at most $N$. Thus, for all $s\in W\setminus \{\alpha\}$ there are at least two distinct points $z_1,z_2\in U'$ such that

$$F(z_i,s)=0,$$ but this implies that $$(z_i,H(z_i,s))=(z_i,T(z_i))\in X_{\bar s} \cap Z \cap (U'\times V')$$ for $i=1,2.$

So for all $s\in W\setminus \{\alpha\}$ one has that:

\begin{equation*}
|X_{\bar s} \cap Z  \cap (U'\times V')|\geq 2.
 \end{equation*}
\end{proof}

\begin{lemma}\label{propIntCreceGeneralVersionLejos}
  
   Let $(X_a)_{a\in Y}$ be a $\mathbb K$-definable family of one dimensional subsets of $K\times K$ indexed by $Y$. Let $Z\subseteq K\times K$ be a $\mathbb K$-definable set of dimension $1$ with no isolated points.
   
Let $\beta\in B$ and $(x',y')\in X_{\bar \beta}\cap Z$. Assume that $\Phi:U\times V\to K$ is a uniform Fr-power expansion of $(X_a)_{a\in Y}$ for $X_\beta$ at $(x',y')$. Then, for any $U'\times V'$ open neighborhood of $(x',y')$, there is an open set $W\ni \beta$ contained on $V$ such that 

    $$|X_{\bar s} \cap Z \cap (U'\times V')|\geq 1$$
    
for all $s\in W$.
   % \item If $\alpha\in U$ is such that $d_N(\alpha)=e_N$, there is $W\ni \alpha $ an open set contained on $U$ such that for all $s\in W\setminus \{\alpha\}$

    %$$|X_{\bar s} \cap Z| > |X_{\bar \alpha} \cap Z|. $$
%\end{enumerate}
   
\end{lemma}

\begin{proof}
%We start proving Clause 1 of the conclusion, so let $\alpha\in U$ be such that $d_N(\alpha)=e_N$.

%Let $F$ be the analytic function defined by:

%$$F(x,s)=T(x)-H(x,s).$$

%Then we have that 

%Our assumptions imply that $x=0$ is (at least) a zero of order $N$ for the function $x\mapsto F(x,\alpha)$. By Fact \ref{continuityOfRootsTheoremf} applied to the function $F(x,s)$ together with the open sets $U'$ and $U$, there is an open neighborhood $W\subseteq U$ of $\alpha$ such that for all $s\in W$ the function $x\mapsto F(x,s)$ has at least $N$ zeros (counting multiplicities) at $U'$. Moreover, using Lemma \ref{puedoRestringirderivada} for the function $s\mapsto d_N(s)$,  by shrinking $W$, we may assume that $d_N(s)\neq 0$ for all $s\in W\setminus\{\alpha\}$. Therefore, the function $x\mapsto F(x,s)$ has no zeros of order $N$ at  $W\setminus\{\alpha\}$. Thus, for all $s\in W_1\setminus \{\alpha\}$ there are exactly two points $z_1,z_2\in U'$ such that:

%$$F(z_i,s)=0$$ but this implies that $$(z_i,h(z_i))\in X_{\bar s} \cap Z \cap (U'\times V')$$ for $i=1,2.$

%So for all $s\in W\setminus \{\alpha\}$ one has that:

%\begin{equation*}
%|X_{\bar s} \cap Z  \cap (U'\times V')|\geq 2.
 %\end{equation*}

 %This finishes the proof of Clause 1. of the conclusion. Let us prove Clause 2.  

As $Z$ has no isolated points and $(x',y')\in Z$, by Proposition \ref{decompositionProp}, there is some irreducible polynomial $L(x,y)$ and some open set $A\subseteq K\times K$ such that 
$$(x',y')\in \{(x,y)\in A:L(x,y)=0\} \subseteq Z.$$

Let $n_0$ be a natural number such that 
$$\left(x,\text{Fr}^{-n_0}(\Phi(x,s))\right)\in X_{\bar s}$$ 
for all $(x,s)$ in some neighborhood of $(x',\alpha)$. 

Define $\Psi(x,s)=\text{Fr}^{-n_0}(\Phi(x,s))$.

By shrinking $U'$ and $V$ we may assume $\Psi(x,s)\in V'$  for all $x\in U'$ and $s\in V$. We may also assume that $(x,\Psi(x,s))\in A$. Therefore, by our conditions on $L$ and $A$, we only need to prove that for all $s$ close enough to $\alpha$ there is some $x\in U'$ such that
$$L(x,\Psi(x,s))=0.$$

So we prove that for $s$ close enough to $\alpha$ the function 
$$L(x,\text{Fr}^{-n_0}(\Phi(x,s)))$$ 
has a zero in $U'$.
Let  $\tilde L$ be the analytic function defined by
$$\tilde L(x,y)=\sum_I a_{I}^{p^{n_0}} x^{i_1 p^{n_0}} y^{i_2}$$
where $$L(x,y)=\sum_I a_{I} x^{i_1} y^{i_2}.$$ 

So  
$$L\left(x,\text{Fr}^{-n_0}(\Phi(x,s))\right)=\text{Fr}^{-n_0}\left(\tilde L(x,\Phi(x,s))\right).$$

Let $$F(x,s):=\tilde L(x, \Phi(x,s)).$$ 

This is an analytic function vanishing at $(x',\alpha)$. Applying Fact \ref{continuityOfRootsTheoremf} to $F$ and the open sets $U'$ and $V$ we get $W\ni \alpha$, an open subset of $V$, such that for all $s\in W$ there is some $x\in U'$ with $\tilde L(x,\Phi(x,s))=0$ and $$L\left(x,\text{Fr}^{-n_0}( \Phi(x,s))\right)=\text{Fr}^{-n_0}\left(\tilde L(x, \Phi(x,s))\right)=\text{Fr}^{-n_0}(0)=0.$$

%So one has that $F_i(x_i,\eta')=0$ then we can apply again Fact \ref{continuityOfRootsTheoremf} to the function $F_i$ with the open sets $U_i$ and $U$. The conclusion of the Fact provide us with an open set $W_i\subseteq U$ containing $\eta'$ such that f

So for all $s\in W$ there is at least one point $x\in U'$ such that 
$$(x,\text{Fr}^{-n_0}(\Phi(x))\in (X_s \cap Z ) \cap (U'\times V'),$$ as required. \end{proof}

As a consequence of those two lemmas we have:

\begin{lemma}\label{intersectionSubeLemmaAnalitic}

Let $(X_a)_{a\in Y}$ be a $\mathbb K$-definable family of one dimensional subsets of $K\times K$ and let $H:U\times U\to K$ be a uniform power expansion of $(X_a)_{a\in Y}$ at $(q,q)$ and set $N:=N(H)$.

 Assume that for all $s\in U$ and for all $b\in X_{\bar s}$ the family $(X_a)_{a\in Y}$ admits a uniform Fr-power expansion for $X_{\bar s}$ at $b$.

Let $Z\subseteq K\times K$ be a $\mathbb K$-definable $1$-dimensional set with no isolated points and let $T:U\to K$ be a power expansion of $Z$ at $(q,q)$.

Then, if $\alpha\in U$ is such that $X_{\bar \alpha}\cap Z$ is finite and  $$d_n(H,\alpha)=d_n(T)$$ for all $n\leq N$, there is $W\ni \alpha $ an open set contained in $U$ such that 

    $$|X_{\bar s} \cap Z| > |X_{\bar\alpha} \cap Z|$$

   for all $s\in W\setminus \{\alpha\}$. In particular, there are infinitely many $a\in Y$ such that 
   
    $$|X_{a} \cap Z| > |X_{\bar \alpha} \cap Z|.$$
    
\end{lemma}

\begin{proof}

Write $$X_{\bar\alpha} \cap Z=\left\{(x_1,y_1)=(q,q),(x_2,y_2),\ldots,(x_n,y_n)\right\}.$$

For $i=1,\ldots,n$ let $U_i$ and $V_i$ be open neighborhoods of $x_i$ and $y_i$ respectively such that 
$$(U_j\times V_j)\cap (U_k\times V_k)=\emptyset$$
for $j\neq k$.

 Apply Lemma \ref{propIntCreceGeneralVersionOrigen} with $U'=U_1$ and $V'=V_1$ and find an open neighborhood $W_1$ of $\alpha$ such that for all $s\in W_1\setminus \{\alpha\}$,

 \begin{equation}\label{haydosGeneral}
|X_{\bar s} \cap Z  \cap (U_1\times V_1)|\geq 2.
 \end{equation}

 For $j=2,\ldots, n$, as the family $(X_a)_{a\in Y}$ admits an uniform Fr-power expansion for $X_{\bar\alpha}$ at $(x_1,y_1)$ we apply Lemma \ref{propIntCreceGeneralVersionLejos} with $U'=U_j$, $V'=V_j$ and $(x',y')=(x_j,y_j)$ and find $W_j$, an open neighborhood of $\alpha$, such that for all $s\in W_j$

 \begin{equation}\label{hayunoGeneral}
|X_{\bar s} \cap Z  \cap (U_j\times V_j)|\geq 1.
 \end{equation}

As 
$$(U_j\times V_j)\cap (U_k\times V_k)=\emptyset$$
for $j\neq k$, equations \ref{haydosGeneral} and \ref{hayunoGeneral} imply that if we define 

$$W=\bigcap_{j=1}^n W_j$$

then, for all $s\in W\setminus \{\alpha\}$ 

$$|X_{\bar s}\cap Z|\geq n+1 >  |X_{\bar\alpha}\cap Z|.$$ \end{proof}

\chapter{The Additive Case}\label{additiveCaseVersion}

%We start with a definition. 

%\begin{defi}
 %   Let $(G,\oplus)$ be a $\mathbb K^{eq}$ definable group. We say that $(G,\oplus)$ is locally isomorphic to the additive group $\mathbb G_a$ if $G$ is Abelian and there is $(A,+)$ a $\mathbb K$-definable subgroup of the additive group $(K,+)$ and a $\mathbb K^{eq}$-definable injective group homomorphism  $i:A\to G$ such that $G/i(A)$ is finite. 
%\end{defi}
In this chapter we prove Zilber's conjecture in the case where the strongly minimal structure is an expansion of a group locally isomorphic to $\mathbb G_a$ whose definable sets are $\mathbb K$-definable, more precisely:

\begin{thm}\label{thmAditiveVersion}
    Let $\mathbb K=(K,+,\cdot,0,1,v,\Gamma)$ be an algebraically closed  valued field with $\text{char} (K)=p\geq 0$. Let $(G,\oplus)$ be a $\mathbb K$-definable group and assume that it is locally isomorphic to $\mathbb G_a=(K,+)$. Let $\mathcal G=(G,\+,\ldots)$ be a first order structure expanding $(G,\+)$ whose definable sets are all $\mathbb K$-definable. Then, if $\mathcal G$ is strongly minimal and non locally modular it interprets an infinite field %that is $\mathbb K$-definable isomorphic to $\mathbb K^f:=(K,+,\cdot,0,1)$. 
 \end{thm}

Recall (Definition \ref{definitionLocllyIsomorphic}) that for us, an infinite $\mathbb K$-definable group $G$ is  locally isomorphic to $\mathbb G_a$ if there is $A\leq\mathbb G_a$ a $\mathbb K$-definable subgroup and $i:A\to G$ a $\mathbb K$-definable injective group homomorphism such that $G/i(A)$ is finite. 
%By the following fact, if there is an infinite field interpretable in $\mathcal G$, it is $\mathbb K$-definable isomorphic to $(K,+,\cdot)$ so we devote this chapter to construct an infinite field interpretable in $\mathcal G$.

%\begin{fact}\label{uniqueFieldFact}(Theorem ?? of \ref{})
%\textbf{Which is the Reference needed here?}
    
%\end{fact}

From now on we fix $\mathcal G=(G,\+,\ldots)$ as in the hypothesis of Theorem \ref{thmAditiveVersion}, let $\ominus$ be the inverse operation on $G$. We also fix $A$ and $i:A\to G$ as in the definition of locally isomorphic to $\mathbb G_a$. If $Z$ is any $\mathcal G$-definable set, we compute the Morley rank and Morley degree of $Z$ in the sense of $\mathcal G$. We say that $Z$ is strongly minimal if it is strongly minimal in the sense of $\mathcal G$.

For notational reasons we will assume that $i:A\to G$ is indeed an inclusion, so $A\leq G$ is a subgroup of finite index. We fix $e=g_0,g_1,\ldots,g_m\in G$ such that for all $g\in G$ there is a unique $j\in\{0,\ldots,m\}$ and $a\in A$ such that $g=g_j\oplus a$. For $j=0,\ldots,m$ let $$G_j:=g_j\+ A=\{g_j\oplus a:a\in A\} $$ so $G$  is the disjoint union of $\{G_j\}_{j=0}^m$.  For $k,l\leq m$ let $i_{k,l}:A\times A\to G_k\times G_l$ defined by $i_{k,l}(x,y)=(g_k\oplus x,g_l\oplus y)$.

By Fact \ref{existsXVersion} (and the subsequent commentary) there is $X\subseteq G\times G$ a $\mathcal G$-definable strongly minimal set that is not $\mathcal G$-affine. From now on we fix such an $X$.

%\begin{defi}
 %   For $a=(a_1,a_2)\in G\times G$ and $Y\subseteq G\times G$ we define the translation of $Y$ by $a$ as:

  %  $$t_a(Y):=\{(x_1\ominus a_1,x_2\ominus a_2):(x_1,x_2)\in Y\}.$$

   % Notice that if $Y$ is $\mathcal G$-definable, then so is $t_a(Y)$.
%\end{defi}

%In order to prove Theorem \ref{thmAditiveVersion} we want to work in the setting in which $\mathbb K$ is complete. In order to do so we provide absolute statements, so we can prove it just for the metric case and then they will follow for the general case. This is the motivation of Definitions \ref{versionLemmaBuenYdef} and \ref{defTangent}.

\section{A notion of tangency}\label{aNotionOfTangencyAdditive}

In order to interpret a field we find a group configuration in $\mathcal G$ that is interalgebraic (in $\mathbb K$) with a standard group configuration for the group $\mathbb G_a\ltimes\mathbb G_m$. Here $\mathbb G_a$ and $\mathbb G_m$ are the additive and multiplicative group of $\mathbb K$ respectively and the semidirect product is given by the action of $\mathbb G_m$ on $\mathbb G_a$ by multiplication.

More precisely, our group configuration will be interalgebraic in $\mathbb K$ with a rank $2$ group configuration of the form:

\begin{equation}\label{gcFullStructureDiagram}
 	\begin{tikzcd}
 	(a_0, a_1) \arrow[ddd, dash] \arrow[rrr, dash] &&& b \arrow[rrr, dash]  &&& a_0+a_1 b \arrow[dddllllll,  dash] 
  \\
&&&&&&
\\
& & && %\arrow[dll, dash]  
%\arrow[ddddll, dash]  
&&
\\
 (b_0,b_1)  \arrow[ddd, dash] &&& b_0+ a_0b_1+ a_1 b_1 b &&& \\ &&&&&&
 \\
 &&&&&&
 \\
 (a_0b_1+b_0,a_1b_1)  \arrow[uuuuuurrr, crossing over, dash] &&&&&&
  	\end{tikzcd}
\end{equation}

where $(a_0,a_1)$ and $(b_0,b_1)$ are $\mathbb K$-generic independent elements of $\mathbb G_a\ltimes \mathbb G_b$.

%It is worth noting that we refer to the elements of $\mathbb G_a\ltimes \mathbb G_b$ as ordered pairs of the form $(c_0,c_1)$, where $c_0$ belongs to $\mathbb G_a$ and $c_1$ belongs to $\mathbb G_b$. 

In this section we describe the relation that will lead to the $\mathbb K$-interalgebricity between this group configuration and the one that we find for $\mathcal G$. Although this is not the shortest way for showing the interalgebricity, we decided to present this approach since it gives a flavor of what was done in Section \ref{computations} but in a simpler setting. However, in the current section we do not discuss how to get a group configuration for $\mathcal G$, this will be done in Section \ref{theProof}.

\begin{defi}\label{defTangent}
For any $a\in K$ let 

$$l_a:=\{(x,ax):x\in K\}.$$

Given $Z\subseteq A\times A$ a one dimensional $\mathbb K$-definable set, we denote $Z \tangent l_a$ if $(0,0)\in Z$ and either $Z\cap l_a$ is infinite or there is an open ball $B\ni a$ such that for all $q\in B$ if $q\neq a$ then, 

$$|l_q\cap Z|>|l_a\cap Z|.$$

\end{defi}

This definition has two main properties:

Lemma \ref{claimTangentFiniteLemma} that shows that if $Z\tangent l_a$ then $a$ is $\mathbb K$-algebraic over the parameter defining $Z$. 

Lemma \ref{claimDerivativeEqualsImpliesTangentLemma} shows that in the complete metric case if $Z$ contains the graph of an analytic function $g$, then $Z\tangent l_{g'(0)}$. 

These two facts together imply that if $Z\ni (0,0)$ contains the graph of a function, then the derivative of the function at $0$ is algebraic over the parameter defining $Z$. %Although this could be proved with a standard $\epsilon-\delta$ argument, we decided to give this approach because it has a simpler formulation. Besides, 
%The proof of Lemma \ref{claimDerivativeEqualsImpliesTangentLemma} shows some of the main arguments that we will use in Section \ref{computations} for proving that the intersection number drops when the curves are tangent, which is the main ingredient we will use for interpreting a field. 

For $Z\subseteq G\times G$ we extend the definition as follows:

\begin{defi}
    Given $Z\subseteq G\times G$ one dimensional and $\mathbb K$-definable, we denote $Z\tangent l_a$ if $(e,e)\in Z$ and $Z\cap (A\times A)\tangent l_a$.
\end{defi}

\begin{lemma} \label{claimTangentFiniteLemma}
    Let $Z\subseteq A\times A$ be a $1$-dimensional $\mathbb K$-definable set such that $(0,0)\in Z$. Then $$\{a\in K:Z\tangent l_a\}$$ is finite. 
\end{lemma}

\begin{proof}
    Assume it is infinite, then, by Corollary \ref{existsInf}, there is an open ball $B\subseteq K$ such that for all $a\in B$, $Z\tangent l_a$. Let $a\in B$ be such that $|Z\cap l_a|\leq |Z\cap l_s|$  for all $s\in B$. By definition of $Z\tangent l_a$, there is an open ball $D$ centered on $a$ such that for all $q\in D\setminus \{a\}$

   $$|l_q\cap Z|>|l_a\cap Z|,$$ but as $B$ and $D$ are open sets containing $a$ there is some $q\in (B\cap D)\setminus a$. But this is a contradiction because by our choice of $a$, $|Z\cap l_a|\leq |Z\cap l_q|$. 
\end{proof}

\begin{lemma}\label{claimDerivativeEqualsImpliesTangentLemma}
 
Let $Z\subseteq A\times A$ be a $1$-dimensional $\mathbb K$-definable set with no isolated points and assume that $\mathbb K$ is complete. Suppose that $U\subseteq A$ is an open neighborhood of $0$ and $g:U\to A$ is an analytic function converging on $U$ such that $g(0)=0$ and $(x,g(x))\in Z$ for all $x\in U$. Then $Z\tangent l_{g'(0)}$.

\end{lemma}

\begin{comment}
\begin{proof}
Notice that the family $(l_a)_{a\in K}$ satisfies the hypothesis of Proposition \ref{propIntCreceGeneralVersion} taking $U=K$, $H(x,s)=xs$ and $i$ as the identity map. Moreover, for all $q\in U$ and all $(b_1,b_2)\in l_q$ we can take $n=0$ and $\Phi(x,s)=xs$ in order to get Clause 5. 

Also $Z$ satisfies the hypothesis of Proposition with $T(x)=g(x)$ and $a=\alpha$.

Therefore, by Clause 3. of the conclusion of Proposition \ref{propIntCreceGeneralVersion} applied to $(l_q)_{q\in K}$ and $Z$, there is a neighborhood $W\ni a$ such that for all $s\in W\setminus \{a\}$,

$$|l_s\cap Z|>|l_a\cap Z|$$ so by definition $l_a\tangent Z$.

\end{proof}
\end{comment}

\begin{proof}
It follows from \ref{intersectionSubeLemmaAnalitic} applied to the family $(l_a)_{a\in U}$. However, in this case the proof is easier and still has the main arguments on intersection theory that we used so we present it as an example of the results on Section \ref{computations}.

If $Z\cap l_{g'(0)}$ is infinite we are done so assume that it is finite and list $Z\cap l_{g'(0)}=\{(0,0)=(x_1,y_1),(x_2,y_2),\ldots,(x_n,y_n)\}$. For $i=1,\ldots,n$ let $U_i$ and $V_i$ be open subsets of $K$ contained on $A$ such that $(x_i,y_i)\in U_i\times V_i$ and $(U_i\times V_i)\cap (U_j\times V_j)=\emptyset$ for $i\neq j$.  

After shrinking $U$, $U_1$ and $V_1$, if necessary, we may assume that $U\subseteq U_1$ and $g(x)\in V_1$ for all $x\in U_1$. 

Consider the analytic function:
$$F(x,q)=g(x)-qx.$$

Let $a:=g'(0)$, then $$F(0,a)=\frac{\partial F}{\partial x}(0,a)=0,$$ so, by Lemmma \ref{derivativeZeroIsDoubleLemma}, $0$ is (at least) a double zero for the function 
$$x\mapsto F(x,a).$$

Thus, we can apply Fact \ref{continuityOfRootsTheoremf} and find an open neighborhood $W_1$ of $a$ such that for all $q\in W_1$ the function
$$x\mapsto F(x,q)$$ has at least two zeros (counting multiplicities) in $U_1$. Moreover, by Lemma \ref{puedoRestringirderivada}, shrinking $W_1$ we may assume that for all $q\in W_1\setminus \{a\}$, $0$ is a simple zero of the function
$$x\mapsto F(x,q).$$

Therefore, for all $q\in W_1\setminus \{a\}$ there are at least two different points $z_1$ and $z_2$ in $U_1$ such that $F(z_1,q)=F(z_2,q)=0$. So $(z_1,g(z_1))$ and $(z_2,g(z_2))$ are different points of $Z\cap l_q$ that lie inside $U_1\times V_1$.

Now for $i>1$ as $(x_i,y_i)$ is not an isolated point of $Z$, using Proposition \ref{decompositionProp} there is some polynomial $L(x,y)$  such that (shrinking $U_i$ and $V_i$, if needed) 
$$\{(x,y)\in U_i\times V_i:L(x,y)=0\}\subseteq Z.$$

Consider the analytic function $F_i$ defined by:
$$F_i(x,q)=L(x,qx).$$

Note that $F_i(x_i,a)=0$, so we can use Fact \ref{continuityOfRootsTheoremf} and find an open neighborhood $W_i$ of $a$ such that for all $q\in  W_i$ there is at least one $z\in U_i$ such that $F_i(z,q)=0$. Moreover, by shrinking $V_i$ we may assume that for all $q\in W_i$ and for all $x\in U_i$, $qx\in V_i$.
Therefore $(z,qz)\in Z\cap l_q \cap (U_i\times V_i)$.

If we define $$W:=\bigcap_i W_i,$$  we have that for all $q\in W\setminus \{a\}$ and for all $i>1$, $$|Z\cap l_q \cap (U_1\times V_1)|\geq 2$$ and  $$|Z\cap l_q \cap (U_i\times V_i)|\geq 1.$$ 

Thus, $|Z\cap l_q|>|Z\cap l_a|$ for all $q\in W\setminus \{a\}$ so by definition $Z\tangent l_a$.
\end{proof}

 \section{Proof of Theorem \ref{thmAditiveVersion}}\label{theProof}

In this section we prove Theorem \ref{thmAditiveVersion}. In order to do so, we first prove Lemma \ref{existsHLemma} that, in the complete case, is the main tool we need to build a field. Then, in order to prove the general case, we present the statement and proof of Lemma \ref{lemmaIntersectionSubeVersion}. This is the main technical tool that we will use to build a field configuration. Finally, we present the proof of Theorem \ref{thmAditiveVersion}.

%We call $$Y^*:=(i \times i)^{-1}(Y)\subseteq A\times A$$ and for $a,c\in Y^*$ let 

%\begin{equation}\label{defX*}
%X_{a,c}^*:=(Y^* - t_a(Y^*))\circ (Y^* -t_c(Y^*))^{-1}.
%\end{equation}

%Notice that in Equation \ref{defX} the translate $t_a(Y)$ is taken using to the group operation of $G$ while in Equation \ref{defX*} the translate $t_a(Y^*)$ is taken using addition (the group operation on $A$).
%Remember (Definition \ref{GaffineDefi}) that we define:

We will use a $\mathcal G$-definable set $Y\subseteq G\times G$ which is strongly minimal and non-$\mathcal G$ affine. By Fact \ref{nonAffineIntersectsFiniteLemmaF} if $Y$ is $\mathcal G$-definable it is not $\mathcal G$ affine if and only if for all but finitely many $a\in Y$ there are just cofinitely many $b\in Y$ such that $t_a(Y)\cap t_b(Y)$ is infinite. Now we  extend the notion for any $H$, $\mathbb K$-definable group, and any $Y\subseteq H\times H$ :

\begin{defi}\label{infiniteTranslatesDef} 
    Let $H$ be a $\mathbb K$ definable group. If $Y\subseteq H\times H$ is $\mathbb K$-definable, we say that $Y$ is \emph{not $H$-affine}  if for all but finitely many $a\in Y$ there are cofinitely many $b\in Y$ such that $t_a(Y)\cap t_b(Y)$ is infinite.
\end{defi}

Notice that if $\phi(\bar x,\bar y,\bar z)$ is a first order formula in the language of fields and 
$$Y(\bar d)=\{(\bar x,\bar y)\in K\times K:\mathbb K\models \phi(\bar x,\bar y,\bar d)\},$$ then by Corollary \ref{existsInf} 
\begin{quote}
   $$ \{\bar d:Y(\bar d)\text{ is not }H\text{-affine}\}$$
\end{quote}
is $\mathbb K$-definable.

\begin{lemma}\label{existsHLemma}

       Assume that $\mathbb K$ is complete. Let $Y^u\subseteq G\times G$ be $\mathbb K$-definable and not $G$-affine such that $Y:=Y^u\cap( A\times A)$ has Property $(\star)$ witnessed by some polynomials $L_i(x,y)$ and $L_{i,n_i}(x,y)$. Assume as well that $k,l$ are such that $i_{k,l}^{-1} (Y^u)$ satisfies Clauses \ref{clause1}, \ref{clause2} and \ref{clause4} of Property $(\star)$ witnessed by some polynomials $L_i^{(k,l)}(x,y)$.

      For $a,c\in Y^u$ set:

      $$X_{a,c}^{u}:=(Y^u-t_a(Y^u))\circ (Y^u-t_c(Y^u))^{-1}$$

       And for $a,c\in Y$:
      $$X_{a,c}:=X_{a,c}^{u}\cap (A\times A)=(Y-t_a(Y))\circ (Y-t_c(Y))^{-1}.$$
       
       Then, there is some open set $U\subseteq A\subseteq  K$ containing $0$, an element $c\in U$, $h:U\to A$ and $H:U\times U\to A$ analytic functions such that:

\begin{enumerate}

    \item $H(0,s)=0$ for all $s\in U$.

    \item For all $s\in U$, $\bar s:=(s,h(s))\in Y$.

     \item For all $s\in U$ there is $U_s\ni 0$ an open subset of $A$ such that $(x,H(x,s))\in X_{\bar s, \bar c}$ for all $x\in U_s$.

     \item The set $$s_1(H):=\left\{\frac{\partial H}{\partial x}(0,s):s\in U\right\}$$ is infinite.

     \item For each $a\in U$ and $b=(b_1,b_2)\in X_{\bar a,\bar c}$ there are neighborhoods $U_b$ of $b_1$ and $V_a$ of $a$, an analytic function $\Phi(x,s)$ defined in $U_b\times V_a$ and a natural number $n$ such that $(x,\text{Fr}^{-n}(\Phi(x,s)))\in X_{\bar s,\bar c}$ for all $(x,s)\in U_b\times V_a$.

     \item\label{frLejosGLemma} For each $a\in U$ and $b=(b_1,b_2)\in i_{k,l}^{-1} (X^{u}_{i(\bar a ),i( \bar c )})$ there are neighborhoods $U_b$ of $b_1$ and $V_a$ of $a$, an analytic function $\Phi(x,s)$ defined in $U_b\times V_a$ and a natural number $n$ such that $(x,\text{Fr}^{-n}(\Phi(x,s)))\in i_{k,l}^{-1} (X^u_{i(\bar s ),i(\bar {c} )})$ for all $(x,s)\in U_b\times V_a$.
\end{enumerate}

%Moreover, if $(k,l)$ is such that $(G_k\times G_l)\cap Y$ is infinite, then for each $a\in $

  \end{lemma}

\begin{proof}

    As $Y$ satisfies Property $(\star)$ witnessed by $L_{i}(x,y,\bar d)$ and $L_{i,n_i}(x,y,\bar d)$ we have:

    $$Y=\bigcup_i\left\{(x,y)\in V_i: L_i(x,y,\bar d)=0\right\},$$ where each $V_i$ is some open subset of $K\times K$ intersecting the set of zeros of $L_i(x,y,\bar d)$ in an infinite set, $L_1(0,0)=0$ and 
    $$\frac{\partial L_1}{\partial y}(0,0)\neq 0.$$ 
    By Fact \ref{implicitFunctionTheoremf} there is an analytic function $h(x)$ converging in a neighborhood $U$ of $0$ such that $h(0)=0$ and $L_1(x,h(x))=0$ for all $x\in U$. Shrinking $U$ we may assume that for all $x\in U$, $(x,h(x))\in V_1$, therefore for all $x\in U,$  $(x,h(x))\in Y$.

    Let $n$ be the maximal natural number ($n$ may be $0$) such that there is some analytic function $f(x)$ with:

    $$h(x)=\kappa x + f(x^{p^n}).$$

Then $$h(x)= \kappa x + \sum_{i\geq 1}d_{ip^n}x^{ip^n}$$ for some coefficients $d_{ip^n}$.

%By Lemma \ref{translatesYAreLocallyFinite} $N(f)$ is finite so 
As $Y$ is not $G$-affine it follows that $N_+(f)$ is finite as it is the first coefficient in which the expansion of $h_{a_1}(x)$ and $h_{a_2}(x)$ differs for $(a_1,a_2)$ $\acl$-generic independent elements. By Lemma \ref{lemmaN} there is $m$ such that $N_+(f)=p^m$. Thus, there is an open neighborhood $U_0$ of $0$ such that for all $a\in U_0$, if we define $$h_a(x)=h(x+a)-h(a)$$ and put $$h_a(x)=\sum_{n\geq 1} d_{n}(a)x^n,$$ then $d_{n}(a)=d_n$ for all $n$ such that 
$$s_n(h)=\{d_n(a):a\in U_0\}$$
is finite.

As the graph of $h_a(x)$ is contained in $t_{\bar a}(Y),$ $Y-t_{\bar a}(Y)$ contains the graph of 
$$(h-h_a)(x)=\sum_{i\geq 1}e_{i}(a)x^{ip^m}$$
where $e_{i}(a)=d_{ip^m}-d_{ip^m}(a)$.
By definition of $N_+(f)$, $\{e_{1}(a):a\in U\}$ is infinite so we can choose $c\in U$ such that $e_1(c)\neq 0$.

%In the same fashion, for all $b\in U$, $Y-t_{\bar b}(Y)$ contains the graph of $$h-h_b=\sum_{i\geq 1}e_{i}(b)x^{ip^m}$$ we can pick $c\in U$ such that $e_{1}(c)\neq 0$. 
For  $a\in U$ let $$G_a(x)=\sum_{i\geq 1}e_{i}(a) x^i$$ so $$(h-h_a)(x)=G_a(x^{p^m}).$$ As $e_{1}(c)\neq 0$ it follows that $G'_c(0)\neq 0$, so Fact \ref{inversionf}, implies that there is an analytic function $G^{-1}_c$ defined in some neighborhood of $0$ such that $G_c(G_c^{-1}(z))=z$ for all $z$ in that neighborhood.

We prove now that the graph of $G_a \circ G_c^{-1}$ is contained in $(Y-t_{\bar a}(Y))\circ (Y-t_{\bar c}(Y))^{-1}=X_{\bar a,\bar c}$:

Let $y=G_a(G_c^{-1}(x))$ and we will prove that $(x,y)\in (Y-t_{\bar a}(Y))\circ (Y-t_{\bar c}(Y))^{-1}=X_{\bar a,\bar c}$. So we have to show that there is some $z$ with $(z,x)\in Y-t_{\bar c}(Y)$ and $(z,y)\in Y-t_{\bar a}(Y)$. Let $z=\text{Fr}^{-m}(G_c^{-1}(x))$. Then, $$(h-h_c)(z)=G_c(z^{p^m})=G_c(G_c^{-1}(x))=x$$ and as the graph of $h-h_c$ is contained in $Y-t_{\bar c}(Y)$, then $(z,x)\in Y-t_{\bar c}(Y)$. Similarly, $$(h-h_a)(z)=G_a(z^{p^m})=G_a(G_c^{-1}(x))=y$$ and as the graph of $h-h_a$ is contained in $Y-t_{\bar a}(Y)$, then $(z,y)\in Y-t_{\bar a}(Y)$, so $(x,y)\in (Y-t_{\bar a}(Y))\circ (Y-t_{\bar c}(Y))^{-1}$.

Now note that  
$$G_c^{-1}(x)=\frac{1}{e_{1}(c)}x + L(x)$$ 
for some analytic function $L$ having a zero of degree at least $2$ at $x=0$.

Therefore, $$G_a\circ G_c^{-1} (x) = \frac{e_{1}(a)}{e_{1}(c)} x + \sum_{i>1} f_{i}(a) x^{i}$$ for some coefficients $f_{i}(a)$. By Proposition \ref{coefficientsAreAnalyticProp}, each $f_{i}(a)$ is a power series in the variable $a$, and the same is true for $e_1(a)$. Therefore, there is a power series $H(x,a)$ such that for all $a$ in a neighborhood of $0$ one has that $H(x,a)=(G_a\circ G_c^{-1})(x)$, so 
$$H_x(0,a)=(G_a\circ G_c^{-1})'(0)=\frac{e_{1}(a)}{e_{1}(c)}$$ 
that takes infinitely many values as $a$ varies. 

We have already proved clauses 1-4 in the statement and only 5 and 6 are missing. Let us prove 5. Let $a\in U$ and let $(b_1,b_2)\in (Y-t_{\bar a}(Y))\circ (Y-t_{\bar c}(Y))^{-1}$, so there is some $z$ such that $(z,b_1)\in Y-t_{\bar c}(Y)$ and $(z,b_2)\in Y-t_{\bar a}(Y)$.

 We start proving:
 
\begin{claim}\label{nLClaim}
    There is a natural number $n$, open sets $U_2\ni z$ and $V\ni a$ and an analytic function 
    $$L:U_2\times V \to K$$
    such that 
    $$\text{Fr}^{-m}(L(x,s))\in t_{\bar s}(Y)$$
    for all $(x,s)\in U_2\times V$.
\end{claim}

\begin{claimproof}

As $(z,b_2)\in Y-t_{\bar a}(Y)$, there are $y_1,y_2$ such that $(z,y_1)\in Y$, $(z,y_2)\in t_{\bar a}(Y)$ and $b_2=y_1-y_2$.

By the Property $(\star)$ for $Y$, as $(z,y_1)\in Y$, there is a natural number $n$ and an irreducible polynomial $L_{i,n}$ such that: 
$$\displaystyle L_{i,n}\left(z,y^{p^n}_1\right)=0, $$   
   $$ \displaystyle\frac{\partial L_{i,n}}{\partial y}\left(z,y^{p^n}_1\right)\neq 0 $$

and for all $(x,y)$ in some neighborhood $V_i$ of $(z,y_1)$, if $L_{i,n}(x,y^{p^n})=0$ then $(x,y)\in Y$. 

By Fact \ref{implicitFunctionTheoremf} there is an analytic function  $g_1$ defined in some neighborhood $U'_1$ of $z$ such that 

$$g_1(z)=y^{p^n}_1$$

%$\text{Fr}^{-n_2}(g_2(z))=y_2$, 

and  
$$L_{i,n}(x,g_1(x))=0$$ so that 
$$L_{i,n}(x, (\text{Fr}^{-n}(g_1(x)))^{p^n})=0$$
for all $x\in U_1$.

Shrinking $U'_1$ one may assume that for all $x\in U'_1$, 
$$(x,\text{Fr}^{-n} \circ g_1(x))\in V_1,$$ 
so 
$$(x,\text{Fr}^{-n} \circ g_1(x))\in Y$$
for all $x\in U'_1$ so that the graph of $\text{Fr}^{-n}\circ g_1$ is  contained in $Y$.

In the same way, as $(z,y_2)\in t_{\bar a}(Y)$, we can find a natural number $m$ and an analytic function $g_2$ defined in a neighborhood $U'_2$ of $z$ such that
$$g_2(z)=y^{p^m}_2$$

and such that the graph of $\text{Fr}^{-m}\circ g_2$ is contained on $t_{\bar a}(Y)$.

We affirm that for $s$ close enough to $a$, the graph of  $$\text{Fr}^{-m}(g_2(x-a+s))+h(a)-h(s)$$ (for $x\in U_2$) is contained in $t_{\bar s}(Y)$.

For proving this, let
$$y:=\text{Fr}^{-m}(g_2(x-a+s))+h(a)-h(s)$$
then $(x,y)\in t_{\bar s}(Y)$ if and only if 
$$(x+s,y+h(s))\in Y$$
which in turn, this is equivalent to 
$$(x+s-a,y+h(s)-h(a))\in t_{\bar a}(Y).$$ 

As the graph of $\text{Fr}^{-m}\circ g_2$ is contained in $t_{\bar a}(Y)$ it is enough to prove that  
$$\text{Fr}^{-m}(g_2(x+s-a))=y+h(s)-h(a)$$ but this is precisely the definition of $y$.

As $\text{Fr}^{-m}(g_2(x-a+s))+h(a)-h(s)=\text{Fr}^{-m}(L'(x,s))$ where $$L'(x,s):=g_2(x-a+s)+h(a)^{p^{m}}-h(s)^{p^{m}},$$ is an analytic function such that the graph of $\text{Fr}^{-m}(L'(x,s))$ is contained on $t_{\bar s}(Y)$ for $s$ close enough to $a$.  

Thus, for $x\in U_2:=U'_1\cap U'_2$ and for $s$ close enough to $a$,
$$\text{Fr}^{-n}(g_1(x)) - \text{Fr}^{-m}(L'(x,s))\in Y-Y_{\bar s}$$ 
Assume $n\geq m$  so $$\text{Fr}^{-n}(g_1(x)) -\text{Fr}^{-m}(L'( x,s))=\text{Fr}^{-n}\left( g_1(x)- \text{Fr}^{n-m}(L'( x,s))\right)$$
so we take $$L(x,s):= g_1(x)- \text{Fr}^{n-m}(L'( x,s))$$ and conclude.
\end{claimproof}

Similarly, as $(z,b_1)\in Y-t_{\bar c}(Y)$ there is an open set $U_1\ni z$, an analytic function 
$$g:U_1\to K$$ 
and a natural number $k_0$ such that 
$$\text{Fr}^{-k_0}(g(b_1))=z$$ and the graph of $\text{Fr}^{-k_0}\circ g$ is contained in $Y-t_{\bar c}(Y)$. Moreover, we may assume that $g'$ is never zero. So there exists an inverse for $g$, say $f$. Thus, the graph of $f\circ \text{Fr}^{k_0}$ is contained in $(Y-t_{\bar c}(Y))^{-1}$. 

Let $U_2$ and $L$ provided by Claim \ref{nLClaim}. Shrinking $U_1$ we may assume that $\hat x:=f(\text{Fr}^{k_0}(x))\in U_2$ for all $x\in U_1$. So for $s\in V$ and $x\in U_1$, $$(x,\text{Fr}^{-n}(L(\hat x,s))\in (Y-t_{\bar s}(Y))\circ (Y-t_{\bar c}(Y))^{-1}.$$
So we take $n$ and $$\Phi(x,s):=L(\hat x,s)$$ to get Clause 5 of the statement. Clause 6. follows similarly as Clause 5.
\end{proof}

\begin{quote}
\textbf{From now on we fix $Y\subseteq G\times G$ a one dimensional $\mathcal G$-definable not $\mathcal G$-affine set having the Property $(\star)$ witnessed by some polynomials $L_i$, $L_{i,n_i}$ and $L_i^{(l,k)}$. For $a,c\in Y$ we set:
\begin{equation}\label{defX}
    X_{a,c}:=(Y-t_a(Y))\circ (Y-t_c(Y))^{-1}.
\end{equation}}
\end{quote}

First we make a definition:
\begin{defi}
    Let $Y\subseteq K\times K$ be a $\mathbb K$-definable with $\dim(Y)=1$. We say that $Y$ has infinite translates if for all except finitely  many $a\in Y$ there are cofinitely many $b\in Y$ such that $t_a(Y)\cap t_b(Y)$ is infinite.
\end{defi}

We will find a group configuration for $\mathcal G$ that is interalgebraic (in $\mathbb K$) with the following shift, by an appropriate $t$, of the rank $2$ group configuration of Diagram \ref{gcFullStructure}.

 \begin{equation}\label{gcFullStructureIntro}
 	\begin{tikzcd}
 	(t+a_0, ta_1) \arrow[ddd, dash] \arrow[rrr, dash] &&& t+b \arrow[rrr, dash]  &&& t+a_0+a_1 b \arrow[dddllllll,  dash] 
  \\
&&&&&&
\\
& & && 
&&
\\
 (t+b_0,tb_1)  \arrow[ddd, dash] &&& t+b_0+ a_0b_1+ a_1 b_1 b &&& \\ &&&&&&
 \\
 &&&&&&
 \\
 (t+a_0b_1+b_0,ta_1b_1)  \arrow[uuuuuurrr, crossing over, dash] &&&&&&
	\end{tikzcd}
\end{equation}

In the complete case, we will use Lemma \ref{existsHLemma} to define a field configuration for $\mathcal G$ that is interalgebraic (in $\mathbb K$) with the one of Diagram \ref{gcFullStructureIntro}. However, to prove Theorem \ref{thmAditiveVersion} in the general case, we provide the following first order statement which, by completeness of the theory $ACVF_{p,q}$, can be proved in a complete field.

\begin{lemma}\label{lemmaIntersectionSubeVersion}

     Let $\phi(\bar x,\bar y,\bar z)$ be a formula in the language of valued fields and let $L_i(x,y,\bar z)$, $L_{i,n_i}(x,y,\bar z)$ and $L_{i,n_i}^{(l,k)}(\bar x,\bar y,\bar z)$ be polynomials. 

      For all $\bar d$, if  $$Y(\bar d):=\{(\bar x,\bar y)\in G^2:\mathbb K\models \phi(x,y,\bar d)\}$$ 

  is not $G$-affine and satisfies Property $(\star)$ witnessed by $L_i(x,y,\bar d)$, $L_{i,n}(x,y,\bar d)$ and $L_{i,n_i}^{(l,k)}(x,y,\bar d)$ then there is some $c\in Y$ such that if for all $a\in Y$ we set 
     
     $$X_{a}:=X_{a,c}$$ as defined on Equation \ref{defX}, there is an open ball $B\subseteq K$ such that for all $t\in B$ and for all $a_0$, $b_0$, $b$, $a_1$, $b_1\in K$, if 
     $$\mathfrak g_f:=(t,t a_1, t+a_0 , tb_1, t+b_0,  t a_1 b_1 ,  t+a_1b_0+a_0,  t+b,  t+a_1 b+a_0,  t+a_1 b_1 b + b_1 a_0 +b_0)\in  B^{10}$$

     there is $\mathfrak g=(\tau,\ \alpha_0,\ \alpha_1,\ \beta_0,\ \beta_1,\ \gamma_0,\ \gamma_1,\ \kappa,\ \eta,\mu)\in Y^{10}$ such that:
     
     If $f(a)$ is the function mapping the $i'$th coordinate of $\mathfrak g$ into the $i'$th coordinate of $\mathfrak g_f$ then: 
   \begin{equation}\label{tangentConditions}\tag{T}
   X_a\tangent l_{f(a)}\ \text{For all }a\in \mathfrak g.
    \end{equation}

    Moreover, if for $\varepsilon,\rho,\sigma\in Y$ we define:    
    $$Z(\varepsilon,\rho,\sigma):=(X_\tau\circ X_\varepsilon)\+(X_\rho\circ X_\sigma)\ominus (X_\tau\circ X_\rho)$$ then for each election $(a,r_0,r_1,r_2)$ of coordinates of $\mathfrak g$  we have:

    \begin{itemize}
        \item If $tf(a)=tf(r_0)+f(r_1)f(r_2)-tf(r_1)$ and $X_\tau\circ X_a\cap Z(r_0,r_1,r_2)$ is finite, then there are infinitely many $s\in Y$ with:
     $$| X_\tau\circ X_a\cap Z(r_0,r_1,r_2)|<| X_\tau\circ X_s\cap Z(r_0,r_1,r_2)|.$$ 

     \item If $tf(a)=f(r_1)f(r_2)$
    and $X_\tau\circ X_a\cap X_{r_1}\circ X_{r_2}$ is finite, there are infinitely many $s\in Y$ with:
    $$|X_\tau\circ X_a\cap X_{r_1}\circ X_{r_2}|<|X_\tau\circ X_s\cap X_{r_1}\circ X_{r_2}|.$$
    \end{itemize}  
 \end{lemma}

\begin{proof}
    By the observations we made after the definition or Property $(\star)$ and after Definition \ref{infiniteTranslatesDef}, the statement of Lemma is first order on the parameter $\bar d$ so it is enough if we prove it for \textit{some} model of ACVF. In particular without loss of generality we may assume that $\mathbb K$ is complete.

    Let $\bar d$ be such that $Y:=Y(\bar d)\subseteq G\times G$ is not $G$-affine and satisfies the Property $(\star)$ witnessed by some polynomials $L_i$ and $L_{i,n_i}$ and $L_{i}^{(l,k)}(x,y,\bar z)$. Apply Lemma \ref{existsHLemma} with $Y^u=Y$ and fix $U$, $c\in U$, $h$ and $H$ as provided by the conclusion. For $s\in U$ call $$d(s):=\frac{\partial H}{\partial x}(0,s).$$
By Corollary \ref{existsInf} and Clause 4 of Claim \ref{existsHLemma}, there is some open ball $B\subseteq s_1(H)$. Let $t\in B$ and $a_0$, $b_0$, $b$, $a_1$, $b_1\in K$ such that
     $$\{t a_1, \ t+a_0 , \ tb_1,\  t+b_0,\  t a_1 b_1 ,\  t+a_1b_0+a_0,\  t+b, \ t+a_1 b+a_0,\  t+a_1 b_1 b + b_1 a_0 +b_0\}\subseteq B.$$

By definition of $B$ we can find  $\tau',\ \alpha'_0,\ \alpha'_1, \beta'_0,\ \gamma'_0,\ \kappa',\ \eta',\ \mu' \in U$ such that:

\begin{itemize}
    \item $d(\tau')=t$
    \item $d(\alpha'_0)=t+a_0$,
    \item $d(\beta'_0)=t+b_0$,
    \item $d(\gamma'_0)=t+a_0b_1+b_0$.
   
    \item $d(\kappa')=t+b$,
    \item $d(\eta')=t+ a_0+a_1b$,
    \item $d(\mu')=t+b_0+a_0b_1+a_1b_1b$,

    \item $d(\alpha'_1)=ta_1$,
    \item $d(\beta'_1)=tb_1$ and
     \item $d(\gamma'_1)=ta_1b_1$.
\end{itemize}

Define $\tau=(\tau',h(\tau'))$, $\alpha_0=(\alpha'_0,h(\alpha'_0))$, $\alpha_1=(\alpha'_1,h(\alpha'_1))$, $\beta_0=(\beta'_0,h(\beta'_0))$, $\beta_1=(\beta'_1,h(\beta'_1))$, $\gamma_0=(\gamma'_0,h(\gamma'_0))$, $\gamma_1=(\gamma'_1,h(\gamma'_1))$, $\kappa=(\kappa',h(\kappa'_0))$, $\eta=(\eta',h(\eta'_0))$ and $\mu=(\mu',h(\mu'_0))$. 

By Lemma \ref{claimDerivativeEqualsImpliesTangentLemma} and Clause 3 of the list of properties for $H$ and $h$ of Lemma \ref{existsHLemma}, this choice of 
$$(\tau,\alpha_0,\alpha_1,\beta_0,\beta_1,\gamma_0,\gamma_1,\kappa,\eta,\mu)$$ 
implies that the condition (\ref{tangentConditions}) of the statement holds.

Now we prove the bullets of the statement.
%\begin{claim}\label{intersectionIncreasesGamma}
 %   There are infinitely many $s\in Y$ with:

  %   $$| X_\tau\circ X_{\gamma_1}\cap X_{\alpha_1}\circ X_{\beta_1}|<| X_\tau\circ X_s\cap X_{\alpha_1}\circ X_{\beta_1}|.$$ 
%\end{claim}

%
%\begin{claimproof}(Proof of Claim \ref{intersectionIncreasesGamma})
    
%\end{claimproof}

 Assume that $(a,r_0,r_1,r_2)$ are as in the first bullet of the conclusion

 \begin{claimproof}

  For $Z\subseteq G\times G$ let $Z^{*}:=Z\cap (A\times A)$ and $Z^{(k,l)}:=i_{k,l}^{-1}(Z)$.
  
    By the conclusion of Lemma \ref{existsHLemma}, $H$ is a uniform power expansion for $(X^*_a)_{a\in Y^*}$ at $(0,0)$.
        So $H_\tau(x,s)=H(H(x,s),\tau')$ is a uniform power expansion for $(X^*_\tau \circ X^*_a)_{a\in Y^*}$ at $(0,0)$.
      
        By Clause 5 of Lemma \ref{existsHLemma} for all $\alpha\in U$ and all $b\in X^*_{\bar\alpha}$, there is a uniform Fr-power expansion of $(X^*_a)_{a\in Y^*}$ for $X^*_{\bar\alpha}$ at $b$. As $\tau'\in U$, in particular, $X^*_\tau$ admits a Fr-power expansion at any point. 
        
        So by Lemma \ref{lemmaCanCompose}, for all $\alpha \in U$ and all $b\in X^*_\tau \circ X^*_{\bar\alpha}$, there is a uniform Fr-power expansion of $(X^*_\tau\circ X^*_a)_{a\in Y^*}$ for $X^*_\tau\circ X^*_{\bar\alpha}$ at $b$.

       Moreover, the function 
       $$T(x) := H(H(x,r'_0),\tau')+H(H(x,r_2'),r'_1)-H(H(x,r'_1),\tau')$$ is a power expansion for $(Z(r_0,r_1,r_2))^*$ at $(0,0)$ that satisfies $$T'(0)=tf(t_0) + f(r_1)f(r_2)-tf(r_1)$$
       
      % =t(t+a_0+ta_1+a_1b-ta_1)=t(t+a_0+a_1b)

       and as $$\frac{\partial H_\tau}{\partial x}(0,a')=tf(a),$$ by our conditions on $(a,r_0,r_1,r_2)$ it follows that $$\frac{\partial H_\tau}{\partial x}(0,a')=T'(0)$$

       Let $Z:=Z(r_0,r_1,r_2)$, so we apply Lemma \ref{intersectionSubeLemmaAnalitic} (with $q=0$, the family $(X^*_\tau\circ X^*_a)_{a\in Y^*}$, the set $Z^*$ and the power expansion $T$ of $Z^*$ at $(0,0)$) and conclude that there is $W\ni a'$ open contained on $U$ such that 
       $$|X_{\tau'}^*\circ X_{a'}^*\cap Z^*|<|X^*_{\tau'}\circ X_s^*\cap Z^*|$$
        for all $s\in W\setminus{a'}$

       so

       \begin{equation}\label{g0creceEquation}
       |X_{\tau}\circ X_{a}\cap Z \cap (A\times A)|<|X_{\tau}\circ X_{i(\bar s)}\cap Z \cap (A\times A)|.
       \end{equation}

       for all $s\in W\setminus\{a'\}$.

       Now for each $k,l$ list $X_{\tau}\cap X_{a} \cap (G_k\times G_l) =\{(x_1,y_1),\ldots,(x_m,y_m)\}$

       So there are $(x^*_j,y^*_j)\in A\times A$ such that $i_{k,l}(x^*_j,y^*_j)=(x_j,y_j)$.

       Now fix $V_j\ni (x_j^*,y_j^*)$ pairwise disjoint open subsets of $A\times A$.

      % $b=(i_k(b_1),i_l(b_2))\in X_\tau\circ X_\eta \cap Z\cap (G_k\times G_l)$ and each $U'\ni (b_1,b_2)$ open subset of $A\times A$, there is $W'\ni \eta'$ an open subset of $U$ such that 

       %$$|X_{\tau}\circ X_{i(\bar s)}\cap Z \cap U'|\geq 1$$

       %for all $s\in W'\setminus \{\eta'\}$.

       By Clause \ref{frLejosGLemma} of Lemma \ref{existsHLemma}, for all $s\in U$ and all $\bar b\in X_{i(\bar s)}^{(k,l)}$ the family $(X_a^{(k,l)})_{a\in Y^*}$ admits an Fr-power expansion for $X_{i(\bar s)}^{(k,l)}$ at $b$. In particular, $X^{(k,l)}_{a}$ admits an Fr-power expansion at any point. Then by Lemma \ref{lemmaCanCompose}, for all $s\in U$ and $(x',y')\in X^{(k,l)}_\tau\circ X_{i(\bar s)}^{(k,l)}$, the family
       $$(X^{(k,l)}_\tau\circ X_a^{(k,l)})_{a\in Y}$$ admits an Fr-power expansion for $ X^{(k,l)}_\tau\circ X_{i(\bar s)}^{(k,l)}$ at $(x',y')$. Thus we may apply Lemma \ref{propIntCreceGeneralVersionLejos} with this family, the point $(x^*_j,y^*_j)$, the open $V_j$ and the set with no isolated points $Z^{(k,l)}$ to conclude that there is some $W^{(k,l)}_j\ni {a'}$ such that 

       $$|X_\tau^{(k,l)}\circ X_{i(\bar s)}^{(k,l)} \cap Z^{(k,l)}\cap V_j|\geq 1$$

       for all $s\in W^{(k,l)}_j$.

    As $V_j$ are disjoints, if $s\in W^{(k,l)}:=\bigcap_j W^{(k,l)}_j $ then 

     $$|X_\tau^{(k,l)}\circ X_{i(\bar {s})}^{(k,l)} \cap Z^{(k,l)}|\geq m.$$

     applying $i_{k,l}$ that is injective and onto $G_k\times G_l$ we get: 
    \begin{equation}\label{gnkCreceEcuacion}
    |X_\tau\circ X_{a} \cap Z \cap (G_k\times G_l)|\leq  |X_\tau\circ X_{i(\bar s)} \cap Z \cap (G_k\times G_l)|
\end{equation}

As $\{G_k\times G_l\}_{k,l}$ is a partition of $G\times G$, Equations \ref{g0creceEquation} and \ref{gnkCreceEcuacion} imply that 

$$|X_\tau\circ X_{a} \cap Z| < |X_\tau\circ X_{i(\bar s)} \cap Z|$$

for all $s\in W\cap \bigcap_{k,l} W^{(k,l)}$.  \end{claimproof}

  The second bullet has a similar proof, this completes the proof of Lemma \ref{lemmaIntersectionSubeVersion}. \end{proof}

We are ready now to present the proof of Theorem \ref{thmAditiveVersion}. 

\begin{proof} \textit{(Proof of Theorem \ref{thmAditiveVersion})}.

Let $\mathcal G=(G,\+,\ldots)$ be as in the hypothesis of the statement of Theorem \ref{thmAditiveVersion}. Let $X\subseteq G\times G$ be as provided by Fact \ref{existsXVersion} so $X$ is not $\mathcal G$-affine. We may assume that $X$ is strongly minimal. 

By Lemma \ref{lemmaYStarG} there is a $\mathcal G$-definable strongly minimal set $Y$ with Property $(\star)$. Moreover, $Y$ is also not $\mathcal G$-affine.

Then we can write:
\begin{equation}\label{decompY}
   Y=\bigcup_i C_i\cap V_i
\end{equation}

where each $C_i$ is a one-dimensional Zariski irreducible closed and each $V_i$ is open.

Let $c\in Y$ and $B\subseteq K$ an open ball provided by Lemma \ref{lemmaIntersectionSubeVersion}. For $a\in Y$ let

$$X_{a}:=X_{a,c}$$ as defined in Equation \ref{defX}.

Notice that $(X_a)_{a\in Y}$ is a $\mathcal G$-definable family. 
%We want to find a group configuration on $\mathbb K$ for the group $\mathbb G_a\ltimes\mathbb G_m$ and translate it inside $B\times B$. 
%Here $\mathbb G_a$ and $\mathbb G_m$ are the additive and multiplicative group of $\mathbb K$ respectively and the semidirect product is given by the action of $\mathbb G_m$ on $\mathbb G_a$ by multiplication.

Let $t\in B$. Since addition and multiplication are continuous, there are $B_0$ and $B_1$ open balls around $0$ and $1$ respectively such that:

\begin{itemize}
  \item $tB_1\subseteq B,$
  \item $t+B_0\subseteq B,$
  \item $t B_1 B_1\subseteq B,$ 
  \item $t+B_1B_0+B_0\subseteq B,$ and
  \item $t+B_1B_1B_0+B_1B_0+B_0\subseteq B$.

\end{itemize}

%$xy\in B_1$ for all $x,y\in B_1$,\\
%and $x_0+x_1x\in B_0$ for all $x_0,x\in B_0$ and all $x_1\in B_1$.

%Moreover, taking $B_1$ with positive radius we have that 

%$tx\in B$ for all $x\in B_1$.

%In addition we take the radius of $B_0$ to be bigger than the radius of $B$ so we have:

%$t+x\in B$ for all $x\in B_0$.

Let $\bar d$ be some parameter such that $G$, $A$, $i:A\to G$ and $Y\subseteq G\times G$ are all $\mathbb K$-definable over $\bar d$. In this proof, whenever we take algebraic closure or we compute the dimension of tuples we will do it over $\bar d$. As $B_0$ and $B_1$ are open balls, let $(a_0,b_0,b)\in B^3_0$ be a triple of $\mathbb K$-dimension $3$ over $t,\bar d$ and let $(a_1,b_1)\in B_1^2$ be a tuple of $\mathbb K$-dimension $2$ over $a_0,b_0,b, t,\bar d$. Then $(a_0,b_0,b,a_1,b_1)$ is a tuple of $\mathbb K$-dimension $5$ over $t,\bar d$.

Then we have a $\mathbb K$-group configuration over $\bar d$ for $\mathbb G_a\ltimes\mathbb G_m$ given by:

\begin{equation}\label{gcFullStructure}
 	\begin{tikzcd}
 	(a_0, a_1) \arrow[ddd, dash] \arrow[rrr, dash] &&& b \arrow[rrr, dash]  &&& a_0+a_1 b \arrow[dddllllll,  dash] 
  \\
&&&&&&
\\
& & && %\arrow[dll, dash]  
%\arrow[ddddll, dash]  
&&
\\
 (b_0,b_1)  \arrow[ddd, dash] &&& b_0+ a_0b_1+ a_1 b_1 b &&& \\ &&&&&&
 \\
 &&&&&&
 \\
 (a_0b_1+b_0,a_1b_1)  \arrow[uuuuuurrr, crossing over, dash] &&&&&&
  	\end{tikzcd}
\end{equation}

Moreover by our choice of $B_0$ and $B_1$ we have that:
$$\{t a_1, \ t+a_0 , \ tb_1,\  t+b_0,\  t a_1 b_1 ,\  t+a_1b_0+a_0,\  t+b, \ t+a_1 b+a_0,\  t+a_1 b_1 b + b_1 a_0 +b_0\}\subseteq B.$$

Let $\tau,\ \alpha_0,\ \alpha_1,\ \beta_0,\ \beta_1,\ \gamma_0,\ \gamma_1,\ \kappa,\ \eta$ and $\mu$ be elements of $Y$ provided by Lemma \ref{lemmaIntersectionSubeVersion} 

%Therefore we can find $\alpha_0\in Y\cap O$ such that $d(\alpha_0)=t + a_0$. Note that $\alpha_0\in K\times K$ therefore $\alpha_0$ is itself an ordered pair of elements of $K$ but we will not use its coordinates so we will not give a name to them.

%In the same way we can find $\beta_0$ and $\gamma_0$ elements of $Y\cap O$ such that:

%\begin{itemize}
 %   \item $d(\alpha_0)=t+a_0$,
  %  \item $d(\beta_0)=t+b_0$ and
   % \item $d(\gamma_0)=t+a_0b_1+b_0$.
    
%\end{itemize}

%Similarly we find $p,q$ and $r$, elements of $Y\cap O$, such that:

%\begin{itemize}
 %   \item $d(p)=t+b$,
  %  \item $d(q)=t+ a_0+a_1b$ and
   % \item $d(r)=t+b_0+a_0b_1+a_1b_1b$.
%\end{itemize}

%And finally there are $\alpha_1,\beta_1$ and $\gamma_1$ elements of $Y\cap O$ such that:

%\begin{itemize}
 %   \item $d(\alpha_1)=ta_1$,
  %  \item $d(\beta_1)=tb_1$ and
   %  \item $d(\gamma_1)=ta_1b_1$.
%\end{itemize}
We start proving:

\begin{claim}\label{genericFiniteIntersectionClaim}
    The families $(\cl(t_a(Y))_{a\in Y}$ and $(\cl(X_\tau \circ X_a))_{a\in Y}$ have both generic finite intersection.
\end{claim}

\begin{claimproof}
    Let $\bar d$ be a tuple of parameters defining the family $(\cl(t_a(Y))_{a\in Y}$. Let $a_1,a_2\in Y$ be $\mathbb K$-generic independent over $\bar d$ and assume that $$\cl(t_{a_1}(Y))\cap \cl(t_{a_2}(Y))$$ is infinite. By Lemma \ref{bastaIndiscernibleLemma} we may assume that there is $\mathcal I=(a_1,a_2,a_3,\ldots)$ an indiscernible sequence over $\bar d$ starting with $(a_1,a_2)$.%If we decompose $Y$ as in Equation \ref{decompY} we may take $a_1,a_2\in C_1\cap V_1$.

    By Equation \ref{decompY} 
    $$t_{a_1}(Y)=\bigcup_i (t_{a_1}(C_i)\cap t_{a_1}(V_i)),$$ then 
    $$\cl(t_{a_1}(Y))=\bigcup_i(t_{a_1}(C_i)).$$ Similarly
    $$\cl(t_{a_2}(Y))=\bigcup_i(t_{a_2}(C_i)).$$

    Thus, if $$\cl(t_{a_1}(Y))\cap \cl(t_{a_2}(Y))$$ is infinite, there are $i,j$ such that $t_{a_1}(C_i)\cap t_{a_2}(C_j)$ is infinite. 

    Since this property is definable and $\mathcal I$ is indiscernible, the same pair $C_i$ $C_j$ works for any pair of elements $a_k,a_l$ if $k<l$. As there are finitely many $C_i$ it implies that $C_i=C_j$. Fix such $i$.

   Now, we may assume that $a_1\in C_i\cap V_i$, so there are infinitely many $a_2\in C_i\cap V_i$ such that $t_{a_1}(C_i)=t_{a_2}(C_i)$. Moreover, if we take $a_2$ close enough to $a_1$ then $t_{a_1}(V_i)\cap t_{a_2}(V_i)\neq \emptyset$ and as it is anon empty open set, in particular it is infinite, so $$t_{a_1}(C_i\cap V_i)\cap t_{a_2}(C_i\cap V_i)$$ is infinite. But  $$t_{a_1}(C_i\cap V_i)\cap t_{a_2}(C_i\cap V_i)\subseteq t_{a_1}(Y)\cap t_{a_2}(Y),$$ contradicting Lemma \ref{nonAffineIntersectsFiniteLemmaF} so we conclude that the family $\cl(t_a(Y))_{a\in Y}$ has generic finite intersection.

    Thus by Lemma \ref{restaCompuestaFiniteIntersectionLemma}, if we define:

    $$\tilde X_a:= \cl\left(\cl\left(\cl(Y)-\cl(t_a(Y))\right) \circ \cl\left(\cl(Y)-\cl(t_c(Y))\right)^{-1}\right)$$

    then the family $(\tilde X_a)_{a\in Y}$ has generic finite intersection which implies that
    $$(\cl(X_a))_{a\in Y}$$ also has generic finite intersection. Applying Lemma \ref{restaCompuestaFiniteIntersectionLemma} again, the family 

    $$(\cl(\cl(X_\tau)\circ \cl(X_a)))_{a\in Y}$$ has generic finite intersection and since $$(\cl(\cl(X_\tau)\circ \cl(X_a))\subseteq \cl(X_\tau\circ X_a),$$  the family 

    $$(\cl(X_\tau\circ X_a))_{a\in Y}$$ has generic finite intersection.\end{claimproof}

Now we are ready to prove:

\begin{claim}\label{claimGCVersionAdditive}
    The following is a rank $2$ group configuration for the structure $\mathcal G$:

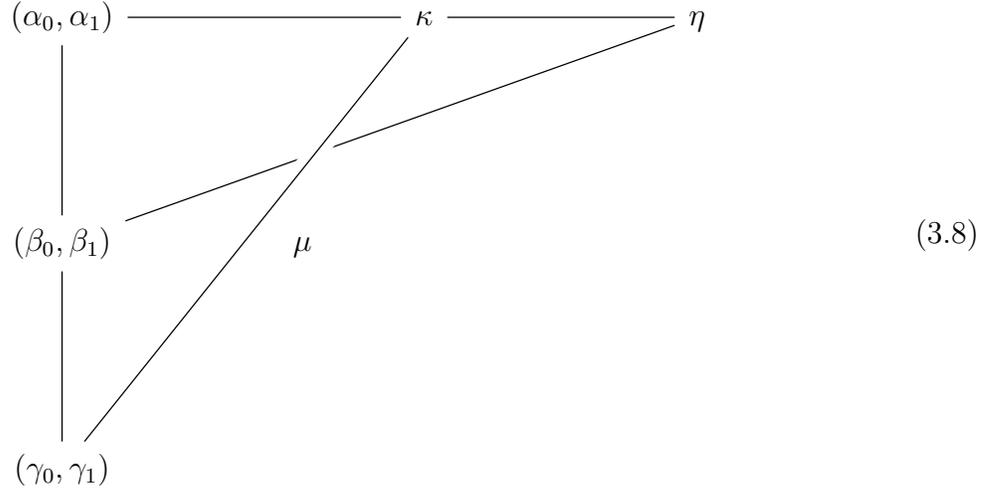
    \begin{equation}\label{GCReduct}
 	\begin{tikzcd}
 	(\alpha_0, \alpha_1) \arrow[ddd, dash] \arrow[rrr, dash] &&& \kappa \arrow[rrr, dash]  &&& \eta \arrow[dddllllll,  dash] 
  \\
&&&&&&
\\
& & && %\arrow[dll, dash]  
%\arrow[ddddll, dash]  
&&
\\
 (\beta_0,\beta_1)   \arrow[ddd, dash] && \mu  &&&& \\ &&&&&&
 \\
 &&&&&&
 \\
  (\gamma_0,\gamma_1)  \arrow[uuuuuurrr, crossing over, dash] &&&&&&
  	\end{tikzcd}
 	\end{equation}
\end{claim}

\begin{claimproof}

As the tuples on Diagram \ref{GCReduct} satisfy condition (\ref{tangentConditions}), by Lemma \ref{claimTangentFiniteLemma} one has that 
$$a_0\in \acl_{\mathbb K} (\alpha_0),$$ moreover 
$$1\geq \dim_{\mathbb K}(\alpha_0)=\dim_{\mathbb K} (a_0,\alpha_0)\geq \dim_{\mathbb K} (a_0)=1.$$

Then $\dim_{\mathbb K}(\alpha_0/a_0)=0$ so $ \alpha_0\in \acl_{\mathbb K}(a_0)$. Therefore, $a_0$ and $\alpha_0$ are $\mathbb K$-interalgebraic. The same is true with all the correspondent tuples in both diagrams, so we conclude that the Diagram \ref{GCReduct} is interalgebraic (in $\mathbb K$) with the Diagram \ref{gcFullStructure}.

Therefore, the rank computed in $\mathbb K$ of a tuple of elements of Diagram  \ref{GCReduct} is the same as the rank of the corresponding tuple on Diagram \ref{gcFullStructure}.  As $\mathcal G$ is a reduct of $\mathbb K$ the rank of tuples computed on $\mathcal G$ is bigger or equal than the rank of tuples computed on $\mathbb K$. So for proving our claim it is enough to prove that the rank computed in $\mathcal G$ does not increase so we have to prove that there are still algebraic relations between the tuples lying on the same lines of the diagram. 

Let us prove that
$$\eta\in \acl_{\mathcal G}(\alpha_0,\alpha_1,\kappa)$$

Let  
$$Z(\alpha_0,\alpha_1,\kappa):=(X_\tau \circ X_{\alpha_0})\+(X_{\alpha_1}\circ X_\kappa)\ominus(X_\tau \circ X_{\alpha_1}).$$

If $$X_\tau\circ X_\eta\cap Z(\alpha_0,\alpha_1,\kappa)$$ is finite, the algebraic relation follows from the first bullet on the conclusion of Lemma \ref{lemmaIntersectionSubeVersion} taking $(a,r_0,r_1,r_2)=(\eta,\alpha_0,\alpha_1,\kappa)$: There are infinitely many $s\in Y$ such that:

$$| X_\tau\circ X_\eta\cap Z(\alpha_0,\alpha_1,\kappa)|<| X_\tau\circ X_s\cap Z(\alpha_0,\alpha_1,\kappa)|.$$

As $Z(\alpha_0,\alpha_1,\kappa)$ is a $\mathcal G$-definable set with parameters $\alpha_0,\alpha_1,\kappa$, the set

$$E:=\{s\in Y: | X_\tau\circ X_\eta\cap Z(\alpha_0,\alpha_1,\kappa)|<| X_\tau\circ X_s\cap Z(\alpha_0,\alpha_1,\kappa)|\}$$ is also $\mathcal G$-definable with parameters $\alpha_0,\alpha_1,\kappa$. As $E$ is infinite and $Y$ is strongly minimal, $Y\setminus E$  is finite and as $\eta\in Y\setminus E$ we conclude that 
$$\eta\in \acl_{\mathcal G}(\alpha_0,\alpha_1,\kappa).$$

If $$X_\tau\circ X_\eta\cap Z(\alpha_0,\alpha_1,\kappa)$$ is infinite, in particular $\cl(X_\tau\circ X_\eta)\cap Z(\alpha_0,\alpha_1,\kappa)$ is infinite, by Claim \ref{genericFiniteIntersectionClaim} the family $(\cl(X_\tau\circ X_a))_{a\in Y}$ has generic finite intersection, so Lemma \ref{noGenericFiniteZLemma} implies that $\eta$ is not $\mathbb K$-generic over $(\alpha_0,\alpha_1,\kappa)$, in particular 
$$\left\{b\in Y:X_\tau \circ X_b\cap Z(\alpha_0,\alpha_1,\kappa)\text{ is infinite}\right\}$$ does not contain any $\mathbb K$-generic element so it is finite. Since it is $\mathcal G$-definable and contains $\eta$ we conclude $$\eta\in \acl_{\mathcal G}(\alpha_0,\alpha_1,\kappa).$$

 The rest of algebraic relations needed for the definition of group configuration are provided by either the first or the second bullet in the conclusion of Lemma \ref{lemmaIntersectionSubeVersion}. \end{claimproof}

 Then we use the group configuration provided by Claim \ref{claimGCVersionAdditive}, apply Fact \ref{fieldConfig} and conclude that there is an infinite field interpreted by $\mathcal G$. This finishes the proof of Theorem \ref{thmAditiveVersion}.
\end{proof}

\chapter{The Multiplicative Case}\label{multiplicativeCaseVersion}

In this chapter we prove Zilber's conjecture in the case where the universe of the strongly minimal structure is $M:=K\setminus \{0\}$, the definable sets are $\mathbb K$-definable and multiplication ($\cdot$) is definable. More precisely:

\begin{thm}\label{thmMultiplicativeVersion}
    Let $\mathbb K=(K,+,\cdot,v,\Gamma)$ be an algebraically closed  valued field with $\text{char} (K)=p\geq 0$. Let $M:=K\setminus\{0\}$ and let $\mathcal M=(M,\cdot,\ldots)$ be a strongly minimal and non locally modular first order structure expanding $\mathbb G_m=(M,\cdot)$ whose definable sets are all $\mathbb K$-definable. Then $\mathcal M$ interprets an infinite field. %that is $\mathbb K$-definable isomorphic to $(K,+,\cdot,0,1)$. 
 \end{thm}
%By Fact \ref{uniqueFieldFact} if $\mathcal M$ interprets an infinite field it is definable isomorphic to $(K,+,\cdot)$ so we will devote the entire Chapter \ref{multiplicativeCaseVersion} to find an infinite field interpretable in $\mathcal M$.
 From now on we fix $\mathcal M=(M,\cdot,\ldots)$ as in the hypothesis of Theorem \ref{thmMultiplicativeVersion}. If $Z$ is $\mathcal M$-definable, we compute the Morley rank and Morley degree of $Z$ in the sense of $\mathcal M$. We say that $Z$ is strongly minimal if it is strongly minimal in the sense of $\mathcal M$. 

 From now on, we fix $X$ as provided by Fact \ref{existsXVersion} (and the subsequent comment) so $X$ is $\mathcal M$-definable, strongly minimal and it is not $\mathcal M$-affine.

By Lemma \ref{lemmaYStar} there is $Y\subseteq M\times M$ having Property $(\star)$. Moreover, $Y$ is also $\mathcal M$-definable, strongly minimal and it is not $\mathcal M$-affine. 

We divide the proof of Theorem \ref{thmMultiplicativeVersion} in two cases:

In Section \ref{findingAGroup} we deal with the case in which $Y$ has finitely many slopes at $(1,1)$ -see Definition \ref{defiYinfiniteSlopes} bellow-. In this case we find a group configuration for $\mathcal M$ that is $\mathbb K$-inter algebraic with an standard group configuration for $\mathbb G_a$. Then we use this group configuration to find a $\mathcal M$-interpretable group that is locally isomorphic to the additive group so we may apply the results of Chapter \ref{additiveCaseVersion} to find a field interpretable in $\mathcal M$. 
 
In Section \ref{findingAField} we deal with the case in which $Y$ has infinitely many slopes at $(1,1)$. In this case we can find a field in the same way as we did in Chapter \ref{additiveCaseVersion} as in this case an analogous of key Lemma \ref{existsHLemma} -which is very particular of the additive case- is given almost for free for the family of translates of $Y$ after composing with the inverse of some fix translate (see Lemma \ref{existsHLemmaMult}). %Some of the arguments here are a repetition of Chapter \ref{additiveCaseVersion} as we did not have enough time to write down adequate generalizations covering both cases.

\begin{defi}\label{defidL1}
    Let $L(x,y)$ be any polynomial, for each $a=(a_1,a_2)\in M\times M$ such that $L(a_1,a_2)=0$ and $\frac{\partial L}{\partial y}(a_1,a_2)\neq 0$ we define:

    $$\mathfrak d(L,a):=\frac{\frac{-\partial L}{\partial x}(a_1,a_2)}{\frac{\partial L}{\partial y}(a_1,a_2)}\frac{a_1}{a_2}.$$
    \end{defi}
\begin{defi}\label{defiYinfiniteSlopes}

    Let $Z\subseteq M\times M$ be a $\mathbb K$-definable set of dimension $1$. We say that $Y$ having Property $(\star)$ witnessed by $L_i$ and $L_{i,n_i}$ \emph{has infinitely many slopes at $(1,1)$} if for all $V\ni(1,1)$ open there is some open set $U$ with $(1,1)\in U\subseteq V$ such that

    $\frac{\partial L_1}{\partial y}(a)\neq 0$ for all $a\in U$
    and 
    $$s_1(Y,L_1,U):=\left\{\mathfrak d(L_1,a):a=(a_1,a_2)\in Z\cap U\text{, }L_1(a_1,a_2)=0 \right\}$$
    is infinite.
\end{defi}

The motivation for this definition is the following observation:

\begin{lemma}\label{hInfDerivadasLemma}
    Assume that $\mathbb K$ is metric. Let $L_1(x,y)$ be a polynomial and let $D$ be its set of zeros. Let $h:U\to M$ be an analytic function such that for all $a\in U$, $L_1(a,h(a))=0$.
    
    Then $$h_a(x):=h(x\cdot a)\cdot h(a)^{-1}$$ is an analytic function converging in a neighborhood of $1$ whose graph is contained in $t_{a}(D)$ and  $$h_a'(1)=\frac{\frac{-\partial L_1}{\partial x}(a_1,a_2)}{\frac{\partial L_1}{\partial y}(a_1,a_2)}\frac{a_1}{a_2}.$$

    Moreover, if $Y\subseteq K\times K$ is a $1$ dimensional set with infinitely many slopes at $(1,1)$ witnessed by $L_i$ and $L_{i,n}$, is such that $(a,h(a))\in Y$ for all $a\in U$, then 
    $$\left\{h_{a}'(1):a\in U\right\}$$ is infinite.
\end{lemma}

\begin{proof}

From the chain rule applied to the equation 

$$L_1(a,h(a))=0$$ we get that 
$$h'(a)=\frac{\frac{-\partial L_1}{\partial x}(a,h(a))}{\frac{\partial L_1}{\partial y}(a,h(a))}$$

and $$h_a'(1)=h'(a)\frac{a}{h(a)}.$$

Now, assume that the graph of $h$ is contained on $Y\cap D$ and $\left\{h_{a}'(1):a\in U\right\}$ is finite, where $D$ is the vanishing set of $L_1$. In particular, there is some $c$ such that 
$$c=\frac{\frac{-\partial L_1}{\partial x}(a_1,a_2)}{\frac{\partial L_1}{\partial y}(a_1,a_2)}\frac{a_1}{a_2}$$ for infinitely many $(a_1,a_2)\in D$.

Thus, if $E$ is the vanishing set of 
$$F(x,y):=c y\frac{\partial L_1(x,y)}{\partial y}+x \frac{\partial L_1(x,y)}{\partial x},$$
then $D\cap E$ is infinite and as $D$ is Zariski Irreducible and $E$ is Zariski closed it follows that $D\subseteq E$ so $F$ vanishes on $D$ and $s_1(Y,L_1,K\times K)=\{c\}$, in particular $Y$ does not have infinitely many slopes at $(1,1)$.
\end{proof}

Notice again that if $\phi(x,y,\bar z)$ is a formula in the language of valued fields and $L_i(x,y,\bar z)$ and $L_{i,n}(x,y,\bar z)$ are polynomials, then if we define:

$$Y(\bar d)=\{(x,y)\in K\times K: \mathbb K\models \phi(x,y,\bar d)\}$$ we have that
$$\left\{\bar d:Y(\bar d)\text{ has infinitely many slopes at }(1,1)\text{ witnessed by }L_i(x,y,\bar d)\text{ and }L_{i.n}(x,y,\bar d)\right\}$$ is $\mathbb K$-definable.

%In this chapter if $Z$ is $\mathcal M$-definable set we say that $Z$ is strongly minimal if it is strongly minimal regarding the structure $\mathcal M$.

\section{Finitely many slopes at $(1,1)$}\label{findingAGroup}

In this section we deal with the case in which $Y$ does not have infinitely many slopes at $(1,1)$. In this case we will find a $\mathcal M$-interpretable group which contains a subgroup of finite index that is $\mathbb K$-definably isomorphic to a subgroup of the additive group.

\subsection{Finding a group}\label{findingAGroupConfiguration}

In this sub-section we prove the following Proposition:

\begin{prop}\label{existsGroupConfigurationMult}
    Let $Y\subseteq M\times M$ be $\mathcal M$-definable and strongly minimal. Assume that $Y$ is not $\mathcal M$-affine and has property $(\star)$ witnessed by some polynomials $L_i$ and $L_{i,n_i}$. Assume also that is not the case that $Y$ has infinitely many slopes at $(1,1)$ witnessed by $L_i$ and $L_{i,n_i}$. Then there is a $\mathcal M$ group configuration, $\mathfrak g=(\alpha,\beta,\gamma,\kappa,\eta,\mu)$

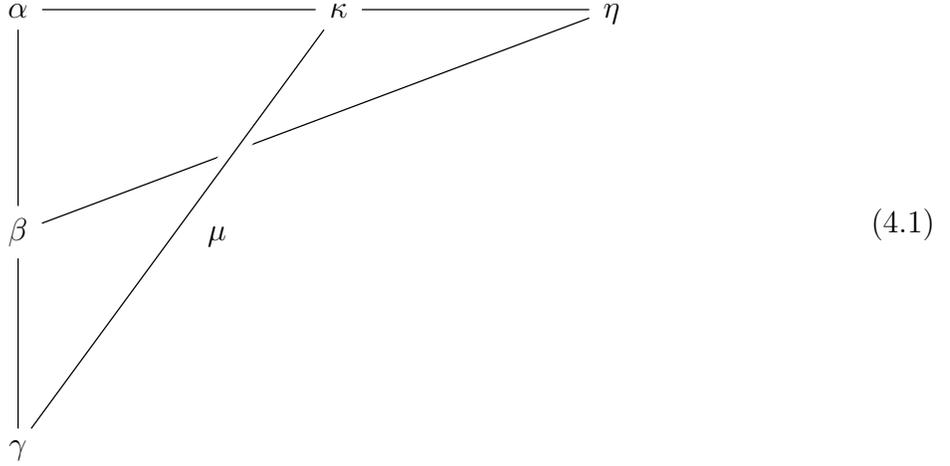
\begin{equation}\label{gcReductGroup}
 	\begin{tikzcd}
 	\alpha \arrow[ddd, dash] \arrow[rrr, dash] &&& \kappa \arrow[rrr, dash]  &&& \eta \arrow[dddllllll,  dash] 
  \\
&&&&&&
\\
& & && %\arrow[dll, dash]  
%\arrow[ddddll, dash]  
&&
\\
 \beta  \arrow[ddd, dash] && \mu &&&& \\ &&&&&&
 \\
 &&&&&&
 \\
 \gamma  \arrow[uuuuuurrr, crossing over, dash] &&&&&&
  	\end{tikzcd}
\end{equation}

that is $\mathbb K$-interalgebraic with a standard group configuration for the additive group $(K,+)$.\end{prop}

To prove Proposition \ref{existsGroupConfigurationMult} we will use the results of Section \ref{computations}. Again we will need a first order statement analogous to Lemma \ref{lemmaIntersectionSubeVersion}.

We need a lemma first:

\begin{lemma}\label{slopeDefinable}
    Let $q\in K$, $U\ni q,V\subseteq K$ be open and $f:U\to V$ be a $\mathbb K$-definable and analytic function whose analytic expansion around $q\in U$ is given by $f(x)=\sum_{n\geq 0} c_n(x-q)^n$. Then $c_n\in \dcl(\bar d,q)$ for all $n\in \mathbb K$.
\end{lemma}

\begin{proof}
    We proceed by induction on $n$:

    $c_0=f(q)$ which is clearly definable over $q,\bar d$.
    $c_1=\lim_{x\to q}\frac{f(x)-f(q)}{x-q}$ which is definable over $q,\bar d$ by an standard $\varepsilon-\delta$ definition. For $N\geq 1$ $c_{N+1}=g'(q)=\lim_{x\to q} \frac{g(x)}{x-q}$ where
   \begin{equation*}
       g(x):=\begin{cases}
       \frac{f(x)-f_N(x)}{(x-q)^N}, & \text{ if }x\neq q\\
       0 & \text{ if }x=q\\
        \end{cases}
   \end{equation*}

   and $$f_N(x):= \sum_{n=0}^N c_n (x-q)^n$$ which by induction is $\mathbb K$-definable over $\bar d, q$.
\end{proof}

\begin{defi}
    Let $F(x,y)$ be a polynomial with vanishing set $D$ 
 let $f(x)$ be an analytic function defined in some neighborhood of $1$ whose graph is contained in $D$. For each $N\in\mathbb N$, each $a\in D$ such that $\frac{\partial F}{\partial y}(a)\neq 0$ and each $c\in D$ such that $\frac{\partial F}{\partial y}(c)\cdot\frac{\partial F}{\partial x}(c)\neq 0$ let $\mathfrak {dc}_N(F,a,c)$ be the coefficient of $(x-1)^N$ in the Taylor expansion around $1$ of  $f_a(g(x))$ where $f_a(x):=f(xa)/f(a)$ and $g(x)$ is the inverse of $f_c(x)$.    
\end{defi}

\begin{remark}\label{remarkSlopeIsAlgebraic}
    As the derivative of $F$, $-\frac{\partial F}{\partial x}/\frac{\partial F}{\partial y}$ is rational and rules of inverse, quotient and chain produce rational functions when applied to rational functions, it follows that the coefficient $\mathfrak {dc}_N(F,a,c)$ is a rational function in the variables $a,c$.
\end{remark}

\begin{defi}
    Let $Y\subseteq M\times M$ be a $\mathbb K$-definable set of dimension $1$, $N$ a natural number, $L(x,y)$ a polynomial and $c\in Y$. We say that \emph{$Y$ has infinitely many slopes of order $N$ respect to $c$ in $L$} if whenever $D$ is the vanishing set of $L$ then $(1,1)$ is an interior point of $Y\cap D$,
    $$\frac{\partial L(c)}{\partial y}\cdot \frac{\partial L(c)}{\partial x}\neq 0$$ 
    
    and for all $V\ni (1,1)$ open there is $U\ni (1,1)$ an open subset of $V$ such that:

        $$\{\mathfrak {dc}_N(D,a,c):a\in U\}$$ is infinite.

 \end{defi}

Notice that with this definition,  having infinitely many slopes of order $1$ coincides with Definition \ref{defiYinfiniteSlopes}.
% Again, the motivation is the following observation:

 %\begin{lemma}\label{hInfDerivadasLemmaGeneral}
  %  Assume that $\mathbb K$ is metric. Let $L_1(x,y)$ be a polynomial and let $D$ be its set of zeros. Let $h:U\to K$ be an analytic function such that for all $a\in U$, $L_1(a,h(a))=0$.
    
   % Then $$h_a(x):=h(x\cdot a)\cdot h(a)^{-1}$$ is an analytic function converging in a neighborhood of $1$ whose graph is contained in $t_{a}(D)$ and  with power expansion around $1$ given by:
 %   $$h_a(x)=1+\sum_{n\geq 1} \mathfrak d_n (L_1,a) (x-1)^n.$$

    %Moreover, if $Y\subseteq K\times K$ is a $1$ dimensional set with infinitely many slopes of order $N$ witnessed by $L_i$ and $L_{i,n}$, and $h:U\to K$ is such that $(a,h(a))\in Y$ for all $a\in U$, then if we define $$H_0(x,s):=\frac{h(xs)}{h(s)}$$ we have that
    %$$s_N(h,U):=\left\{d_N(H_0,a):a\in U\right\}$$ is infinite.
%\end{lemma}

%\begin{proof}
 %   Follows from Lemma \ref{defDnLemma}
%\end{proof}

Again, if $N$ is a natural number, $\phi(x,y,\bar z)$ is a formula in the language of valued fields and $L(x,y,\bar z)$ is a polynomial, then if we define:

$$Y(\bar d):=\{(x,y)\in K\times K: \mathbb K\models \phi(x,y,\bar d)\}$$ we have that

\begin{equation*}
\left\{(\bar d,c):Y(\bar d)\text{ has infinitely many slopes of order }N\text{ respect to }c\text{ in }L(x,y,\bar d)\right\}
\end{equation*}

is $\mathbb K$-definable.

We keep notation of Definition \ref{defidn} setting $q=1$.
First we prove:

\begin{lemma}\label{existsHEasyClaim}
    Assume that $\mathbb K$ is complete. Let $Y\subseteq K\times K$ be not $\mathbb G_m$-affine with the Property $(\star)$ witnessed by some polynomials $L_{i}$ and $L_{i,n_i}$.

 Then for each $B_1\ni (1,1)$ open, there is $c\in Y\cap B_1$, $V\ni 1$ open and functions $h(x)$ and $H(x,s)$ analytic in $V$ and $V\times V$ respectively such that if for $\alpha\in Y$ we define

     $$X_\alpha:=t_\alpha(Y)\circ t_c(Y)^{-1},$$

        then:
        \begin{enumerate}
             \item $\bar x:=(x,h(x))\in Y$ for all $x\in V$.
             \item The graph of $H(x,s)$ is contained on $X_{\bar s}$ for all $s\in V$.
           % \item $N_0 = N(H) p^M$.
            \item $d_1(H,s)=1$ and $d_n(H,s)=0$ for all $s\in V$ and all $1<n<N(H)$.
            
        \end{enumerate}
    \end{lemma}

    \begin{proof}
   As $L_1(1,1)=0$ and $\frac{\partial L_1}{\partial y}\neq 0$ we may apply Fact \ref{implicitFunctionTheoremf} and find $V\ni 1$ open and $h:V\to K$ analytic such that 
    $$\bar x:=(x,h(x))\in Y$$ for all $x\in V$. Moreover, Property $(\star)$ implies that if $D$ is the set of zeros of $L_1(x,y)$, $(1,1)$ is an interior point of 
    $Y\cap D$. Thus, shrinking $V$ if necessary, we may assume that $$(x,h(x))\in Y\cap D$$ for all $x\in V$.

    For $a\in V$ let $h_a$ be the analytic function defined by
    $$h_a(x)=h(x\cdot a)\frac{1}{h(a)}, $$ 
    
    and define $H_0(x,s):=h_s(x)$ an analytic function. Shrinking $V$ we may assume that $H_0$ converges in $V\times V$.

   So the power expansion of $h_a(x)$ around $1$ is given by: 
    
    $$h_a(x)=1+\sum_{n\geq 1} d_n(H_0,a) (x-1)^n.$$

    Notice that the graph of $h_a$ is contained in $t_{\bar a}(Y)$. By Remark \ref{remarkSlopeIsAlgebraic}, for each $n$  $\mathfrak {dc}_n(F,x,c)= \mathfrak {dc}_n(F,y,c)$ is an algebraic relation between $x$ and $y$, so $N_0:=N(H_0)$ is finite as it is the first $n$ such that $\mathfrak {dc}_n(F,a_1,c)\neq \mathfrak {dc}_n(F,a_2,c)$ for any pair $(a_1,a_2)$ of algebraic independent elements.

     As $Y$ does not have infinitely many slopes at $(1,1)$, Lemma \ref{hInfDerivadasLemma} implies that $N_0>1$.

    Shrinking $U$ we may assume that for $n<N_0$, $d_n(H_0,s)$ is constant as $s$ varies in $V\setminus\{1\}$.

    Let $r$ be the maximal natural number ($r$ may be $0$) such that exists some function $g(x)$ analytic in some neighborhood of $1$ with $$h(x)=g(x^{p^r}).$$  
    Shrinking $V$ we assume that $g$ is analytic in $V$.    
    By our definition of $r$, $g$ is not a function in the variable $x^p$, thus by Fact \ref{identityTheoremf}, $g'(x)$ is not the zero function. Then, $g'(x)=0$ just for finitely many $x\in V$. Let $c'\in V$  such that $g'(c')\neq 0$ and $d_{N_0}(H_0,\text{Fr}^{-M}(c'))\neq 0$. Moreover, we choose $c'$ such that if we define $c''=\text{Fr}^{-r}(c')$ and $c:=(c'',h(c''))\in Y$, then $c\in Y\cap B_1$.

    As $$g_{c'}(1)=g'(c')\frac{c'}{g(c')}\neq 0,$$ by Fact \ref{implicitFunctionTheoremf} there is an inverse for $g_{c}$, say $f$ converging in an open neighborhood of $1$. Shrinking $V$ we may assume that both, $f$ and $g$, converge in $V$.

   Now we define:
   $$H(a,x):=g_{a^{p^r}} \circ f (x) = g(a^{p^r} f(x))\frac{1}{g(a^{p^r})}.$$

   This is analytic in a neighborhood of $(1,1)$, shrinking $V$ we may assume that it is analytic in $V\times V$.

   Now we prove that for $a\in V$ the graph of $H(x,a)$ is contained in $X_{\bar a}$.

   Let $(a,x_0)\in V\times V$ and let $y_0:=H(a,x_0)$ we will prove that $(x_0,y_0)\in X_{\bar a}$. Let $z:=f(x_0)$ so $g_c(z)=g(c'z)/g(c')=x_0$ and $g_a(z)=g(az)/g(a)=y_0$.
   
   Let $w:=\text{Fr}^{-r}(z)$. Then $$h_{c''}(w)=h(c'' w)\frac{1}{h(c'')}=g\left((c''w)^{p^r}\right)\frac{1}{g(c{''}^{p^r})}=g(c'z)\frac{1}{g(c')}=g_{c'}(z)=x_0.$$ As the graph of $h_{c''}$ is contained in $t_{c}(Y)$, $(w,x_0)\in t_{c}(Y)$ so 
   \begin{equation}\label{e1}
   (x_0,w)\in t_c(Y)^{-1}.    
   \end{equation}

   Moreover, 
   $$h_{a}(w)=h(a w)\frac{1}{h(a)}=g\left((aw)^{p^r} \right)\frac{1}{g(a^{p^r})}=g_{a^{p^M}}(w^{p^r})=g_{a^{p^r}}(z)=y_0.$$ And as the graph of $h_a$ is contained on $X_{\bar a}$ it follows that:
   \begin{equation}\label{e2}
       (w,y_0)\in X_{\bar a}. 
   \end{equation}

  Equations \ref{e1} and \ref{e2} together imply that
   $$(x_0,y_0)\in t_{\bar a}(Y)\circ t_c(Y)^{-1}=X_{\bar a}.$$  

Now we prove that -possibly after shrinking $V$- $d_1(H,s)=1$ and $d_n(H,s)=0$ for all $s\in V$ and all $1<n<N_0 p^r$.

Write the power expansion of $g_{a^{p^r}}$ and $g_{c'}$ around $1$ as:

$$g_{a^{p^r}}(x) = 1 + \sum_{n\geq 1} b_n(a)(x-1)^n$$

and 
$$g_{c'}(x) = 1 +  \sum_{n\geq 1} b_n(c'')(x-1)^n.$$

As $h_a(x)=g_{a^{p^r}}(x^{p^r})$ and $h_{c''}(x)=g_{c'}(x^{p^r})$, it follows that $h_a$ and $h_{c''}$ have power expansions around $1$ given by:

$$h_a(x)= 1+ \sum_{n\geq 1}b_n(a) (x-1)^{n p^r} $$
and 

$$h_{c''}(x)=1+\sum_{n\geq 1}b_n(c'') (x-1)^{np^r}.$$

So by definition of $d_n(H_0,a)$ we have that $b_n(a)=d_{np^r}(H_0, a)$ and $b_n(c'')=d_{np^r}(H_0,c'')$. Moreover, $d_k(H_0,a)=d_k(H_0,c'')=0$ for all $k\in \mathbb N$ such that $p^r\nmid k$. As we are assuming that $d_{N_0}(H_0,c'')\neq 0$ it follows that $N_0$ is a multiple of $p^r$ so $N_0=N p^r$ for some $N$. We are also assuming that for $n<N_0$, $d_n(H_0,s)$ is constant as $s$ varies on $V\setminus \{1\}$, in particular, for all $a\in V\setminus \{1\}$ and all $n<N_0$ 

$d_n(H_0,a)=d_n(H_0,c'')$ so:

$$h_{a}(x)= 1 + b_1 (x-1)^{p^r} + \ldots  + d_{N_0}(H_0,a)(x-1)^{N_0} + \sum_{n > N} d_{np^r}(H_0,a) (x-1)^{np^r}$$

and 

$$h_{c''}(x)= 1 + b_1(x-1)^{p^r} + \ldots  +  d_{N_0}(H_0,a)(x-1)^{N_0}+ \sum_{n > N} d_{np^r}(H_0,c'')(x-1)^{np^r},$$
where for $i< N$
$$b_i:=b_i(a)=d_{ip^r}(H_0,a)=d_{ip^r}(H_0,c'')=b_i(c'').$$

So 

$$g_{a^{p^r}}(x)=1+b_1(x-1)+\ldots+d_{N_0}(H_0,a)(x-1)^{N}+\sum_{n>N}d_{np^r}(H_0,a)(x-1)^n$$

and 

$$g_{c'}(x)=1+b_1(x-1)+\ldots+d_{N_0}(H_0,c'')(x-1)^{N}+\sum_{n>N}d_{np^r}(H_0,c'')(x-1)^n.$$

By Lemma \ref{N} (applied to the function $g_{c'}(x+1)-1$) as $f$ is the inverse of $g_{c'}$, for $n < N$ the coefficient of $(x-1)^n$ in the inverse of $g_{a^{p^r}}$ is the same as the coefficient of $(x-1)^n$ in $f(x)$. Thus, if the power expansion of $f$ around $1$ is given by:
$$f(x)=1 + c_1 (x-1) +\ldots + c_{N-1}(x-1)^{N-1}+c_{N} (x-1)^{N}+ \sum_{n>N} c_n (x-1)^n,$$ 

the inverse of $g_{a^{p^r}}$ has a power expansion around $1$ of the form:

$$g_{a^{p^r}}^{-1}(x)=1 + c_1 (x-1) +\ldots + c_{N-1}(x-1)^{N-1}+ c_{N}(a) (x-1)^{N}+ \sum_{n>N} e_n (x-1)^{n}$$ for some coefficients $e_n$.

As for $n<N$ the coefficient of $(x-1)^n$ in $f$ is the same as the coefficient of $(x-1)^n$ in $g^{-1}_{a^{p^r}}$, Lemma \ref{compositionMultLemma} implies that for $n<N$ the coefficient of $(x-1)^n$ in 
$$g_{a^{p^r}}\circ f$$ is the same as the coefficient of $(x-1)^n$ in

%for any $n$ the coefficient of $(x-1)^n$ in the composition
%$$g_{a^{p^M}}\circ f$$
%is determinated by the coefficients of $(x-1),(x-1)^2,\ldots,(x-1)^n$ in the power expansions of $g_{a^{p^M}}$ and $f$. 

$$g_{a^{p^r}}\circ g^{-1}_{a^{p^r}}=1+(x-1).$$ 
Therefore,
\begin{equation}\label{ln}
    g_{a^{p^r}}\circ f (x) = 1 + (x-1) + d_{N}(H,a)(x-1)^{N}+\sum_{n>N} d_n(H,a)(x-1)^n
\end{equation}
so $d_1(H,a)=1$ and $d_n(H,a)=0$ for all $1<n<N$. Then it is enough if we show that $N(H)=N$. It follows from Equation \ref{composeCoefEquation} that the coefficient $d_{N}(H,a)$ in the Equation \ref{ln} is given by:

\begin{equation}\label{dNEquation}
    d_{N}(H,a)=b_1 c_{N} + b_2 f_2+\ldots+ b_{N-1} f_{N-1} + d_{N_0}(H_0,a) c_{1}^{N}
\end{equation}

where for $n=2,\ldots, N-1$

$$f_n=\sum_{i_1+\ldots+i_n=N} b_{i_1}\cdots b_{i_n}.$$

Notice that for $n=2,\ldots,N-1$, $f_n$ only depends on $b_1,\ldots,b_{N-1}$. So the first $N-1$ terms on the right side of Equation \ref{dNEquation} only depend on $b_1,\ldots,b_{N-1}$ and $c_1,\ldots, c_N$ that are constant as $a$ varies. So as $d_{N_0}(H_0,a)$ takes infinitely many values as $a$ varies, the same is true for $d_{N}(H,a)$, thus $N(H)=N$. \end{proof}

Now we present a first order statement analogous to Lemma \ref{lemmaIntersectionSubeVersion}.

\begin{lemma}\label{lemmaIntersectionSubeVersionMultGroup}

     Let $\phi(x,y,\bar z)$ be a formula in the language of valued fields, $L_i(x,y,\bar z)$ and $L_{i,n_i}(x,y,\bar z)$ be polynomials and let $N > 1$ be a natural number, then 

      For all $\bar d$ the following is true:
      
      Assume that
      $$Y:=\{(x,y)\in K^2:\mathbb K\models \phi(x,y,\bar d)\}$$
       is not $\mathbb G_m$-affine, has Property $(\star)$ and does not have infinitely many slopes at $(1,1)$ witnessed by $L_i(x,y,\bar d)$ and $L_{i,n_i}(x,y,\bar d)$.
       
      Then, for each open ball $B_1\ni (1,1)$, there is some $c\in Y\cap B_1$ such that whenever $N$ is the minimum natural number such that $Y$ has infinitely many slopes of order $N$ respect to $c$ in $L_1$ and we define:

   $$X_a:=t_a(Y)\circ t_c(Y)^{-1},$$
   
   there is an open ball $B\subseteq K$ such that for all $t\in B$ and for all $a_0$, $b_0$, $b\in K$, if 
     $$\mathfrak g_t:=(t,t+a_0, t+b, t+a_0+b, t+b_0, t+a_0+b_0+b)\in B^6$$
          
     there is $\mathfrak g:=(\tau,\alpha,\beta,\gamma,\kappa,\eta,\mu)\in Y^6$ such that:
   \begin{equation*}\label{tangentConditionsMultGroup}\tag{TN}
   f(a):=\mathfrak{dc}_N(L_1,a,c) \text{ maps the }i\text{'th coordinate of }\mathfrak g \text{ to the }i\text{'th coordinate of }\mathfrak g_t
    \end{equation*}

    Moreover, for all $a,r_1,r_2$ in $\mathfrak g$ if 
    $$t+f(a)=f(r_1)+f(r_2)$$
   and $X_\tau \circ X_a\cap X_{r_1}\circ X_{r_2}$ is finite, there are infinitely many $s\in Y$ with:
   $$|X_\tau\circ X_a\cap X_{r_1}\circ X_{r_1}|<| X_\tau\circ X_s\cap X_{r_1}\circ X_{r_2}|.$$

 \end{lemma}
 
    \begin{proof}
    By the observation we made after Lemma \ref{hInfDerivadasLemma} the statement of lemma is first order on the parameter $\bar d$ so it is enough if we prove it for \textit{some} model of ACVF. In particular, without loss of generality, we may assume that $\mathbb K$ is complete.

    Let $\bar d$ be such that $Y:=Y(\bar d)$ is not $\mathbb G_m$-affine, satisfies Property $(\star)$ witnessed by some polynomials $L_i$ and $L_{i,n_i}$ and also that that $Y$ does not have infinitely many slopes at $(1,1)$ witnessed by $L_1$.
    
    %Let $c'\in U$ such that $d_N(H_0,c')\neq 0$, set $c:=(c',h(c'))\in Y$  and for $\alpha\in Y$ define:

   Let $c\in Y$, $V$, $U$, $h$ and $H$ be as provided by Lemma \ref{existsHEasyClaim} applied to $Y$ and $B_1$. Shrinking $U$ we may assume that $H$ is analytic in $U\times U$ and $h$ is analytic in $U$. So if $N$ is such that the conditions of the Statement of Lemma \ref{lemmaIntersectionSubeVersionMultGroup} the only option is that $N=N(H)$ so we may assume that this is the case.

    As $s_N(H)$ is infinite, by Corollary \ref{existsInf} there is some open ball $B\subseteq s_N(H)$. Let $t\in B$ and $a_0$, $b_0$, $b\in K$ such that:
     $$\{t + a_0,t+b_0, t+a_0+b_0, t+b, t+a_0+b, t+a_0+b_0+b\}\subseteq B.$$
In this proof, for $s\in U$ we will write $d(s)$ instead of $d_N(H,s)$.

By definition of $B$ we can find  $\tau',\ \alpha',\ \beta',\ \gamma',\ \kappa',\ \eta', \mu'\in U$ such that:

\begin{itemize}
    \item $d(\tau')=t$
    \item $d(\alpha')=t+a_0$,
    \item $d(\beta')=t+b_0$,
    \item $d(\gamma')=t+a_0+b_0$,
      \item $d(\kappa')=t+b$,
    \item $d(\eta')=t+ a_0+b$ and
    \item $d(\mu')=t+a_0+b_0+b$.
\end{itemize}

Define $\tau=(\tau',h(\tau'))$, $\alpha=(\alpha',h(\alpha'))$, $\beta=(\beta',h(\beta'))$, $\gamma=(\gamma',h(\gamma'))$, $\kappa=(\kappa',h(\kappa'))$, $\eta=(\eta',h(\eta'))$ and $\mu=(\mu',h(\mu'))$. 

%By Lemma \ref{claimDerivativeEqualsImpliesTangentLemma}, 
This choice of 
$$(\tau,\alpha,\beta,\gamma,\kappa,\eta,\mu)$$ 
imply that the conditions (\ref{tangentConditionsMultGroup}) of the statement holds.

Let us prove the `moreover' part 
%\begin{claim}\label{intersectionIncreasesGamma}
 %   There are infinitely many $s\in Y$ with:

  %   $$| X_\tau\circ X_{\gamma_1}\cap X_{\alpha_1}\circ X_{\beta_1}|<| X_\tau\circ X_s\cap X_{\alpha_1}\circ X_{\beta_1}|.$$ 
%\end{claim}

%
%\begin{claimproof}(Proof of Claim \ref{intersectionIncreasesGamma})
    
%\end{claimproof}

%\begin{claim}\label{intersectionIncreasesEtaMultGroup}

 Let $a,r_1,r_2$ be coordinates of $\mathfrak g$ such that $$t+f(a)=f(r_1)+f(r_2)$$ and assume that 
 
 $$X_\tau\circ X_a\cap X_{r_1} \circ X_{r_1}$$ is finite
%\end{claim}
% \begin{claimproof}
       
    Definition of $H$ implies that $H$ is a uniform power expansion for $(X_a)_{a\in Y}$ at $(1,1)$.
        So $H_\tau(x,s):=H(H(x,s),\tau')$ is a uniform power expansion for $(X_\tau \circ X_a)_{a\in Y}$ at $(1,1)$.
      
        For all $s\in U$ and all $(b_1,b_2)\in X_{\bar s}$, there is a uniform Fr-power expansion of $(X_a)_{a\in Y}$ for $X_{\bar s}$ at $(b_1,b_2)$. As $\tau'\in U$, in particular $X_\tau$ admits a Fr-power expansion at any point. 
        
        So by Lemma \ref{lemmaCanCompose}, for all $s \in U$ and all $(b_1,b_2)\in X_\tau \circ X_{\bar s}$, there is a uniform Fr-power expansion of $(X_\tau\circ X_a)_{a\in Y}$ for $X_\tau\circ X_{\bar s}$ at $(b_1,b_2)$.

       Moreover, the function 
       $$T(x) := H(H(x,r_2'),r_1')$$ is a power expansion for $Z:=X_{r_2}\circ X_{r_1}$ at $(1,1)$ that satisfies $$d_N(T)=f(r_1)+f(r_2)$$

       As $H_\tau(x,1)= h_{\tau'}( h(x))$ then by Clause 3 of Lemma \ref{existsHEasyClaim} we may apply Lemma \ref{compositionMultLemma} and conclude that:

       $$d_N(H_\tau,a')=d_N(H,\tau')+d_N(H,a')=t+f(a')=f(r_1)+f(r_2)=d_N(T)$$

       so we apply Lemma \ref{intersectionSubeLemmaAnalitic} (with $q=1$, the family $(X_\tau\circ X_a)_{a\in Y}$, the set $Z$ and the power expansion $T$ of $Z$ at $(1,1)$) and conclude. % \end{claimproof}

  %The other bullets have similar proofs, 
  This completes the proof of Lemma \ref{lemmaIntersectionSuberVersionMultiplicative}. \end{proof}

 Now we prove Proposition \ref{existsGroupConfigurationMult}. 

 \begin{proof}(Proof of Proposition \ref{existsGroupConfigurationMult})

 Let $Y$ be $\mathcal M$-definable, not $\mathcal M$-affine and strongly minimal and assume that $Y$ has Property $(\star)$  witnessed by some polynomials $L_i$ and $L_{i,n_i}$.  As for each $N$, 
 \begin{multline*}
      Y(N):=\{c\in Y: N \text{ is minimal such that }Y\text{ has infinitely many slopes} \\\text{ of order }N  \text{ with respect to }c\text{ witnessed by }L_i\text{ and }L_{i,n_i}\}
 \end{multline*}
 is a $\mathbb K$ definable set, there is some $N$ and some open ball $B_1\ni (1,1)$ such that $Y\cap B_1\subseteq Y(N)$. 
%Let $\mathcal G=(G,+,\ldots)$ be as in the hypothesis of the statement of Theorem \ref{thmAditiveVersion}. Let $X\subseteq G\times G$ be as provided by Fact \ref{existsXVersion} so $X$ is not $\mathcal G$-affine. We may assume that $X$ is strongly minimal. 

%By Lemma \ref{lemmaYStar} there is a $\mathcal G$-definable strongly minimal set $Y$ with Property $(\star)$. Moreover $Y$ is also not $\mathcal G$-affine.

%Then we can write:

%\begin{equation}\label{decompY}
 %   Y=\bigcup_i C_i\cap V_i
%\end{equation}

%where each $C_i$ is a one-dimensional Zariski irreducible closed and each $V_i$ is open. 

    %As $t_{a_1}(Y)$ and $t_{a_2}(Y)$ are both strongly minimal it implies that As both are Zariski irreducible closed sets it implies that $t_{a_1}(C_i)=t_{a_2}(C_j)=:D$. 
    
    %As $a_1$ is generic, it is an interior point of $C_i\cap Y$ so $(0,0)$ is an interior point of $D\cap t_{a_1}(Y)$. In the same way $(0,0)$ is also an interior point of $D\cap t_{a_2}(Y)$. So there is an open set $U\ni 0$ such that $(U\times U)\cap D\subseteq t_{a_1}(Y) \cap t_{a_2}(Y)$ but this contradicts Lemma \ref{translatesYAreLocallyFinite}.
        
    %As $(a_1,a_2)$ is $\mathbb K$-generic independent, it is also $\mathcal G$-generic independent so 
    %$$t_{a_1}(Y)\cap t_{b}(Y)$$ is infinite for all but finitely many $b\in Y$. In particular there is some $i$ and so

Let $c\in Y\cap B_1$ and $B\subseteq K$ an open ball provided by Lemma \ref{lemmaIntersectionSubeVersionMultGroup} applied to $N$ and $B_1$. Thus, $N$ is the minimal natural number such that $Y$ has infinitely many slopes of order $N$ respect to $c$ witnessed by $L_i$ and $L_{i,n_i}$. For $a\in Y$ let
$$X_{a}:=t_a(Y)\circ t_c(Y)^{-1}.$$
Notice that $(X_a)_{a\in Y}$ is a $\mathcal M$-definable family. 
%We want to find a group configuration on $\mathbb K$ for the group $\mathbb G_a\ltimes\mathbb G_m$ and translate it inside $B\times B$. 
%Here $\mathbb G_a$ and $\mathbb G_m$ are the additive and multiplicative group of $\mathbb K$ respectively and the semidirect product is given by the action of $\mathbb G_m$ on $\mathbb G_a$ by multiplication.

Let $t\in B$. Since addition is continuous, there is $B_0$, an open ball around $0$ such that:
$$t+B_0\subseteq B.$$

%$xy\in B_1$ for all $x,y\in B_1$,\\
%and $x_0+x_1x\in B_0$ for all $x_0,x\in B_0$ and all $x_1\in B_1$.

%Moreover, taking $B_1$ with positive radius we have that 

%$tx\in B$ for all $x\in B_1$.

%In addition we take the radius of $B_0$ to be bigger than the radius of $B$ so we have:

%$t+x\in B$ for all $x\in B_0$.

Let $(a_0,b_0,b)\in B^3_0$ be a triple of $\mathbb K$-dimension $3$ over $t,d$ (where $d$ is the parameter defining $Y$) 

Then, we have a rank 1 $\mathbb K$-group configuration for $\mathbb G_a$ given by:

\begin{equation}\label{gcFullStructureMultGroup}
 	\begin{tikzcd}
 	a_0 \arrow[ddd, dash] \arrow[rrr, dash] &&& b \arrow[rrr, dash]  &&& a_0+b \arrow[dddllllll,  dash] 
  \\
&&&&&&
\\
& & && %\arrow[dll, dash]  
%\arrow[ddddll, dash]  
&&
\\
 b_0  \arrow[ddd, dash] &&& a_0+ b_0+ b &&& \\ &&&&&&
 \\
 &&&&&&
 \\
 a_0+b_0  \arrow[uuuuuurrr, crossing over, dash] &&&&&&
  	\end{tikzcd}
\end{equation}

Moreover, by our choice of $B_0$ we have that:

$$\{ t+a_0 ,\  t+b_0,\   t+a_0+b_0,\  t+b, \ t+b+a_0,\  t+ b + a_0 +b_0\}\subseteq B.$$

Let $\tau,\ \alpha,\ \beta,\ \gamma,\ \kappa,\ \eta$ and $\mu$ be elements of $Y$ provided by Lemma \ref{lemmaIntersectionSubeVersionMultGroup}

Now we are ready to prove:

\begin{claim}\label{claimGCVersionMultGroup}
    The following is a rank $1$ group configuration for the structure $\mathcal M$:

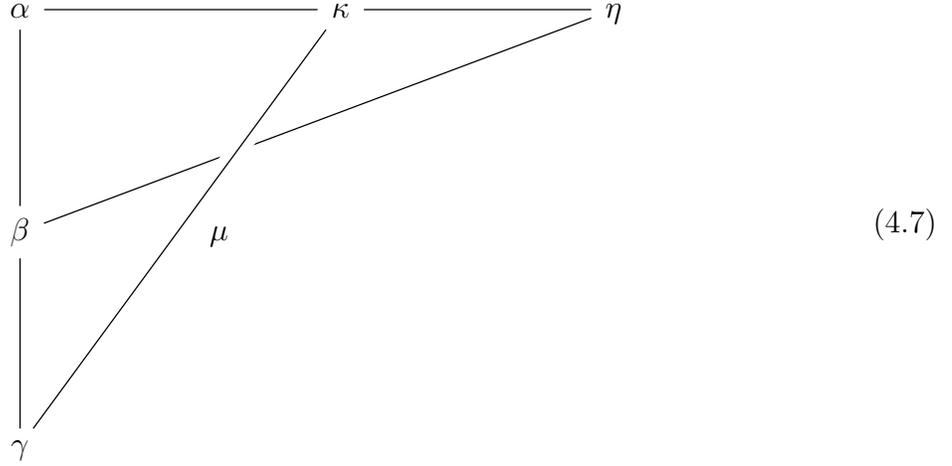
    \begin{equation}\label{GCReductMultGroup}
 	\begin{tikzcd}
 	\alpha \arrow[ddd, dash] \arrow[rrr, dash] &&& \kappa \arrow[rrr, dash]  &&& \eta \arrow[dddllllll,  dash] 
  \\
&&&&&&
\\
& & && %\arrow[dll, dash]  
%\arrow[ddddll, dash]  
&&
\\
 \beta   \arrow[ddd, dash] && \mu  &&&& \\ &&&&&&
 \\
 &&&&&&
 \\
  \gamma  \arrow[uuuuuurrr, crossing over, dash] &&&&&&
  	\end{tikzcd}
 	\end{equation}
\end{claim}

\begin{claimproof}

As the tuples in Diagram \ref{GCReductMultGroup} satisfy condition (\ref{tangentConditionsMultGroup}), by Lemma \ref{slopeDefinable}
$$a\in \acl_{\mathbb K} (\alpha),$$ moreover 
$$1\geq \dim_{\mathbb K}(\alpha)=\dim_{\mathbb K} (a,\alpha)\geq \dim_{\mathbb K} (a)=1.$$

Then $\dim_{\mathbb K}(\alpha/a)=0$ so $ \alpha\in \acl_{\mathbb K}(a)$. Therefore, $a$ and $\alpha$ are $\mathbb K$-interalgebraic. The same is true with all the corresponding tuples in both diagrams, so we conclude that the Diagram \ref{GCReductMultGroup} is interalgebraic (in $\mathbb K$) with the Diagram \ref{gcFullStructureMultGroup}.

Therefore, the rank computed in $\mathbb K$ of a tuple of elements of Diagram  \ref{GCReductMultGroup} is the same as the rank of the corresponding tuple in Diagram \ref{gcFullStructureMultGroup}.  As $\mathcal M$ is a reduct of $\mathbb K$ the rank of tuples computed on $\mathcal M$ is at least the dimension of tuples computed in $\mathbb K$. So to prove our claim it is enough to prove that the rank computed in $\mathcal M$ does not increase. For doing so we have to prove that there are still algebraic relations between the tuples lying on the same line of the diagram. 

Let us prove that
$$\eta\in \acl_{\mathcal G}(\alpha,\kappa)$$

If $ X_\tau\circ X_\eta\cap X_\alpha\circ X_\kappa$ is infinite we proceed in exactly the same way as we did in the proof of Theorem \ref{thmAditiveVersion}, so we may assume that it is finite.

In this case the algebraic relation follows from the  conclusion of Lemma \ref{lemmaIntersectionSubeVersionMultGroup}: 

There are infinitely many $s\in Y$ such that:

$$| X_\tau\circ X_\eta\cap X_\alpha\circ X_\kappa|<| X_\tau\circ X_s\cap  X_\alpha\circ X_\kappa|.$$

As $ X_\alpha\circ X_\kappa$ is a $\mathcal M$-definable set with parameters $\alpha,\kappa$, the set

$$E:=\{s\in Y: | X_\tau\circ X_\eta\cap  X_\alpha\circ X_\kappa|<| X_\tau\circ X_s\cap  X_\alpha\circ X_\kappa|\}$$ is also $\mathcal M$-definable with parameters $\alpha,\kappa$. As $E$ is infinite and $Y$ is strongly minimal, $Y\setminus E$  is finite and as $\eta\in Y\setminus E$ we conclude that 
$$\eta\in \acl_{\mathcal M}(\alpha,\kappa).$$

%If $$X_\tau\circ X_\eta\cap Z(\alpha_0,\alpha_1,\kappa)$$ is infinite, in particular $\cl(X_\tau\circ X_\eta)\cap Z(\alpha_0,\alpha_1,\kappa)$ is infinite, by Claim \ref{genericFiniteIntersectionClaim} the family $(\cl(X_\tau\circ X_a))_{a\in Y}$ has generic finite intersection, so Lemma \ref{noGenericFiniteZLemma} implies that $\eta$ is not $\mathbb K$-generic over $(\alpha_0,\alpha_1,\kappa)$, in particular 
%$$\left\{b\in Y:X_\tau \circ X_b\cap Z(\alpha_0,\alpha_1,\kappa)\text{ is infinite}\right\}$$ does not contain any $\mathbb K$-generic element so it is finite. Since it is $\mathcal G$-definable and contains $\eta$ we conclude $$\eta\in \acl_{\mathcal G}(\alpha_0,\alpha_1,\kappa).$$

 The rest of algebraic relations needed for the definition of group configuration have analogous proofs \end{claimproof}

 This finishes the proof of Proposition \ref{existsGroupConfigurationMult}.
\end{proof}

As a Corollary we have:

\begin{prop}\label{propExistsGroupMult}
    If $Y$ has finitely many slopes at $(1,1)$, there is $(G,\oplus)$ a $\mathcal M$-interpretable group of Morley rank $1$ and $(\alpha_0,\beta_0,\kappa_0)\in G^3$ and $(a_0,b_0,b)\in \mathbb G_a^3$ $\mathbb K$-generics such that $(\alpha_0,\beta_0,\alpha_0\+\beta_0)$ is $\mathbb K$-inter algebraic with $(a_0,b_0,a_0+b_0)$.
    \end{prop}

    \begin{proof}
        It is a consequence of Proposition \ref{existsGroupConfigurationMult} and Fact \ref{groupConfig}.
    \end{proof}

\subsection{Finding a group locally isomorphic to $\mathbb G_a$}

Remember that for us an infinite $\mathbb K$-definable Abelian group $G$ is locally isomorphic to $\mathbb G_a=(K,+)$ if there are $A\leq \mathbb G_a$ a $\mathbb K$-definable subgroup and $i:A\to G$ a $\mathbb K$-definable injective group homomorphism such that $G/i(A)$ is finite.

In this sub-section we describe a way to use the group that we found in Sub-section \ref{findingAGroupConfiguration} to find a group locally isomorphic to $\mathbb G_a$.

\begin{thm}\label{gcIsLocallyAdditiveThm}
Let $\mathbb K=(K,+,\cdot,v,\Gamma)$ be a saturated model of ACVF and let $K_0\leq K$ be a $2^{\omega}$-saturated sub-model.  Let $\mathcal M=(M,\cdot,\ldots)$ be a strongly minimal and not locally modular structure expanding the multiplicative group whose definable sets are all $\mathbb K$-definable. Assume that there is a $\mathcal M$-interpretable Abelian group $(G,\oplus)$ over $K_0$ and $\mathfrak g_0=(\alpha_0,\beta_0,\alpha_0\+\beta_0)\in G^3,$ a tuple of pairwise $\mathbb K$-generic independent elements. Assume as well that there are $(a_0,b_0)\in \mathbb G^2_a$ such that $\mathfrak g_a=(a_0,b_0,a_0+b_0)$ is a tuple of pairwise $\mathbb K$-generic independent elements that, moreover, is $\mathbb K$-interalgebraic with $\mathfrak g_0$ over $K_0$. 

Then there is a finite subgroup $F\leq G$ such that  $(G/F,\oplus)$ is locally isomorphic to $\mathbb G_a$.
%Let $(H,\oplus)$ be a one dimensional $\mathbb K^{eq}$-definable group. Assume that there are $\alpha,\beta\in H$ such that $\tp(\alpha/K_0)$ and $\tp(\beta/K_0)$ are both wide and $a,b\in \mathbb G_a$ are generic independent such that $\acl(a K_0)=\acl(\alpha K_0)$, $\acl(bK_0)=\acl(\beta K_0)$ and $\acl((a+b)K_0)=\acl((\alpha\oplus\beta) K_0)$. 
%Then, there is a finite subgroup $F\leq I$, a $\mathbb K$-definable subgroup $A\leq \mathbb G_a$ and a $\mathbb K^{eq}$-definable injective group homomorphism $\phi:A\to I/F$ such that $\phi(A)$ has finite index on $I/F$.
\end{thm}
%From now on we fix $K_0$, $G$ $\mathfrak g_0$ and $\mathfrak g_a$ as in the hypothesis of Theorem \ref{gcIsLocallyAdditiveThm}.

%We will need the followting technical result:

%\begin{lemma}
 %   Let $(G,\oplus)$ be a $\mathbb K$-definable group let $D\leq G$ be a $\mathbb K$-type definable subgroup of bounded index on $G$. Then for all $(\alpha_0,\beta_0,\gamma_0)\in G$ there is some $c\in  G$ having the same type as $a$ over $K_0$ such that $a\ominus c\in D$.
%\end{lemma}

%\begin{proof}
    
%\end{proof}

For proving Theorem \ref{gcIsLocallyAdditiveThm} we need some preliminary results, first we state a well known fact that follows, for example, from Theorems 10.4 (b) and 10.2 of \cite{wagon}.

\begin{fact}\label{haarMeasureFact}
If $G$ is an Abelian group, there is a finite and finitely additive invariant measure on the subsets of $G$. 
\end{fact}

As we are working on an NIP theory, the following is an instance of Theorem 2.19 of \cite{montenegroOnshuus}.

\begin{fact}\label{algebraicGrouChunkMontenegro}
    Let $(G,\oplus)$ be a $\mathbb K$-definable Abelian group and let $\mu_0$ be as provided by Fact \ref{haarMeasureFact}. Then there is a type definable subgroup $D$ of $G$ such that $\mu_0(X)>0$ for all definable $X\supseteq D$ (we say then that $D$ is wide), an algebraic group $(H,+_H)$ and a $\mathbb K$-definable finite-to-one group homomorphism $\phi:D\to H$. 
\end{fact}

Now we present the proof of Theorem \ref{gcIsLocallyAdditiveThm}.

\begin{proof}(Proof of Theorem \ref{gcIsLocallyAdditiveThm})

Let $G$, $\mathfrak g_0$ and $\mathfrak g_a$ as in the hypothesis of Theorem \ref{gcIsLocallyAdditiveThm}.

We will use Fact \ref{algebraicGrouChunkMontenegro} but for this we need to ensure that $G$ is $\mathbb K$-definable, so we prove:

\begin{claim}
    $G$ is in $\mathbb K$-definable bijection with a $\mathbb K$-definable set.
\end{claim}

\begin{claimproof}
As $G$ is $\mathcal M$-interpretable and any strongly minimal structure eliminates imaginaries except by finite, (Lemma 8.2.9 of \cite{marker}) we may assume that $G=W/E$ with $W$ $\mathcal M$-definable and $E$ a $\mathcal M$-definable equivalence relation on $W$ with finite classes. 

But then $W$ and $E$ are also $\mathbb K$-definable and as any field eliminates finite imaginaries (Lemma 3.2.16 of \cite{marker})  $W/E$ is in $\mathbb K$-definable bijection with some $\mathbb K$-definable set. 
\end{claimproof}
So we may assume that $G$ is $\mathbb K$-definable, then we may apply Fact \ref{algebraicGrouChunkMontenegro} and find an algebraic group $(H,+_H)$, a wide $\mathbb K$-type definable subgroup $D\subseteq G$ and $\phi:D\to H$, a $\mathbb K$-definable finite-to-one group homomorphism.
Now we prove:
\begin{claim}
In this situation:
\begin{enumerate}

    \item There is some small model $K_1$ extending elementary to $K_0$ and 

    $$\mathfrak g'_a=(a'_0,b'_0,a'_0+b'_0)\in \mathbb G_a^3$$
    and 
    $$\mathfrak g_H=(\alpha_H,\beta_H,\alpha_H +_H \beta_H)\in H^3,$$

which are inter-algebraic over $K_1$ and, in addition, the elements of each tuple are pairwise generic independent over $K_1$.

    \item $H$ is $\mathbb K$-definably isomorphic to $\mathbb G_a$.
\end{enumerate}
\end{claim}
\begin{claimproof}
    
Let us prove Clause 1:

As $D$ is wide in $G$ it has bounded index so $|G/D|\leq 2^{\omega}$. For each $s\in G/D$ let $c_s\in G$ be a representative for $s$, then  $K_0\cup \{c_s:s\in G/D\}$ is small so there is some small model $K_1$ containing it.

Let $\pi:G\to G/D$ be the canonical projection.

Then we may take $\alpha',\beta',\alpha'\oplus \beta'\in G$ realizing a type over $K_1$ extending $\tp(\alpha_0,\beta_0, \alpha_0\+\beta_0/K_0)$ such that $(\alpha',\beta')$ is $\acl$ generic over $K_1$. Then, as $(\alpha_0,\beta_0, \alpha_0\+\beta_0)$ is $\mathbb K$-inter-algebraic with $(a_0,b_0,a_0+b_0)$ there is some tuple $(a'_0,b'_0)\in \mathbb G_a^2$ such that $(\alpha',\beta',\alpha'\+\beta')$ is $\mathbb K$-inter algebraic over $K_1$ with $(a'_0,b'_0,a'_0+b'_0)$.

Let $\alpha_D:=\alpha'\ominus c_{\pi(\alpha')}$ and $\beta_D:=\beta'\ominus c_{\pi(\beta')}$.

By choice of $c_s$ it follows that $\alpha_D,\beta_D\in D$.

As $(\alpha_D,\beta_D,\alpha_D\oplus\beta_D)$ is $\mathbb K$-inter algebraic over $K_1$ with $(\alpha',\beta',\alpha'\oplus\beta')$, it is also inter-algebraic with $(a'_0,b'_0,a'_0+b'_0)$. Taking $\alpha_H=\phi (\alpha_D)$ and $\beta_H=\phi(\beta_D)$ we get the result.

For proving Clause 2 notice that by Fact \ref{aclIgualFact} the algebraic relations between $g'_a$ and $g_H$ occurs on ACF. By strong minimality of ACF, the Zariski closure of $H$ is an algebraic group such that for any two generic independent $(\rho,\eta)\in H^2$ there are $(x,y)\in \mathbb G_a^2$ such that $(\rho,\eta,\zeta +_H \eta)$ is ACF-inter-algebraic with $(x,y,x+y)$. This implies\footnote{This may be a well known fact but we still include a proof after the proof of Theorem \ref{gcIsLocallyAdditiveThm}.}that $H$ is definably isomorphic in ACF with $\mathbb G_a$ \end{claimproof}
 
From now on we may assume that $(H,+_H)=(\mathbb G_a,+)$. Now we show that we can reduce to a definable group and not just a type definable one. 
\begin{claim}
    There is $C\leq G$ a definable subgroup with $D\subseteq C$ and $\phi':C\to \mathbb G_a$ a finite-to-one definable group homomorphism extending $\phi$.
\end{claim}
\begin{claimproof}
    We follow the proof of Proposition 4.7 of \cite{acosta}.
By compactness there are $\mathbb K$-definable sets $U_1$ and $U_0$ with $U_1=U_1^{-1}$, $D\subseteq U_1\subseteq U_0\subseteq H$ and $U_1 \oplus U_1\subseteq U_0$ such that the formula defining $\phi$ also defines a function $\phi:U_0\to \mathbb G_a$ that is still finite to one and respects the group operation.

Moreover, as $D$ is a group and $\phi$ is a finite-to-one group homomorphism, all the fibers have the same size, say $m$. So by compactness we can also choose $U_1$ such that $\phi \arrowvert_{U_1}$ is $m$-to-one.

By compactness and the fact that each type definable subgroup of $\mathbb G_a$ is an intersection of definable groups (Proposition 36 of \cite{acostaACVF}), there is some definable group $B\subseteq \mathbb G_a$ such that $\pi_2(C)\subseteq B\subseteq \pi_2(U_1)$.

Take $$C:=\phi^{-1}(B)\cap U_1.$$ 
Clearly $C$ is definable, contains $D$ and $\phi$ is defined on $D$ and respects the group operation. So we only have to prove that $C$ is a subgroup of $G$. If $a\in C$ then $a^{-1}\in U_1$ and $\phi(a^{-1})=-\phi(a)\in B$ so $a^{-1}\in C$. Now let $a,b\in C$ so $a\+ b \in U_0$, moreover $\phi(a\+ b)=\pi(a)+\pi(b)$. As $B$ is a subgroup containing $\phi(a)$ and $\phi(b)$ it follows that $\phi(a\+ b)\in B$. In particular, as $B\subseteq \phi(U_1)$, there is some $c\in U_1$ such that $\phi(a\+ b)=\pi_2(c)$. So $a\+ b$ is one of the $m$ preimages (via $\phi$) of $\phi(c)$ so it follows from our conditions on $U_1$ that $a\+ b\in U_1$ so $a\+ b\in C$. \end{claimproof}

Let $F=\ker \phi$ and $A=\phi(C)$ so $\phi$ induces a finite-to-one group isomorphism $C/F\to  A$, thus the inverse of $\phi$ is injective and its image is $C/F$. As $\mu_0(C)>0$ and $\mu_0(G)=1$ it follows that $C$ has finite index in $G$ so $C/F$ has finite index in $G/F$ and then $G/F$ is locally isomorphic to $\mathbb G_a$.
\end{proof}

\begin{remark}\label{finiteQuotientAdditiveIsAdditiveLemma}
    Let $F\leq \mathbb G_a$ be a finite subgroup, then $\mathbb G_a/F$ is $\mathbb K^f$-definably isomorphic to $\mathbb G_a$.
\end{remark}
\begin{proof}
     For any $a\in\mathbb G_a$ let $\langle a\rangle$ be the subgroup of $\mathbb G_a$ generated by $a$.
     By induction it is enough to prove that if $a\in \mathbb G_a$ is such that $\langle a\rangle$ is finite, then $\mathbb G_a/ \langle a\rangle$ is $\mathbb K^f$-definable isomorphic to $\mathbb G_a$.

     Notice that if $\text{char} K=0$ then $\langle a\rangle$ is infinite so we may assume that $\text{char}(K)=p>0$.

     Define $$\phi:\mathbb G_a\to \mathbb G_a$$ the $\mathbb K^f$-definable group homomorphism given by $$\phi(x)= (a^{-1}x)^{p}-(a^{-1}x).$$
     Then, $\phi$ is onto $\mathbb G_a$ and 
     $$\ker (\phi)=\left\{x:a^{-1}x\in \{0,\ldots,p-1\}\right\}=\langle a \rangle.$$\end{proof}
     %Let $N$ be the minimum natural number such that $\underbrace{a+a+\ldots +a }_{\text{$N$ times}}=0$ and 

\begin{remark}
    Let $(H,+_H)$ be an infinite $\mathbb K^f$-definable Abelian group and assume that for each $(\rho,\eta)\in H^2$ generic there is $(x,y)\in \mathbb G^2_a$ such that $(\rho,\eta,\rho +_H\eta)$ is $\mathbb K^f$-interalgebraic with $(x,y,x+y)$. Then $H$ is $\mathbb K^f$-definably isomorphic to $\mathbb G_a$.
\end{remark}    

\begin{proof}
    We will work in the pure field structure $\mathbb K^f$ which is strongly minimal. Consider the $\mathbb K^f$ definable group $\mathbb G_a\times H$ and let $\pi_1:\mathbb G_a\times H\to \mathbb G_a$ and $\pi_2:\mathbb G_a\times H\to H$ be the projection on the first and second coordinate respectively.
    By strongly minimality $D:=\stab^{f}(\rho,x)\leq \mathbb G_a\times H$ is a $\mathbb K^f$-type definable subgroup of $\mathbb G_a\times H$ such that for all but finitely many $a\in \mathbb G_a$ and for all but finitely may  $\alpha\in H$ the fibers $\pi_2^{-1}(\alpha)\cap D$  and $\pi_1^{-1}(a)\cap D$ are both non-empty and finite. In particular $H$ is strongly minimal in $\mathbb K^f$. 
    Again, by strong minimality of $\mathbb K^f$, each type definable group is an intersection of $\mathbb K^f$-definable groups. In particular, there is $C\leq \mathbb G_a\times H$ a $\mathbb K^f$-definable subgroup such that the fibers of both projections are finite and non-empty.

    Let $F_1:=\pi_1(\ker(\pi_2\arrowvert_C))$. This is a finite subgroup of $\mathbb G_a$. By Remark \ref{finiteQuotientAdditiveIsAdditiveLemma} $\mathbb G_a/F_1$ is $\mathbb K^{f}$ definable isomorphic to $\mathbb G_a$ so we may replace $\mathbb G_a$ by this quotient and assume that $\pi_2\arrowvert_C$ is injective. Then $\pi_2:C\to H$ is an isomorphism with its image, which is cofinite but as it is a subgroup of $H$ it has to be equal to $H$ so $\pi_2:C\to H$ is onto. In the same way, $\pi_1:C\to \mathbb G_a$ is onto. Then if we call $F:=\pi_2(\ker(\pi_1\arrowvert_C))$ then $H/F$ is isomorphic to $\mathbb G_a$ via $(\pi_2\arrowvert_C)^{-1}\circ \pi_1$. 

   If $H$ is an elliptic curve then it is complete so $H/F$ is also complete but $\mathbb G_a$ is not complete so $H$ can not be an elliptic curve. Moreover, if $H$ is the multiplicative group then, by a proof similar to the proof of Remark \ref{finiteQuotientAdditiveIsAdditiveLemma}, it follows that $H/F$ is isomorphic to $\mathbb G_m$ which is not isomorphic to $\mathbb G_a$ as the first has torsion of any order and the second only has torsion of order multiples of $p$. So $H$ can not be isomorphic to $\mathbb G_m$. As the only $\mathbb K^f$ definable strongly minimal groups up to $\mathbb K^{f}$ definable isomorphism are the elliptic curves $\mathbb G_m$ and $\mathbb G_a$, it follows that $H$ is $\mathbb K^f$-definably isomorphic to $\mathbb G_a$.
\end{proof}

%Ts the conhe main tool we will use in this Section the following Proposition.  Theorem 2.19 from \cite{montenegroOnshuus}. Here we present a statement slightly different but the proof is the same.
%Throughout this sub-section we fix $\mathbb K_0=(K_0,+,\cdot,0,1,v,\Gamma)$ a $\omega$-saturated model of ACVF
%and $\mathbb K$ a saturated elementary extension of $\mathbb K_0$. We take algebraic closure ($\acl$) in $\mathbb K^{eq}$ and whenever we say `generic' we mean generic respect to $\acl$ in $\mathbb K^{eq}$. 
\section{Infinitely many slopes at $(1,1)$}\label{findingAField}

In this Section we prove Theorem \ref{thmMultiplicativeVersion} assuming that there is some $\mathcal M$-definable $Y\subseteq M\times M$ with infinitely many slopes at $(1,1)$. More precisely, we prove:

\begin{prop}\label{propMultiplicativeUnderAssumption}
    Assume that there is some $\mathcal M$-definable strongly minimal set $Y\subseteq M\times M$ having infinitely many slopes at $(1,1)$ witnessed by some polynomials $L_i$ and $L_{i,n_i}$, then the structure $\mathcal M$ interprets an infinite field. 
\end{prop}

The assumption that $Y$ has infinitely many slopes at $(1,1)$  allows us to use the family of translates of $Y$, $(t_a(Y))_{a\in Y}$ to interpret a field similarly as we did in Chapter \ref{additiveCaseVersion}. 

From now on in this section we fix $Y\subseteq M\times M$, a $\mathcal M$-definable set having infinitely many slopes at $(1,1)$ witnessed by some polynomials $L_i$ and $L_{i,n}$. For $\alpha\in Y$ we define:

\begin{equation}
    X_\alpha:=t_\alpha(Y).
\end{equation}

In Chapter \ref{additiveCaseVersion} our definition of $X_\alpha$ was more complicated because we needed to ensure that the family has infinitely many derivatives at $(0,0)$. Now we are making this assumption at $(1,1)$ for the family of translates, so we can to prove a statement analogous to Lemma \ref{existsHLemma} but for this simpler family of curves:

\begin{lemma}\label{existsHLemmaMult}
    Assume that $\mathbb K$ is complete and $Y\subseteq K\times K$ has infinitely many slopes at $(1,1)$ witnessed by some polynomials $L_i(x,y)$ and $L_{i,n}(x,y)$. Then, there is some open set $U\subseteq K$ containing $0$, $h:U\to K$ and $H:U\times U\to K$ analytic functions such that:

\begin{enumerate}

    \item $H(1,s)=1$ for all $s\in U$.

    \item For all $s\in U$, $\bar s:=(s,h(s))\in Y$.

     \item For all $s\in U$ there is a neighborhood $U_s\subseteq K$ of $1$ such that $(x,H(x,s))\in X_{\bar s}$ for all $x\in U_s$.

     \item The set $$s_1(H):=\left\{\frac{\partial H}{\partial x}(1,s):s\in U\right\}$$ is infinite.

     \item For each $a\in U$ and $b=(b_1,b_2)\in X_{\bar a}$ there are neighborhoods $U_b$ of $b_1$ and $V_a$ of $a$, an analytic function $\Phi(x,s)$ defined in $U_b\times V_a$ and a natural number $n$ such that for all $(x,s)\in U_b\times V_a$ one has that $(x,\text{Fr}^{-n}(\Phi(x,s)))\in X_{\bar s}$.
\end{enumerate}
\end{lemma}

\begin{proof}
 As $Y$ satisfies condition $(\star)$ witnessed by $L_{i}(x,y,\bar d)$ and $L_{i,n}(x,y,\bar d)$ one has that:

    $$Y=\bigcup_i\left\{(x,y)\in V_i: L_i(x,y,\bar d)=0\right\},$$ where each $V_i$ is some open subset of $K\times K$ intersecting the set of zeros of $L_i(x,y,\bar d)$, $L_1(1,1)=0$ and 
    $$\frac{\partial L_1}{\partial y}(1,1)\neq 0.$$ 
    By Fact \ref{implicitFunctionTheoremf} there is an analytic function $h(x)$ converging in a neighborhood $U$ of $1$ such that $h(1)=1$ and $L_1(x,h(x))=0$ for all $x\in U$. Shrinking $U$ we may assume that for all $x\in U$, $(x,h(x))\in V_1$, therefore for all $x\in U,$  $(x,h(x))\in Y$. 

    Moreover, for $s\in U$ the function $h_s(x)=h(x\cdot s)\cdot(h(s))^{-1}$ is analytic in a neighborhood of $1$ and its graph is contained on $t_{(s,h(s))}(Y)$ so we take 
    $$H(x,s)=\frac{h(x\cdot s)}{h(s)}$$ in order to get conditions 1-3. As
    $$\frac{\partial H}{\partial x}(1,s)=h_s'(1),$$ condition $4$ follows from Lemma \ref{hInfDerivadasLemma}.

Only condition 5 is missing. For this let $a\in U$ and let $(b_1,b_2)\in t_{\bar a}(Y)$, thus $$c=(c_1,c_2):=(a\cdot b_1,h(a)\cdot b_2)\in Y.$$

By the property $(\star)$ on $Y$, as $(c_1,c_2)\in Y$, there is some natural number $n$ and an irreducible polynomial $L_{i,n}$ such that: 
$$\displaystyle L_{i,n}\left(c_1,c_2^{p^n}\right)=0, $$   
   $$ \displaystyle\frac{\partial L_{i,n}}{\partial y}\left(c_1,c_2^{p^n}\right)\neq 0 $$

and for all $(x,y)$ in some neighborhood $V_i$ of $(c_1,c_2)$, if $L_{i,n}(x,y^{p^n})=0$ then $(x,y)\in Y$. 

By Fact \ref{implicitFunctionTheoremf} there is an analytic function  $g$ defined in some neighborhood $U_1$ of $c_1$ such that 

$$g(c_1)=c^{p^n}_2$$

%$\text{Fr}^{-n_2}(g_2(z))=y_2$, 

and  
$$L_{i,n}(x,g(x))=0$$
for all $x\in U_1$ so that 
$$L_{i,n}(x, (\text{Fr}^{-n}(g(x)))^{p^n})=0.$$

Shrinking $U_1$ one may assume that for all $x\in U_1$, 
$$(x,\text{Fr}^{-n} \circ g(x))\in V_1,$$ 
so 
$$(x,\text{Fr}^{-n} \circ g(x))\in Y$$
for all $x\in U_1$ so that the graph of $\text{Fr}^{-n}\circ g$ is  contained in $Y$.

Thus, for all $x\in U_1$ and for $s$ close enough to $a$, $$(x\cdot s^{-1}, \text{Fr}^{-n} \circ g(x) \cdot h(s)^{-1})\in Y_{\bar s}$$
so the graph of $x\mapsto \text{Fr}^{-n}( g(x\cdot s))\cdot h(s)^{-1}$ is contained on $t_{\bar s}(Y)$

Thus, we can take this $n$ and $$\Phi(x,s):=g(x\cdot s)\cdot \text{Fr}^{n}\left(h(s)^{-1}\right)$$ to get Clause 5 of the statement.\end{proof}

This Lemma allows us to construct a group configuration for $\mathcal M$ in a very similar way as in Section \ref{theProof}. The reason why the construction of the additive case also works for the multiplicative case is the following observation:

\begin{lemma}\label{multIsAdditionInDerivativeLemma}
Assume that $\mathbb K$ is metric. Let $U\ni 1$ be open and $h:U\to K$ and $g:U\to K$ be analytic functions such that $h(1)=g(1)=1$, then:
$$(h\cdot g)'(1)=h'(1)+g'(1),$$ $$(h/g)'(1)=h'(1)-g'(1)$$and $$(h\circ g)'(1)=h'(1)\cdot g'(1).$$

\end{lemma}

\begin{proof}
    By Leibniz rule we have that:
$$(h\cdot g)'(1)=h'(1)\cdot g(1) + h(1)\cdot g'(1)=h'(1)+g'(1).$$

Using again Leibniz rule we get:

$$(h/g)'(1)=\frac{h'(1)g(1)-h(1)g'(1)}{g(1)^2}=h'(1)-g'(1).$$

Finally, by chain rule we get:
$$(h\circ g)'(1)=h'(g(1))\cdot g'(1)=h'(1)\cdot g'(1).$$\end{proof}

From now on the proof is very similar to the proof of Theorem \ref{thmAditiveVersion}.

Remember from Definition \ref{defidL1} that for $L(x,y)$ a polynomial and $a\in K\times K$ such that $L(a)=0\neq \frac{\partial L}{\partial y}(a)$ we set:
$$\mathfrak d(L,a):=\frac{\frac{-\partial L}{\partial x}(a_1,a_2)}{\frac{\partial L}{\partial y}(a_1,a_2)}\frac{a_1}{a_2}.$$

First, we prove a first order statement (in the language of valued fields) analogous to Lemma \ref{lemmaIntersectionSubeVersion}:

\begin{lemma}\label{lemmaIntersectionSuberVersionMultiplicative}
Let $\phi(x,y,\bar e)$ be a formula with parameters $\bar e$, let $L_i(x,y,\bar e)$ and $L_{i,n}(x,y,\bar e)$ be polynomials. 

For all $\bar d$, if 
$$Y=\{(x,y)\in K^2: \mathbb K\models \phi(x,y,\bar d)\}$$ has infinitely many slopes at $(1,1)$ witnessed by $L_i(x,y,\bar d)$ and $L_{i,n}(x,y,\bar d)$,  then if for $a \in Y$ we define

$$X_a:=t_a(Y)$$

There is an open ball $B$ such that for all
$t\in B$ and for all $a_0$, $b_0$, $b$, $a_1$, $b_1\in K$, if 
     $$\mathfrak g_f:=(t,t a_1, \ t+a_0 , \ tb_1,\  t+b_0,\  t a_1 b_1 ,\  t+a_1b_0+a_0,\  t+b, \ t+a_1 b+a_0,\  t+a_1 b_1 b + b_1 a_0 +b_0)\in B^{10}$$
 
     there is $\mathfrak g=(\tau,\alpha_0,\alpha_1,\beta_0,\beta_1,\gamma_0,\gamma_1,\kappa,\eta,\mu)\in Y^{10}$ such that:
   \begin{equation}\label{tangentConditionsMult}\tag{TM}
   % \begin{array}{lll}       
    %    X_\tau\tangent r_{t},&& \\
     %    X_{\alpha_1}\tangent r_{ta_1},&\ X_{\alpha_0}%%\tangent r_{t+a_0},&\\ 
     %    X_{\beta_1}\tangent r_{tb_1},&\ X_{\beta_0}\tangent r_{t+b_0},&\\ 
      %   X_{\gamma_1} \tangent r_{ta_1b_1},&\  X_{\gamma_0}\tangent r_{t+a_1b_0+a_0},&\\
       % X_\kappa\tangent r_{t+b},&\ X_\eta\tangent r_{t+a_1b+a_0},&\ X_\mu \tangent r_{t+a_1b_1b+b_1a_0+b_0}. 
    %\end{array}
    f(a):=\mathfrak d(L_1,a) \text{ maps the }i\text{'th coordinate of }\mathfrak g\text{ into the }i\text{'th coordinate of }\mathfrak g_f.
    \end{equation}

    Moreover, if for $r_0,r_1,r_2\in Y$ we define: 
        $$Z(r_0,r_1,r_2):=(X_\tau\circ X_{r_0})\cdot (X_{r_1}\circ X_{r_2})\menosmult(X_\tau\circ X_{r_1})$$ 
        then, for all $a,r_0,r_1,r_2$ coordinates of $\mathfrak g$ we have:
        \begin{itemize}
    \item If $tf(a)=tf(r_0)+f(r_1)f(r_2)-tf(r_1)$
       and $X_\tau\circ X_a\cap Z(r_0,r_1,r_2)$ is finite, there are infinitely many $s\in Y$ with:
     $$| X_\tau\circ X_a\cap Z(r_0,r_1,r_2)|<| X_\tau\circ X_s\cap Z(r_0,r_1,r_2)|.$$
    \item If $tf(a)=f(r_1)f(r_2)$
    and $X_\tau\circ X_a\cap X_{r_1}\circ X_{r_2}$ is finite, there are infinitely many $s\in Y$ with:
    $$|X_\tau\circ X_a\cap X_{r_1}\circ X_{r_2}|<|X_\tau\circ X_s\cap X_{r_1}\circ X_{r_2}|.$$
    \end{itemize}

   % \begin{itemize}
    %    \item  If $X_\tau\circ X_\eta\cap Z(\alpha_0,\alpha_1,\kappa)$ is finite, then there are infinitely many $s\in Y$ with:

     %$$| X_\tau\circ X_\eta\cap Z(\alpha_0,\alpha_1,\kappa)|<| X_\tau\circ X_s\cap Z(\alpha_0,\alpha_1,\kappa)|.$$
      % \item If $X_\tau\circ X_\mu\cap Z(\gamma_0,\gamma_1,\kappa)$ is finite, then there are infinitely many $s\in Y$ with:

     %$$|X_\tau\circ X_\mu\cap Z(\gamma_0,\gamma_1,\kappa)|<| X_\tau\circ X_s\cap Z(\gamma_0,\gamma_1,\kappa)|.$$
     
     %\item  If $X_\tau\circ X_\mu\cap Z(\beta_0,\beta_1,\eta)$ is finite,  there are infinitely many $s\in Y$ with:

    %$$|X_\tau\circ X_\mu\cap Z(\beta_0,\beta_1,\eta)|<| X_\tau\circ X_s\cap Z(\beta_0,\beta_1,\eta)|.$$

     %\item If $X_\tau \circ X_{\gamma_1}\cap X_{\alpha_1}\circ X_{\beta_1}$ is finite, then there are infinitely many $s\in Y$ with:

    %$$|X_\tau \circ X_{\gamma_1}\cap X_{\alpha_1}\circ X_{\beta_1}|<| X_\tau\circ X_s\cap X_{\alpha_1}\circ X_{\beta_1}|.$$ 

    %\item If $X_\tau \circ X_{\gamma_0}\cap X_{\alpha_1}\circ X_{\beta_1}$ is finite, then there are infinitely many $s\in Y$ with:

    %$$|X_\tau \circ X_{\gamma_0}\cap X_{\alpha_1}\circ X_{\beta_1}|<| X_\tau\circ X_s\cap X_{\alpha_1}\circ X_{\beta_1}|.$$
     
    %\end{itemize}

%there is some open ball $B$ such that for all $a\in B$ there is some $\alpha \in Y$ such that $t_{\alpha}(Y) \tangent r_a$
    
\end{lemma}
\begin{proof}
    It is very similar to the proof of Lemma \ref{lemmaIntersectionSubeVersion}.

       As the statement of Lemma is a first order assertion on the parameter $\bar d$, it is enough if we prove it for some model of ACVF. In particular, without loss of generality we may assume that $\mathbb K$ is complete.

    Fix $U, h$ and $H$ as provided by Lemma \ref{existsHLemmaMult}. For $s\in U$ call $$d(s):=\frac{\partial H}{\partial x}(1,s).$$
By Corollary \ref{existsInf} and Clause 4 of Claim \ref{existsHLemmaMult}, there is some open ball $B$ contained on $s_1(H)$. Let $t\in B$ and $a_0$, $b_0$, $b$, $a_1$, $b_1\in K$ such that
     $$\{t a_1, \ t+a_0 , \ tb_1,\  t+b_0,\  t a_1 b_1 ,\  t+a_1b_0+a_0,\  t+b, \ t+a_1 b+a_0,\  t+a_1 b_1 b + b_1 a_0 +b_0\}\subseteq B.$$

By definition of $B$ we can find  $\tau',\ \alpha'_0,\ \beta'_0,\ \gamma'_0,\ \kappa',\ \eta'$ and $\mu'$ elements of $U$ such that:

\begin{itemize}
    \item $d(\tau')=t$
    \item $d(\alpha'_0)=t+a_0$,
    \item $d(\beta'_0)=t+b_0$,
    \item $d(\gamma'_0)=t+a_0b_1+b_0$,
       \item $d(\kappa')=t+b$,
    \item $d(\eta')=t+ a_0+a_1b$,
    \item $d(\mu')=t+b_0+a_0b_1+a_1b_1b$,
    \item $d(\alpha'_1)=ta_1$,
    \item $d(\beta'_1)=tb_1$ and
     \item $d(\gamma'_1)=ta_1b_1$.
\end{itemize}

Define $\tau=(\tau',h(\tau'))$, $\alpha_0=(\alpha'_0,h(\alpha'_0))$, $\alpha_1=(\alpha'_1,h(\alpha'_1))$, $\beta_0=(\beta'_0,h(\beta'_0))$, $\beta_1=(\beta'_1,h(\beta'_1))$, $\gamma_0=(\gamma'_0,h(\gamma'_0))$, $\gamma_1=(\gamma'_1,h(\gamma'_1))$, $\kappa=(\kappa',h(\kappa'_0))$, $\eta=(\eta',h(\eta'_0))$ and $\mu=(\mu',h(\mu'_0))$. 

By Lemma \ref{hInfDerivadasLemma} and Clause 3 of the list of properties for $H$ and $h$, this choice of 
$$(\tau,\alpha_0,\alpha_1,\beta_0,\beta_1,\gamma_0,\gamma_1,\kappa,\eta,\mu)$$ 
implies that the conditions (\ref{tangentConditionsMult}) of the statement holds.

Now we prove the first bullet of the `moreover' part of the statement:

Let $a,r_0,r_1,r_2$ such that 
$$tf(a)=tf(r_0)+f(r_1)f(r_2)-tf(r_1)$$ and assume that $ X_\tau\circ X_a\cap Z(r_0,r_1,r_2)$ is finite.

        As $H$ is a uniform analitification for $(X_a)_{a\in Y}$ at $(1,1)$,  $H_\tau(x,s)=H(H(x,s),\tau')$ is a uniform analitification for $(X_\tau \circ X_a)_{a\in Y}$ at $(1,1)$.
      
        By Clause 5 of Lemma \ref{existsHLemma} for all $\alpha\in U$ and all $b\in X_{\alpha}$, there is a uniform Fr-analitification of $(X_a)_{a\in Y}$ for $X_\alpha$ at $(1,1)$. As $\tau'\in U$, in particular, $X_\tau$ admits a Fr-analitification at any point. 
        
        So by Lemma \ref{lemmaCanCompose}, for all $\alpha \in U$ and all $b\in X_\tau \circ  X_\alpha$, there is a uniform Fr-analitification of $(X_\tau\circ X_a)_{a\in Y}$ for $X_\tau\circ X_\alpha$ at $b$.

       Moreover, the function 
       $$T(x) := H(H(x,r'_0),\tau')\cdot H(H(x,r_2'),r'_1)\cdot^{-1} H(H(x,r'_1),\tau')$$ is an analitification for $Z(r_0,r_1,r_2)$ at $(1,1)$. By Lemma \ref{multIsAdditionInDerivativeLemma}, $$T'(1)=tf(r_0) + f(r_1)f(r_2)-tf(r_1)$$

       so by hypothesis $T'(1)= tf(a)$

       Thus, $$\frac{\partial H_\tau}{\partial x}(1,a)=tf(a)=T'(1)$$

       so we apply Lemma \ref{intersectionSubeLemmaAnalitic} and conclude. 

  Second bullet has similar proof, this completes the proof of Lemma \ref{lemmaIntersectionSubeVersion}
\end{proof}

Now we present the proof of Proposition \ref{propMultiplicativeUnderAssumption}. It is very similar to the proof of Theorem \ref{thmAditiveVersion}.

\begin{proof}\textit{(Proof of Proposition \ref{propMultiplicativeUnderAssumption})}

Let $Y\subseteq M\times M$ be a $\mathcal M$-definable not $\mathcal M$-affine strongly minimal set with infinitely many slopes at $(1,1)$ witnessed by some polynomials $L_i$ and $L_{i,n}$.

Let $B\subseteq K$ an open ball provided by Lemma \ref{lemmaIntersectionSuberVersionMultiplicative}. For $a\in Y$ let

$$X_{a}=t_a(X).$$

Notice that $(X_a)_{a\in Y}$ is  $\mathcal M$-definable family. 

%We want to find a group configuration on $\mathbb K$ for the group $\mathbb G_a\ltimes\mathbb G_m$ and translate it inside $B\times B$. 
%Here $\mathbb G_a$ and $\mathbb G_m$ are the additive and multiplicative group of $\mathbb K$ respectively and the semidirect product is given by the action of $\mathbb G_m$ on $\mathbb G_a$ by multiplication.

Let $t\in B$. Given that addition and multiplication are continuous functions, there are $B_0$ and $B_1$ open balls around $0$ and $1$ respectively such that:

\begin{itemize}
  \item $tB_1\subseteq B,$
  \item $t+B_0\subseteq B,$
  \item $t B_1 B_1\subseteq B,$ 
  \item $t+B_1B_0+B_0\subseteq B,$ and
  \item $t+B_1B_1B_0+B_1B_0+B_0\subseteq B$.

\end{itemize}

As $B_0$ and $B_1$ are open balls, let $(a_0,b_0,b)\in B^3_0$ be a triple of $\mathbb K$-dimension $3$ over $t,d$, and let $(a_1,b_1)\in B_1^2$ be a tuple of $\mathbb K$-dimension $2$ over $a_0,b_0,b, t,d$. Then $(a_0,b_0,b,a_1,b_1)$ is a tuple of $\mathbb K$-dimension $5$ over $t,d$.

Then we have a standard group configuration for $\mathbb G_a\ltimes\mathbb G_m$ given by:

\begin{equation}\label{gcFullStructureMultiplicative}
 	\begin{tikzcd}
 	(a_0, a_1) \arrow[ddd, dash] \arrow[rrr, dash] &&& b \arrow[rrr, dash]  &&& a_0+a_1 b \arrow[dddllllll,  dash] 
  \\
&&&&&&
\\
& & && %\arrow[dll, dash]  
%\arrow[ddddll, dash]  
&&
\\
 (b_0,b_1)  \arrow[ddd, dash] &&& b_0+ a_0b_1+ a_1 b_1 b &&& \\ &&&&&&
 \\
 &&&&&&
 \\
 (a_0b_1+b_0,a_1b_1)  \arrow[uuuuuurrr, crossing over, dash] &&&&&&
  	\end{tikzcd}
\end{equation}

Moreover, by our choice of $B_0$ and $B_1$ we have that:

$$\{t a_1, \ t+a_0 , \ tb_1,\  t+b_0,\  t a_1 b_1 ,\  t+a_1b_0+a_0,\  t+b, \ t+a_1 b+a_0,\  t+a_1 b_1 b + b_1 a_0 +b_0\}\subseteq B.$$

Let $\tau,\ \alpha_0,\ \alpha_1,\ \beta_0,\ \beta_1,\ \gamma_0,\ \gamma_1,\ \kappa,\ \eta$ and $\mu$ be elements of $Y$ provided by Lemma \ref{lemmaIntersectionSubeVersion} 

%Therefore we can find $\alpha_0\in Y\cap O$ such that $d(\alpha_0)=t + a_0$. Note that $\alpha_0\in K\times K$ therefore $\alpha_0$ is itself an ordered pair of elements of $K$ but we will not use its coordinates so we will not give a name to them.

%In the same way we can find $\beta_0$ and $\gamma_0$ elements of $Y\cap O$ such that:

%\begin{itemize}
 %   \item $d(\alpha_0)=t+a_0$,
  %  \item $d(\beta_0)=t+b_0$ and
   % \item $d(\gamma_0)=t+a_0b_1+b_0$.
    
%\end{itemize}

%Similarly we find $p,q$ and $r$, elements of $Y\cap O$, such that:

%\begin{itemize}
 %   \item $d(p)=t+b$,
  %  \item $d(q)=t+ a_0+a_1b$ and
   % \item $d(r)=t+b_0+a_0b_1+a_1b_1b$.
%\end{itemize}

%And finally there are $\alpha_1,\beta_1$ and $\gamma_1$ elements of $Y\cap O$ such that:

%\begin{itemize}
 %   \item $d(\alpha_1)=ta_1$,
  %  \item $d(\beta_1)=tb_1$ and
   %  \item $d(\gamma_1)=ta_1b_1$.
%\end{itemize}

\begin{claim}\label{claimGCVersionMultiplicative}
    The following is a rank $2$ group configuration for the structure $\mathcal M$:

    \begin{equation}\label{GCReductMultiplicative}
 	\begin{tikzcd}
 	(\alpha_0, \alpha_1) \arrow[ddd, dash] \arrow[rrr, dash] &&& \kappa \arrow[rrr, dash]  &&& \eta \arrow[dddllllll,  dash] 
  \\
&&&&&&
\\
& & && %\arrow[dll, dash]  
%\arrow[ddddll, dash]  
&&
\\
 (\beta_0,\beta_1)   \arrow[ddd, dash] && \mu  &&&& \\ &&&&&&
 \\
 &&&&&&
 \\
  (\gamma_0,\gamma_1)  \arrow[uuuuuurrr, crossing over, dash] &&&&&&
  	\end{tikzcd}
 	\end{equation}
\end{claim}

\begin{claimproof}

As the tuples on Diagram \ref{GCReduct} satisfies conditions (\ref{tangentConditions}), by Lemma \ref{claimTangentFiniteLemma} one has that 
$$a_0\in \acl_{\mathbb K} (\alpha_0),$$ moreover 
$$1\geq \dim_{\mathbb K}(\alpha_0)=\dim_{\mathbb K} (a_0,\alpha_0)\geq \dim_{\mathbb K} (a_0)=1.$$

Then $\dim_{\mathbb K}(\alpha_0/a_0)=0$ so $ \alpha_0\in \acl_{\mathbb K}(a_0)$. Therefore, $a_0$ and $\alpha_0$ are $\mathbb K$-interalgebraic. The same is true with all the correspondent tuples in both diagrams, so we conclude that the Diagram \ref{GCReduct} is interalgebraic (in $\mathbb K$) with the Diagram \ref{gcFullStructure}.

Therefore, the rank computed on $\mathbb K$ of a tuple of elements of Diagram  \ref{GCReductMultiplicative} is the same as the rank of the correspondent tuple on Diagram \ref{gcFullStructureMultiplicative}.  As $\mathcal M$ is a reduct of $\mathbb K$ the rank of tuples computed on $\mathcal M$ is bigger or equal than the rank of tuples computed on $\mathbb K$. So for proving our claim, it is enough to prove that the rank computed on $\mathcal M$ does not increase. For doing so we have to prove that there are still algebraic relations on the tuples lying on the same line of the diagram. 

Let us prove that
$$\eta\in \acl_{\mathcal G}(\alpha_0,\alpha_1,\kappa)$$

Let  
$$Z(\alpha_0,\alpha_1,\kappa):=(X_\tau \circ X_{\alpha_0})\cdot (X_{\alpha_1}\circ X_\kappa)\menosmult(X_\tau \circ X_{\alpha_1}).$$

If $$X_\tau\circ X_\eta\cap Z(\alpha_0,\alpha_1,\kappa)$$ is infinite we proceed in the same fashion as in the proof of Theorem \ref{thmAditiveVersion}. So let us assume that it is not the case, then the algebraic relation follows from the first bullet on the conclusion of Lemma \ref{lemmaIntersectionSuberVersionMultiplicative} applied to $(a,r_0,r_1,r_2)=(\eta,\alpha_0,\alpha_1,\kappa)$, by the conclusion there are infinitely many $s\in Y$ such that:

$$| X_\tau\circ X_\eta\cap Z(\alpha_0,\alpha_1,\kappa)|<| X_\tau\circ X_s\cap Z(\alpha_0,\alpha_1,\kappa)|.$$

As $Z(\alpha_0,\alpha_1,\kappa)$ is a $\mathcal M$-definable set with parameters $\alpha_0,\alpha_1,\kappa$, the set

$$E:=\{s\in Y: | X_\tau\circ X_\eta\cap Z(\alpha_0,\alpha_1,\kappa)|<| X_\tau\circ X_s\cap Z(\alpha_0,\alpha_1,\kappa)|\}$$ is also $\mathcal M$-definable with parameters $\alpha_0,\alpha_1,\kappa$. As $E$ is infinite and $Y$ is strongly minimal, $Y\setminus E$  is finite and as $\eta\in Y\setminus E$ we conclude that 
$$\eta\in \acl_{\mathcal M}(\alpha_0,\alpha_1,\kappa).$$

%If $$X_\tau\circ X_\eta\cap Z(\alpha_0,\alpha_1,\kappa)$$ is infinite, in particular $\cl(X_\tau\circ X_\eta)\cap Z(\alpha_0,\alpha_1,\kappa)$ is infinite, by Claim \ref{genericFiniteIntersectionClaimMultiplicative} the family $(\cl(X_\tau\circ X_a))_{a\in Y}$ has generic finite intersection, so Lemma \ref{noGenericFiniteZLemma} implies that $\eta$ is not $\mathbb K$-generic over $(\alpha_0,\alpha_1,\kappa)$, in particular 
%$$\left\{b\in Y:X_\tau \circ X_b\cap Z(\alpha_0,\alpha_1,\kappa)\text{ is infinite}\right\}$$ does not contain any $\mathbb K$-generic element so it is finite. Since it is $\mathcal M$-definable and contains $\eta$ we conclude $$\eta\in \acl_{\mathcal M}(\alpha_0,\alpha_1,\kappa).$$

 The rest of algebraic relations needed for the definition of group configuration are provided by either the first or the second bullet in the conclusion of Lemma \ref{lemmaIntersectionSubeVersion}. 
\end{claimproof}

 Then we use the group configuration provided by Claim \ref{claimGCVersionMultiplicative}, apply Fact \ref{fieldConfig} and conclude that there is an infinite field interpreted by $\mathcal M$. This finishes the proof of Proposition \ref{propMultiplicativeUnderAssumption}.
\end{proof}

\section{Proof of Theorem \ref{thmMultiplicativeVersion}}

Now we present the proof of Theorem \ref{thmMultiplicativeVersion}:

\begin{proof}(Proof of Theorem \ref{thmMultiplicativeVersion})

Let us assume that that $\mathbb K$ is $\lambda$-saturated for some big enough cardinal $\lambda$.

Let $\mathcal M=(M,\cdot,\ldots)$ as in the statement of Theorem \ref{thmMultiplicativeVersion}. Fix $X\subseteq M\times M$ strongly minimal and not $\mathcal M$-affine. By Lemma \ref{lemmaYStar} there is $Y\subseteq M\times M$ a $\mathcal M$-definable with parameters $\bar d$ that is strongly minimal and not $\mathcal M$-affine and there is also a finite number of polynomials $L_i$ and $L_{i,n_i}$ such that $Y$ has the Property $(\star)$ witnessed by $L_i$ and $L_{i,n_i}$. 

If $Y$ has infinitely many slopes at $(1,1)$ witnessed by $L_i$ and $L_{i,n_i}$ we apply Proposition \ref{propMultiplicativeUnderAssumption} and conclude. So we may assume that this is not the case, therefore we may apply Proposition \ref{existsGroupConfigurationMult}, let $a_0,b_0,b\in \mathbb G_a$ and $\mathfrak g=(\alpha,\beta,\gamma,\kappa, \eta,\mu)\in Y^6$ be as provided by the conclusion. 

Apply Theorem \ref{gcIsLocallyAdditiveThm} and find $G$ a $\mathcal M$-interpretable group that is locally isomorphic to a subgroup of $\mathbb G_a$. As $(G,\oplus)$ is $\mathcal M$-definable, if we denote by $\mathcal G=(G,\oplus, \ldots)$ the first order structure induced by $\mathcal M$ on $G$, then $\mathcal H$ is strongly minimal and non locally modular so satisfies the hypothesis of Theorem \ref{thmAditiveVersion}, so there is an infinite field interpretable in $\mathcal G$ and as $\mathcal G$ is interpretable in $\mathcal M$, a copy of the same infinite field is interpretable in $\mathcal M$.
\end{proof}

\newpage
\bibliographystyle{alpha}
\bibliography{biblio}
%\nocite{*}
\end{document}